\documentclass[12pt]{amsart}
\usepackage[utf8]{inputenc}
\usepackage{geometry,verbatim,upgreek}
\usepackage[shadow]{todonotes}
\usepackage{amssymb,latexsym,amsfonts,mathtools}
\usepackage{mathrsfs} 
\usepackage{enumitem} 
\usepackage{microtype} 
\usepackage[shadow]{todonotes}
\usepackage{tikz}
\usetikzlibrary{arrows,backgrounds,patterns,shapes,calc,decorations.pathreplacing}
\newenvironment{enumerate-(a)}{\begin{enumerate}[label={\upshape (\alph*)}, leftmargin=2pc]}{\end{enumerate}}
\newenvironment{enumerate-(i)}{\begin{enumerate}[label={\upshape (\roman*)}, leftmargin=2pc]}{\end{enumerate}}
\theoremstyle{plain}
\newtheorem{theorem}{Theorem}[section]

\newtheorem{proposition}[theorem]{Proposition}
\newtheorem{lemma}[theorem]{Lemma}
\newtheorem{corollary}[theorem]{Corollary}
\newtheorem{claim}{Claim}[theorem]
\theoremstyle{definition}
\newtheorem{definition}[theorem]{Definition}
\theoremstyle{remark}
\newtheorem{remark}[theorem]{Remark}
\newtheorem*{remark*}{Remark}
\newtheorem*{remarks*}{Remarks}
\newtheorem{remarks}[theorem]{Remarks}
\newtheorem{example}[theorem]{Example}
\newtheorem{examples}[theorem]{Examples}

\newcommand{\AND}{\mathbin{\, \wedge \,}}
\newcommand{\OR}{\mathbin{\, \vee \,}}
\renewcommand{\implies}{\Rightarrow}	
\newcommand{\IMPLIES}{\mathbin{\, \Rightarrow \,}}
\renewcommand{\iff}{\mathbin{\Leftrightarrow }}
\newcommand{\IFF}{\mathbin{\, \Leftrightarrow \,}}
\newcommand{\FORALL}[1]{\forall {#1} \, }
\newcommand{\FORALLS}[2]{\forall^{#1} {#2} \, }
\newcommand{\EXISTS}[1]{\exists {#1} \, }

\newcommand{\EXISTSS}[2]{\exists^{#1} {#2} \, }
\newcommand{\R}{\mathbb{R}}	
\newcommand{\Q}{\mathbb{Q}}	
\newcommand{\Mid}{\boldsymbol\mid}		
\newcommand{\equalsdef}{\stackrel{\text{\tiny\rm def}}{=}}	
\newcommand{\setof}[2]{\mathopen \{{#1}\Mid{#2} \mathclose\}} 
\newcommand{\setofLR}[2]{\left \{{#1} \Mid {#2} \right\}} 
\newcommand{\set}[1]{\mathopen \{ {#1} \mathclose \}} 
\newcommand{\setLR}[1]{\left \{ {#1} \right \}} 
\newcommand{\Pow}{\mathscr{P}}		
\newcommand{\seqof}[2]{\mathopen \langle #1 \Mid #2 \mathclose \rangle} 
\newcommand{\seqofLR}[2]{\left \langle #1 \Mid #2 \right \rangle} 
\newcommand{\seq}[1]{\mathopen \langle #1 \mathclose \rangle}	
\newcommand{\eq}[1]{{\boldsymbol [}{#1} {\boldsymbol ]}}	
\newcommand{\symdif}{\mathop{\triangle}}	
\newcommand{\LOC}[2]{{#1}_{\lfloor {#2}\rfloor}} 
\DeclareMathOperator{\PrTr}{PrTr} 
\newcommand{\pre}[2]{\prescript{#1}{}{#2}}				
\newcommand{\Pre}[2]{\prescript{#1}{}{#2}}				
\DeclareMathOperator{\dom}{dom} 		
\DeclareMathOperator{\ran}{ran} 		
\newcommand{\conc}{{}^\smallfrown}		
\DeclareMathOperator{\lh}{lh}				
\newcommand{\onto}{\twoheadrightarrow}
\DeclareMathOperator{\Int}{Int} 	
\DeclareMathOperator{\Cl}{Cl} 		
\DeclareMathOperator{\Fr}{Fr} 		
\newcommand{\Nbhd}{{\boldsymbol N}\!} 
\DeclareMathOperator{\Ball}{B} 
\newcommand{\ocinterval}[2]{\left ( {#1} ; {#2} \right ]}
\newcommand{\cointerval}[2]{\left [ {#1} ; {#2} \right)} 
\newcommand{\bSigma}{\boldsymbol{\Sigma}}
\newcommand{\bPi}{\boldsymbol{\Pi}}

\newcommand{\bDelta}{\boldsymbol{\Delta}}
\DeclareMathOperator{\KK}{\mathbf{K}}	
\DeclareMathOperator{\Gdelta}{\mathbf{G}_\delta}	
\DeclareMathOperator{\Fsigma}{\mathbf{F}_\sigma}	
\DeclareMathOperator{\Ksigma}{\mathbf{K}_\sigma}	
\DeclareMathOperator{\Fsigmadelta}{\mathbf{F}_{\sigma\delta}}

\DeclareMathOperator{\Gdeltasigma}{\mathbf{G}_{\delta\sigma}} 
\DeclareMathOperator{\NULL}{\textrm{\scshape Null}}	
\DeclareMathOperator{\Bor}{\textrm{\scshape Bor}}	
\DeclareMathOperator{\MEAS}{\textrm{\scshape Meas}}	
\DeclareMathOperator{\MALG}{\textrm{\scshape Malg}}	
\DeclareMathOperator{\Solid}{\textrm{\scshape Sld}}
\DeclareMathOperator{\Dual}{\textrm{\scshape Dl}}
\DeclareMathOperator{\qDual}{\textrm{\scshape qDl}}
\DeclareMathOperator{\Spongy}{\textrm{\scshape Spng}}
\newcommand{\cardLR}[1]{\left | #1 \right |}		
\newcommand{\card}[1]{\mathopen | #1 \mathclose |}		
\newcommand{\leqW}{\leq_{\mathrm{W}}}
\DeclareMathOperator{\supt}{supt}
\newcommand{\body}[1]{\left [ {#1} \right ]} 
\newcommand{\markdef}[1]{\textbf{#1}}
\newcommand{\densitytree}{\boldsymbol{D}}
\newcommand{\absval}[1]{\left \| {#1} \right \|}
\newcommand{\measurecantor}{\mu^{\mathrm{C}}}
\newcommand{\measurebaire}{\mu^{\mathrm{B}}}

\newcommand{\density}{\mathscr{D}} 
\newcommand{\oscillation}{\mathscr{O}} 
\DeclareMathOperator{\diam}{diam}	
\DeclareMathOperator{\Blur}{Blr}
\DeclareMathOperator{\Sharp}{Shrp}
\DeclareMathOperator{\Exc}{Exc}
\DeclareMathOperator{\exc}{exc}
\renewcommand{\restriction}{\mathpunct{\upharpoonright}}
\DeclareMathOperator{\Lev}{Lev}
\usepackage[naturalnames]{hyperref}

\def\version{0}
\ifthenelse{0 = \version}{\newcommand{\nota}[1]{\todo{#1}}}{%
	\ifthenelse{1 = \version}{\newcommand{\nota}[1]{\null}} }

\usepackage[style=alphabetic,sorting=nyt,backend=bibtex]{biblatex}
\title{Lebesgue density and exceptional points}
\author{Alessandro Andretta}
\address{Dipartimento di matematica, Università di Torino, via Carlo Alberto 10, 10123 Torino---Italy}
\email{alessandro.andretta@unito.it}
\author{Riccardo Camerlo}
\address{Dipartimento di scienze matematiche, Politecnico di Torino, Corso Duca degli Abruzzi 24, 10129 Torino---Italy}
\address{Département des systèmes d'information, Université de Lausanne, Quartier UNIL-Dorigny  Bâtiment Internef,  1015 Lausanne}
\email{camerlo@calvino.polito.it}
\author{Camillo Costantini}
\address{Dipartimento di Matematica, Università di Torino, via Carlo Alberto 10, 10123 Torino---Italy}
\email{camillo.costantini@unito.it}
\date{\today} 

\begin{document}
\begin{abstract}
Work in the measure algebra of the Lebesgue measure on \( \pre{\omega}{2} \): for comeager many \( \eq{A} \) the set of points \( x \) such that the density of \( x \) at \( A \) is not defined is \( \bSigma^{0}_{3} \)-complete; for some compact \( K \) the set of points \( x \) such that the density of \( x \) at \( K \) exists and it is different from \( 0 \) or \( 1 \) is \( \bPi^{0}_{3} \)-complete; the set of all \( \eq{K} \) with \( K \) compact is \( \bPi^{0}_{3} \)-complete. 
There is a set (which can be taken to be open or closed) in \( \R^n \) such that the density of any point is either \( 0 \) or \( 1 \), or else undefined.
Conversely, if a subset of \( \R^n \) is such that the density exists at every point, then the value \( 1/2 \) is always attained.
On the route to this result we show that Cantor space can be embedded in a measured Polish space in a measure-preserving fashion.
\end{abstract}
\maketitle

\section{Statement of the main results}
In this paper we study from the point of view (and with the methods) of descriptive set theory, some questions stemming from real analysis and measure theory.
In order to state our results we recall a few definitions.
The density of a measurable set \( A \) at a point \( x \in X \) is the limit \( \density_A ( x ) = \lim_{ \varepsilon {\downarrow} 0 } \mu ( A \cap \Ball ( x ; \varepsilon ) ) / \mu ( \Ball ( x ; \varepsilon ) ) \), where \( \mu \) is a Borel measure on the metric space \( X \) and \( \Ball ( x ; \varepsilon ) \) is the open ball centered at \( x \) of radius \( \varepsilon \).
Let \( \Sharp ( A ) \) be the collection of all points \( x \) where \( 0 < \density_A ( x ) < 1 \), and let \( \Blur ( A ) \) be the collection of all points \( x \) where the limit \( \density_A ( x ) \) does not exist.
The Lebesgue density theorem says that \( A \symdif \setofLR{x \in X}{ \density_A ( x ) = 1 } \) is null, and hence \( \Blur ( A ) \cup \Sharp ( A ) \) is null, when \( ( X , d , \mu ) \) is \emph{e.g.} the Euclidean space \( \R^n \) with the usual distance and the Lebesgue measure, or the Cantor space \( \pre{\omega }{2} \) with the usual ultrametric and the coin-tossing measure.
If \( \Blur ( A ) = \emptyset \), i.e. \( \density_A ( x ) \) exists for any \( x \), then \( A \) is said to be solid; at the other extreme of the spectrum there are the spongy sets, that is sets \( A \) such that there are no points of intermediate density and there are points \( x \) where \( \density_A ( x ) \) does not exist, i.e., \( \Sharp ( A ) = \emptyset \) and \( \Blur ( A ) \neq \emptyset \).
(Examples of solid sets are the balls in \( \R^n \) and the clopen sets in the Cantor space; it is not hard to construct a spongy set in the Cantor space, but the case of \( \R^n \) is another story.)
All these notions are invariant under perturbations by a null set, so they can be defined on the measure algebra \( \MALG ( X , \mu ) \).

We prove a few results on these matters.
Theorem~\ref{thm:KisPi03complete} shows that for a large class of spaces \( ( X , d , \mu ) \), the set \( \mathscr{K} \) of all \( \eq{K} \in \MALG \) with \( K \) compact is in \( \Fsigmadelta \setminus \Gdeltasigma \), i.e.~it is \( \bPi^{0}_{3} \)-complete, in the logicians' parlance.
The result still holds for \( \mathscr{F} \) the set of all \( \eq{F} \in \MALG \) with \( F \) closed.
The result is first proved for the Cantor space \( \pre{\omega}{2} \) with the usual coin-tossing measure, and then extended to the general case by means of a construction enabling us to embed the Cantor space into \( ( X , \mu ) \) in a measure preserving way (Theorem~\ref{thm:embeddingCantorinPolish}).
Restricting ourselves to the Cantor space, we show that for comeager many \( \eq{A} \in \MALG \) the set \( \Blur ( A ) \) is \( \Gdeltasigma \setminus \Fsigmadelta \), i.e.~\( \bSigma^{0}_{3} \)-complete (Theorem~\ref{thm:blurrypointsSigma03}), and that \( \Sharp ( K ) \) is \( \bPi^{0}_{3} \)-complete, for some compact set \( K \) (Theorem~\ref{thm:sharppointsPi03}).
Finally we address the issue of solid and spongy sets in Euclidean spaces: we show that if \( A \) is solid, then it has density \( 1 / 2 \) at some point (Corollary~\ref{cor:nodualisticsetsinRn}), and that spongy sets exist (Theorem~\ref{thm:spongy}).

The paper is organized as follows.
Section~\ref{sec:notationsandpreliminaries} collects some standard facts and notations used throughout the paper, while Section~\ref{sec:densityfunction} summarizes the basic results on the density function and the Lebesgue density theorem; these two section can be skipped on first read.
Section~\ref{sec:Cantorsets} is devoted to the problem of embedding the Cantor space in a Polish space, while a characterization of compact sets in the measure algebra is given in Section~\ref{sec:compactsetsinMALG}.
Section~\ref{sec:exceptionalpoints} is devoted to the study of \( \Blur ( A ) \) and \( \Sharp ( A ) \), while the study of solid sets in \( \R^n \) and the construction of spongy subset of \( \R^n \) is carried out in Section~\ref{sec:solid&spongy}.

\section{Notation and preliminaries}\label{sec:notationsandpreliminaries}
The notation of this paper is standard and follows closely that of~\cite[][]{Kechris:1995kc,Andretta:2013uq}, but for the reader's convenience we summarize it below.

\subsection{Polish spaces}
In a topological space \( X \), the closure, the interior, the frontier, and the complement of \( Y \subseteq X \) are denoted by \( \Cl Y \), \( \Int Y \), \( \Fr Y \), and \( Y^\complement \).
A topological space is Polish if it is separable and completely metrizable.
In a metric space \( ( X , d ) \), the open ball of center \( x \) and radius \( r \geq 0 \) is \( \Ball ( x ; r ) \), with the understanding that \( \Ball ( x ; 0 ) = \emptyset \). 
The collection \( \Bor ( X ) \) of all Borel subsets of \( X \) is stratified in the Borel hierarchy \( \bSigma^{0}_{ \alpha } ( X ) \), \( \bPi^{0}_{ \alpha } ( X ) \), \( \bDelta^{0}_{ \alpha } ( X ) \), with \( 1 \leq \alpha < \omega _1 \).
Namely: \( \bSigma^{0}_{1} \) is the collection of open sets, \( \bSigma^{0}_{ \alpha } \) is the collection of sets \( \bigcup_{n} A_n \) with \( A_n \in \bPi^{0}_{ \beta _n} \) and \( \beta _n < \alpha \), and \( \bPi^{0}_{ \alpha } = \setofLR{ A^\complement }{ A \in \bSigma^{0}_{ \alpha } } \).
We also set \( \bDelta^{0}_{ \alpha } = \bSigma^{0}_{ \alpha } \cap \bPi^{0}_{ \alpha } \).
Thus \( \bDelta^{0}_{1} \) are the clopen sets, \( \bPi^{0}_{1} \) are the closed sets, \( \bSigma^{0}_{2} \) are the \( \Fsigma \) sets, \( \bPi^{0}_{2} \) are the \( \Gdelta \) sets, \( \bPi^{0}_{3} \) are the \( \Fsigmadelta \) sets, and so on.
The collections of all compact and of all \( \sigma \)-compact subsets of \( X \) are denoted by \( \KK ( X ) \) and \( \KK_ \sigma ( X ) \), respectively.
If \( X \) is Polish, then \( \KK ( X ) \) endowed with the Vietoris topology is Polish.

A function \( f \colon X \to Y \) between Polish spaces is of \markdef{Baire class \( \xi \)} if the preimage of any open \( U \subseteq Y \) is in \( \bSigma^{0}_{1 + \xi } \).
The collection of all Baire class \( \xi \) functions from \( X \) to \( Y \) is denoted by \( \mathscr{B}_ \xi ( X , Y ) \) or simply by \( \mathscr{B}_ \xi \) when \( X \) and \( Y \) are clear from the context.

A \markdef{measurable space} \( ( X , \mathcal{S} ) \) consists of a \( \sigma \)-algebra \( \mathcal{S} \) on a nonempty set \( X \).
A measurable space \( ( X , \mathcal{S} ) \) is \markdef{standard Borel} if \( \mathcal{S} \) is the \( \sigma \)-algebra of the Borel subsets of \( X \), for some suitable Polish topology on \( X \).

\subsection{Sequences and trees}\label{subsec:sequences&trees}
\subsubsection{Sequences}
The set of all functions from \( J \) to \( I \) is denoted by \( \Pre{ J }{ I } \).
The set \( \pre{ < \omega }{I} = \bigcup_{n} \pre{ n}{ I } \) is the set of all finite sequences from \( I \), and \( \pre{ \leq \omega }{ I } = \pre{ < \omega }{ I } \cup \pre{ \omega }{ I } \).
The \markdef{length} of \( x \in \pre{ \leq \omega }{I} \) is the ordinal \( \lh ( x ) = \dom ( x ) \).
The \markdef{concatenation of \( s \in \pre{ < \omega }{I} \) with \( x \in \pre{ \leq \omega }{I} \)} is \( s \conc x \in \pre{ \leq \omega }{ I } \) defined by \( s \conc x ( n ) = s ( n ) \) if \( n < \lh ( s ) \), and \( s \conc x ( n ) = x ( i ) \) if \( n = i + \lh ( s ) \).
We often blur the difference between the sequence \( \seq{ i } \) of length \( 1 \) with its unique element \( i \) and write \( t \conc i \) instead of \( t \conc \seq{ i } \).
The sequence of length \( N \leq \omega \) that attains only the value \( i \) is denoted by \( i^{ ( N ) } \).

\subsubsection{Trees}
A \markdef{tree} on a nonempty set \( I \) is a \( T \subseteq \pre{ < \omega }{I} \) closed under initial segments; the \markdef{body} of \( T \) is \( \body{T} = \setofLR{ b \in \pre{ \omega }{I} }{ \FORALL{n \in \omega } ( b \restriction n \in T ) } \).
A tree \( T \) on \( I \) is \markdef{pruned} if \( \FORALL{t \in T} \EXISTS{s \in T} ( t \subset s ) \).
The set \( \body{T} \) is a topological space with the topology generated by the sets 
\[ 
\Nbhd_t ^{\body{T}}= \Nbhd_t = \setofLR{x \in \body{T} }{ x \supseteq t }
\] 
with \( t \in T \).
This topology is induced by the metric \( d_T ( x , y ) = 2^{-n} \) where \( n \) is least such that \( x ( n ) \neq y ( n ) \).
This is actually a complete metric, and an ultrametric, i.e the triangular inequality holds in the stronger form \( d ( x , z ) \leq \max \setLR{ d ( x , y ) , d ( y , z ) } \).
Therefore \( \body{T} \) is zero-dimensional, i.e. it has a basis of clopen sets.
A nonempty closed subset of \( \body{T} \) is of the form \( \body{S} \) with \( S \) a pruned subtree of \( T \).
If \( T \) is a tree on a countable set \( I \), then \( \body{T} \) is separable, and therefore it is a Polish space.

The \markdef{localization} of \( X \subseteq \pre{ \leq \omega }{I} \) at \( s \in \pre{ < \omega }{I} \) is
\[
\LOC{X}{s} = \setofLR{ t \in \pre{ \leq \omega }{I } }{ s \conc t \in X } . 
\]
Thus if \( A \subseteq \pre{\omega}{I} \) then \( s \conc \LOC{A}{s} = A \cap \Nbhd^{\mathcal{X}}_s \), where \( \mathcal{X} = \body{ \pre{ < \omega }{I} } \).
Note that if \( T \) is a tree on \( I \) and \( t \in T \), then \( \body{\LOC{T}{t}} = \LOC{ \body{T}}{t} \).

A function \( \varphi \colon S \to T \) between pruned trees is
\begin{itemize}
\item
\markdef{monotone} if \( s_1 \subseteq s_2 \implies \varphi ( s_1 ) \subseteq \varphi ( s_2 ) \),
\item
\markdef{Lipschitz} if it is monotone and \( \lh s \leq \lh \varphi ( s ) \),
\item
\markdef{continuous} if it is monotone and \( \lim_n \lh \varphi ( x \restriction n ) = \infty \) for all \( x \in \body{S} \).
\end{itemize}
If \( \varphi \) is Lipschitz then it is continuous, and a continuous \( \varphi \) induces a continuous function
\[ 
f _ \varphi \colon \body{S} \to \body{T} , \quad f _ \varphi ( x ) = \bigcup_{n} \varphi ( x \restriction n ) ,
\] 
and every continuous function \( \body{S} \to \body{T} \) arises this way.
If \( \varphi \) is Lipschitz, then \( f_ \varphi \) is Lipschitz with constant \( \leq 1 \), that is \( d_T ( f ( x ) , f ( y ) ) \leq d_S ( x , y ) \), and every such function arises this way.
These definitions can be extended to similar situations.
For example, letting \( \pre{ < \omega \times \omega }{ I } = \bigcup_{n} \pre{ n \times n }{ I } \), we say that \( \varphi \colon \pre{ < \omega \times \omega }{I} \to T \) is Lipschitz if 
\[
 \FORALL{n} \FORALL{m < n} \FORALL{a \in \pre{ n \times n }{I} } \left ( \varphi ( a \restriction m \times m ) \subset \varphi ( a ) \right ) . 
\]
Such \( \varphi \) defines a continuous map from the space \( \pre{ \omega \times \omega }{I} \) (which is homeomorphic to \( \pre{ \omega }{I} \)) to \( \body{T} \).

\subsection{The Cantor and Baire spaces}
The \markdef{Cantor space} \( \pre{\omega }{2} \) is the body of the complete binary tree \( \pre{ < \omega }{2} \). 
A subset of a separable metric space is a \markdef{Cantor set} if it is nonempty, compact, zero-dimensional, and perfect (i.e. without isolated points).
By a theorem of Brouwer's~\cite[][Theorem 7.4]{Kechris:1995kc} every Cantor set is homeomorphic to \( \pre{\omega }{2} \), whence the name.
The typical example of such set is \( E_{1/3} \), the closed, nowhere dense, null subset of \( [ 0 ; 1 ] \) usually known as \emph{Cantor's middle-third set}.
See Section~\ref{sec:Cantorsets} for more examples of Cantor sets.

The \markdef{Baire space} \( \pre{\omega}{\omega} \) is the body of \( \pre{< \omega}{\omega} \).
If \( T \) is pruned, then \( \body{T} \) is compact iff \( T \) is finitely branching, and therefore every compact subset of \( \pre{\omega}{\omega} \) has empty interior.
The Baire set is homeomorphic to \( [ 0 ; 1 ] \setminus \mathbb{D} \), where \( \mathbb{D} = \setofLR{ k \cdot 2^{-n}}{ 0 \leq k \leq 2^n \wedge n \in \omega } \) is the set of dyadic numbers, via the map 
\begin{equation}\label{eq:homeomorphismBaire}
G \colon \pre{\omega}{\omega} \to [ 0 ; 1 ] \setminus \mathbb{D} , \qquad \setLR{G ( x ) } = \bigcap_{n} I ( x \restriction n )
\end{equation}
where the \( I ( s ) \) (for \( s \in \pre{< \omega}{\omega} \)) are the closed intervals with endpoints in \( \mathbb{D} \) defined as follows: \( I ( \emptyset ) = [ 0 ; 1 ] \), and if \( I ( s ) = [ a ; b ] \), then 
\[ 
I ( s \conc k ) = 
\begin{cases}
[ b - ( b - a ) 2^{- k} ; b - ( b - a ) 2^{- k - 1} ] & \text{if \( \lh s \) is odd,}
\\
[ a + ( b - a ) 2^{- k - 1} ; a + ( b - a ) 2^{- k } ] & \text{if \( \lh s \) is even,}
\end{cases}
\]
see~\cite[][Chapter VII, \S 3]{Levy:2002pt}.
By Cantor's theorem \( \mathbb{D} \setminus \set{ 0 , 1 } \) is order isomorphic to any countable dense set \( D \subseteq \R \), and hence there is a homeomorphism \( ( 0 ; 1 ) \to \R \) that maps \( ( 0 ; 1 ) \setminus \mathbb{D} \) onto \( \R \setminus D \). 
In other words, \( \pre{\omega}{\omega} \) is homeomorphic to \( \R \setminus D \) where \( D \) is countable dense set; in particular, it is homeomorphic to the set of irrational numbers.

\subsection{Measures}
A \markdef{measure space} \( ( X , \mathcal{S} , \mu ) \) consists of a \( \sigma \)-algebra \( \mathcal{S} \) on a nonempty set \( X \) and a \( \sigma \)-additive measure \( \mu \) with domain \( \mathcal{S} \).
We always assume that \( \mu \) is \markdef{nonzero}, that is \( \mu ( X ) > 0 \).
Given a measure space \( ( X , \mathcal{S} , \mu ) \) we say that \( \mu \) is \markdef{non-singular}%
\footnote{In the literature these measures are also called non-atomic or continuous, but in this paper the adjective \emph{continuous} is reserved for a different property (Definition~\ref{def:continuousmeasure}).} 
or \markdef{diffuse} if \( \mu ( \set{x} ) = 0 \) for all \( x \in X \), it is a \markdef{probability measure} if \( \mu ( X ) = 1 \), it is \markdef{finite} if \( \mu ( X ) < \infty \), it is \markdef{\( \sigma \)-finite} if \( X = \bigcup_{n} X_n \) with \( X_n \in \mathcal{S} \) and \( \mu ( X_n ) < \infty \). 
Following Carathéodory, \( \mathcal{S} \) can be extended to \( \MEAS_\mu \), the \( \sigma \)-algebra of \markdef{\( \mu \)-measurable sets}, and the measure can be uniquely extended to a measure (still denoted by \( \mu \)) on \( \MEAS_\mu \).
A set \( N \in \MEAS_\mu \) is \markdef{null} if \( \mu ( N ) = 0 \), that is if there is \( A \in \mathcal{S} \) such that \( N \subseteq A \) and \( \mu ( A ) = 0 \).
A set \( A \in \MEAS_\mu \) is \markdef{nontrivial} if \( A , A^\complement \notin \NULL_\mu \).

For \( A , B \in \MEAS_\mu \) set \( A \subseteq_\mu B \iff \mu ( A \setminus B ) = 0 \), and 
\[
 A =_\mu B \iff A \subseteq_\mu B \wedge B \subseteq_\mu A \iff \mu ( A \symdif B ) = 0 .
 \]
Taking the quotient of \( \MEAS_\mu \) by the ideal \( \NULL_\mu \) or equivalently by the equivalence relation \( =_\mu \), we obtain the \markdef{measure algebra} of \( \mu \)
\[
\MALG ( X , \mu ) = \frac{\MEAS_\mu }{ \NULL_\mu} = \frac{ \mathcal{S}}{ \mathcal{S} \cap \NULL_\mu} ,
\]
which is a boolean algebra.
(Whenever possible we will drop the mention to \( X \) and/or \( \mu \) in the definition of measure algebra.)
The measure \( \mu \) induces a function on the quotient
\[
\hat{ \mu } \colon \MALG \to [ 0 ; + \infty ] , \quad \hat{ \mu } ( \eq{A} ) = \mu ( A ) .
\]
We often write \( \mu ( \eq{A} ) \) or \( \mu \eq{A} \) instead of \( \hat{ \mu } ( \eq{A} ) \).
The set \( \MALG_\mu \) is endowed with the topology generated by the sets \( \mathcal{B} _{ \eq{A} , r } = \setof{ \eq{B} }{ \mu ( A \symdif B ) < r } \) for \( \eq{A} \in \MALG \) and \( r > 0 \).
When \( \mu \) is finite, this topology is metrizable with the distance \( \delta ( \eq{A} , \eq{B} ) = \mu ( A \symdif B ) \), and therefore \( \mathcal{B} _{ \eq{A} , r } = \Ball ( \eq{A} ; r ) \).

A \markdef{Borel measure} on a topological space \( X \) is a measure \( \mu \) defined on \( \Bor ( X ) \), the collection of all Borel subsets of \( X \); we say that \( \mu \) is \markdef{fully supported} if \( \mu ( U ) > 0 \) for all nonempty open set \( U \).
A Borel measure is \markdef{inner regular} if \( \mu ( A ) = \sup \setof{ \mu ( F )}{ F \subseteq A \wedge F \text{ is closed}} \); it is \markdef{outer regular} if \( \mu ( A ) = \inf \setof{\mu ( U )}{U \supseteq A \wedge U \text{ is open}} \).
A finite Borel measure on a metric space is both inner and outer regular.
A Borel measure is \markdef{locally finite} if every point has a neighborhood of finite measure; hence in a second countable space a locally finite measure is automatically \( \sigma \)-finite.
A \markdef{Radon space} \( ( X , \mu ) \) is a Hausdorff topological space \( X \) with a locally finite Borel measure which is \markdef{tight}, that is \( \mu ( A ) = \sup \setofLR{\mu ( K )}{K \subseteq A \wedge K \in \KK ( X ) } \).
A \markdef{metric measure space} \( ( X , d , \mu ) \) is a metric space endowed with a Borel measure; if the underlying topological space is Polish we will speak of \markdef{Polish measure space}.
Every finite Borel measure on a Polish space is tight. 
In this paper, unless otherwise stated, we \emph{work in a fully supported, locally finite metric measure space}. 
The space \( \MALG_\mu \) is Polish when \( X \) is Polish and \( \mu \) is Borel and finite.
If moreover \( \mu \) is a non-singular, probability measure on \( X \) then \( \MALG_\mu \) is isomorphic to the measure algebra constructed from the Lebesgue measure \( \lambda \) on \( [ 0 ; 1 ] \)~\cite[][Theorem 17.41]{Kechris:1995kc}. 

If \( \mu \) is nonsingular, then \( \lim_{ \varepsilon {\downarrow} 0 } \mu ( \Ball ( x ; \varepsilon ) ) = 0 \), for all \( x \in X \).
The next definition strengthens this fact.

\begin{definition}\label{def:continuousmeasure}
Let \( ( X , d , \mu ) \) be fully supported, locally finite metric masure space.
Then \( \mu \) is
\begin{itemize}[leftmargin=1pc]
\item
\markdef{continuous} if for all \( x \in X \) the map \( \cointerval{0}{+\infty} \to [ 0 ; +\infty ] \), \( r \mapsto \mu ( \Ball ( x ; r ) ) \), is continuous,
\item
\markdef{uniform} if \( \mu ( \Ball ( x ; r ) ) = \mu ( \Ball ( y ; r ) ) \) for all \( x , y \in X \), i.e.~if the measure of an open ball depends only on its radius.
\end{itemize}
\end{definition}

The Lebesgue measure on \( \R^n \) is the typical example of a continuous and uniform measure.
If a measure is continuous, then a much stronger form of continuity holds.

\begin{lemma}\label{lem:continuityofmeasure}
If \( \mu \) is continuous, then the function 
\[ 
B \colon X \times \cointerval{0}{+\infty} \to \MALG , \quad ( x , r ) \mapsto \eq{\Ball ( x , r ) }
\] 
is continuous. 
In particular the map \( X \times \cointerval{0}{+\infty} \to [ 0 ; + \infty ] \), \( ( x , r ) \mapsto \mu ( \Ball ( x , r ) ) \) is continuous. 
\end{lemma}

\begin{proof}
Fix \( ( x , r ) \in X \times \cointerval{0}{+\infty} \), in order to prove continuity of \( B \) in \( ( x , r ) \). 
Fix also \( \varepsilon \in \cointerval{0}{+\infty} \). 
There is \( \delta \in \cointerval{0}{+\infty} \) such that 
\[ 
\forall r' \in \cointerval{0}{+\infty} \left ( \card{ r - r' } < \delta \implies \card{ \mu ( \Ball ( x ; r ) ) - \mu ( \Ball ( x ; r' ) ) } < \varepsilon \right ) .
\]
Let \( ( x' , r' ) \in X \times \cointerval{0}{+\infty} \) with \( d ( x , x' ) < \delta / 4 \) and \( \card{ r - r' } < \delta / 4 \).
If \( r > \frac{\delta }2 \), then
\[
\Ball ( x ; r - \frac { \delta }{2} ) \subseteq \Ball ( x ; r' - \textstyle \frac{ \delta}{4} ) \subseteq \Ball ( x' ; r' ) \subseteq \Ball ( x ; r' + \textstyle \frac{ \delta}{4} ) \subseteq \Ball ( x ; r + \textstyle \frac{ \delta}{2} ) ,
\]
so
\begin{align*}
\mu ( \Ball ( x ; r ) \symdif \Ball ( x' ; r' ) ) 
& = \mu ( \Ball ( x ; r ) \setminus \Ball ( x' ; r' ) ) + \mu ( \Ball ( x' ; r' ) \setminus \Ball ( x ; r ) )
\\
 & \leq \mu ( \Ball ( x ; r ) \setminus \Ball ( x ; r - \textstyle \frac{ \delta}{2} ) ) + \mu ( \Ball ( x ; r + \textstyle \frac{ \delta}{2} ) \setminus \Ball ( x ; r ) )
 \\
 & < 2 \varepsilon. 
\end{align*} 
On the other hand, if \( r \leq \frac{\delta }2 \), then \( \Ball ( x ' ; r ' ) \subseteq \Ball ( x ; r + \frac{ \delta }{2} ) \) as well, so
\begin{equation*}
\begin{split}
\mu ( \Ball ( x ; r ) \symdif \Ball ( x' ; r' ) ) & = \mu ( \Ball ( x ; r ) \setminus \Ball ( x' ; r' ) ) + \mu ( \Ball ( x' ; r' ) \setminus \Ball ( x ; r ) )
\\
 & \leq \mu ( \Ball ( x ; r ) ) + \mu ( \Ball ( x ; r + \textstyle \frac{ \delta}{2} ) \setminus \Ball ( x ; r ) )
 \\
 &< 2 \varepsilon. \qedhere
\end{split}
\end{equation*}
\end{proof}

Using an argument as in Lemma~\ref{lem:continuityofmeasure} one can prove

\begin{lemma}\label{lem:uniformcontinuityofmeasure}
The function \( B \) from Lemma~\ref{lem:continuityofmeasure} is uniformly continuous if
\[
 \FORALL{ \varepsilon > 0 } \EXISTS{ \delta > 0 } \FORALL{x \in X} \FORALL{r , r' \geq 0 } \bigl ( \card{ r - r' } < \delta \implies \card{ \mu ( \Ball ( x ; r ) ) - \mu ( \Ball ( x ; r' ) ) } < \varepsilon \bigr ) .
\]
\end{lemma}

\subsubsection{Measures on the Cantor and Baire spaces}\label{subsubsec:measureonCantor}
A zero-dimensional Polish space can be identified, up to homeomorphism, with a closed subset of \( \pre{\omega}{\omega} \).
Let \( T \) be a pruned tree on \( \omega \); a locally finite Borel measure \( \mu \) on \( \body{T} \subseteq \pre{\omega}{\omega} \) is completely described by its values on the basic open sets \( \Nbhd_s \) with \( s \in T \), so it can be identified with a map 
\[ 
w \colon T \to [ 0 ; M ]
\] 
where \( M = \mu ( \body{T} ) \leq + \infty \), and such that \( w ( \emptyset ) = M \), \( T_\infty = \setof{ t \in T}{ w ( t ) = \infty } \) is a well-founded (possibly empty) tree, and for all \( t \in T \setminus T_ \infty \) 
\[ 
w ( t ) = \sum_{ t \conc i \in T , i \in \omega } w ( t \conc i ) .
\] 
Thus if the measure is finite then \( T_\infty = \emptyset \).
If we require the measure to be fully supported, just replace in the definition above \( [ 0 ; M ] \) with \( \ocinterval{0}{ M } \).
The measure is non-singular just in case
\[
 \lim_{n \to \infty} w ( x \restriction n ) = 0 .
\]
The \markdef{Lebesgue measure} \( \measurecantor \) on \( \pre{\omega }{2} \) is determined by \( w \colon \pre{ < \omega }{2} \to \ocinterval{0}{1} \), \( w ( s ) = 2^{- \lh s } \); it is also known as the \markdef{Bernoulli} or \markdef{coin tossing measure}.
The \markdef{Lebesgue measure} \( \measurebaire \) on \( \pre{\omega}{\omega} \) is determined by \( w \colon \pre{< \omega}{\omega} \to \ocinterval{0}{1} \), \( w ( s ) = \prod_{i < \lh ( s )} 2^{ - s ( i ) - 1} \).
Both \( \measurecantor \) and \( \measurebaire \) are non-singular, and neither is continuous, as the next result shows.
The reasons for tagging \( \measurecantor \) and \( \measurebaire \) with the name ``Lebesgue'' is that they are induced by the Lebesgue measure on \( \R \) via suitable embeddings---for \( \measurebaire \) apply \( G \colon \pre{\omega}{\omega} \to [ 0 ; 1] \) of~\eqref{eq:homeomorphismBaire}, and for \( \measurecantor \) see Example~\ref{xmp:Cntorofmeasure2}.

\begin{proposition}\label{prop:measuresonBairenotcontinuous}
Let \( T \) be a pruned tree on \( \omega \), and let \( \mu \) be a locally finite, non-singular, fully supported Borel measure on \( \body{T} \).
Then for each \( x \in \body{T} \) the set of discontinuity points of \( r \mapsto \mu ( \Ball ( x ; r ) ) \) accumulates to \( 0 \).
\end{proposition}

\begin{proof}
Let \( w \colon T \to [ 0 , \infty ] \) be the map inducing \( \mu \).
As \( \mu \) is fully supported and non-singular, then \( \body{T} \) has no isolated points and \( \FORALL{ s \in T } \EXISTS{ t \in T} ( s \subset t \wedge w ( s ) > w ( t ) ) \).
Thus for each \( x \in \body{T} \) and each \( n \) such that \( w ( x \restriction n ) < + \infty \) and \( x \restriction n \) has more than one immediate successor in \( T \), 
\[ 
w ( x \restriction n ) = \lim_ {\varepsilon \downarrow 2^{-n} } \mu ( \Ball ( x ; \varepsilon ) ) > \mu ( \Ball ( x ; 2^{-n} ) ) = w ( x \restriction n + 1 ) . \qedhere
\]
\end{proof}

In particular, Proposition~\ref{prop:measuresonBairenotcontinuous} applies to \( \measurecantor \) and \( \measurebaire \).

\section{Cantor sets}\label{sec:Cantorsets}
\subsection{Cantor-schemes}\label{subsec:Cantorschemes}
A \markdef{Cantor-scheme} in a metric space \( ( X , d ) \) is a system \( \seqofLR{U_s }{ s \in \pre{ < \omega }{2} } \) of nonempty open subsets of \( X \) such that 
\begin{itemize}
\item 
\( \Cl ( U_{s \conc i} ) \subseteq U_s \), for all \( s \in \pre{ < \omega }{2} \) and \( i \in \set{ 0 , 1} \),
\item
\( \Cl ( U_{s \conc 0 } ) \cap \Cl ( U_{s \conc 1 } ) = \emptyset \).
\end{itemize}
If it also satisfies
\begin{itemize}
\item
\( \lim_{n \to \infty} \diam ( U_{ z \restriction n } ) = 0 \), for all \( z \in \pre{\omega }{2} \),
\end{itemize}
we say that it has \markdef{shrinking diameter}.
A Cantor-scheme of shrinking diameter in a complete metric space yields a continuous injective \( F \colon \pre{\omega }{2} \to X \)
\begin{equation}\label{eq:homeoCantorscheme}
F ( z ) = \text{the unique point in } \bigcap_{n} \Cl ( U_{ z \restriction n} ) . 
\end{equation}
Thus \( \ran F \) is a Cantor subset of \( X \).
Conversely, if \( F \colon \pre{\omega }{2} \to K \subseteq X \) witnesses that \( K \) is a Cantor set, then there is a Cantor-scheme of shrinking diameter that yields \( K \): let \( U_\emptyset = X \), and for each \( s \in \pre{ < \omega }{2} \) let \( K_s = F ( \Nbhd_s ) \) and let \( U_{s \conc i} = \Ball ( K_{s \conc i} ; r_s / 3) \) where \( r_s = d ( K_{ s \conc 0} , K_{s \conc 1 } ) \).

\begin{example}\label{xmp:Cntorofmeasure2}
Fix \( \varepsilon _n > 0 \) such that \( \sum_{n = 0}^\infty 2^n \varepsilon _n = 1 \), and consider \( \seqof{ U_s }{ s \in \pre{ < \omega }{2} } \), the Cantor-scheme on \( \R \) defined as follows: each \( U_s \) is an open interval \( ( a_s ; b_s ) \) with \( a_\emptyset = 0 \), \( b_\emptyset = 2 \), and 
\[
 a_{s \conc 0 } = a_s , \quad b_{s \conc 0 } = ( a_s + b_s - \varepsilon _{\lh s} ) / 2 , \quad a_{s \conc 1 } = ( a_s + b_s + \varepsilon _{\lh s} ) / 2, \quad b_{ s \conc 1 } = b_s .
\]
In other words, \( U_{s \conc 0 } \) and \( U_{s \conc  1 } \) are obtained by removing from \( U_s \) a closed centered interval of length \( \varepsilon _{\lh s} \).
This scheme has shrinking diameter, so we obtain a Cantor set \( K \subseteq [ 0 ; 2 ] \).
Note that for this Cantor scheme the function \( F \), defined as in~\eqref{eq:homeoCantorscheme}, is measure preserving between \( \pre{\omega }{2} \) with \( \measurecantor \) and \( K \) with the induced Lebesgue measure \( \lambda \).
\end{example}

Cantor-schemes on \( \R \) can be generalized by using ternary sequences instead of binary ones.
Let \( \seqof{ K_s , I_s^- , I_s^+ }{ s \in \pre{ < \omega}{ \set{ -1 , 0 , 1 } } } \) be such that \( K_s = [ a_s ; b_s ] \) and \( I_s^- = ( c_s^- ; d_s^- ) \), \( I_s^+ = ( c_s^+ ; d_s^+ ) \), with \( a_s < c_s^- < d_s^- < c_s^+ < d_s^+ < b_s \) and \( K_{s \conc \seq{-1}} = [ a_s ; c_s^- ] \), \( K_{s \conc \seq{0}} = [ d_s^- ; c_s^+ ] \), and \( K_{s \conc \seq{1}} = [ d_s^+ ; b_s ] \).
In other words, the intervals \( K_{ s \conc \seq{ i } } \) with \( i \in \set{ -1 , 0 , 1 } \) are obtained by removing from \( K_s \) two open intervals \( I_s^- \) and \( I_s^+ \).
Let 
\[
 K^{( n )} = \bigcup_{s \in \pre{ n }{ \set{ -1 , 0 , 1 } } } K_s \text{ and } K = \bigcap_{n \in \omega } K^{( n )} .
\] 
We dub this a \markdef{triadic Cantor-construction}. 
Note that \( K^{( n )} \) is the disjoint union of the closed intervals \( K_s \) for \( s \in \pre{ n }{ \set{ -1 , 0 , 1 } } \); in other words these \( K_s \) are the connected components of \( K^{( n )} \).
We say that this construction has shrinking diameter if \( \lim_{n \to \infty }\card{K_{ z \restriction n }} = 0 \) for all \( z \in \pre{ \omega}{ \set{ -1 , 0 , 1 } } \), and in this case we have a homeomorphism just like in~\eqref{eq:homeoCantorscheme}, that is \( F \colon \pre{ \omega}{ \set{ -1 , 0 , 1 } } \to K \), \( F ( z ) ={} \)the unique element of \( \bigcap_{n \in \omega } K_{ z \restriction n } \).
Since \( \pre{ \omega }{ 2 } \) and \( \pre{ \omega }{ 3 } \) are homeomorphic, this is just a Cantor-construction in disguise.

If the triadic Cantor-construction is of \emph{non-shrinking} diameter, a map like in~\eqref{eq:homeoCantorscheme} is undefined, and the map \( \pre{ \omega}{ \set{ -1 , 0 , 1 } } \to \KK ( \R ) \), \( z \mapsto \bigcap_{n \in \omega } K_{ z \restriction n } \) is not continuous.
On the other hand, regardless whether the Cantor-construction is of shrinking diameter, there is a continuous surjection
\begin{equation}\label{eq:continuouscodingspongy}
G \colon K \onto \pre{ \omega}{ \set{ -1 , 0 , 1 } } , \quad K \ni x \mapsto G ( x ) \colon \omega \to \set{ -1 , 0 , 1 } 
\end{equation}
defined as follows: if \( K_s \) is the connected component of \( K^{( n )} \) to which \( x \) belongs, 
\[ 
G ( x ) ( n ) = i \IFF x \in K_{ s \conc \seq{ i }} .
\]
Note that the connected components of \( K \) are the \( \bigcap_{n} K_{ z \restriction n} \), for \( z \in \pre{ \omega}{ \set{ -1 , 0 , 1 } } \).
In Section~\ref{subsec:spongy} we define a spongy subset of \( \R \) via a triadic Cantor-construction of non-shrinking diameter.

\subsection{Embedding the Cantor set in a measure preserving way}
A basic result in Descriptive Set Theory states that an uncountable Polish space contains a Cantor set.
The next result shows that the embedding can be taken to be measure-preserving.

\begin{theorem}\label{thm:embeddingCantorinPolish}
Suppose \( \mu \) and \( \nu \) are nonsingular Borel measures on a Polish space \( ( X , d ) \) and on the Cantor set \( \pre{\omega }{2} \), respectively. 
Suppose also \( \nu \) is fully supported, and that
\[
 \exists Y \in \Bor ( X ) \left ( \nu ( \pre{\omega }{2} ) < \mu ( Y ) < \infty \right ) . \tag{\( * \)}
\]
Then there is a continuous injective \( H \colon \pre{\omega }{2} \to X \) that preserves the measure.
\end{theorem}

The assumption (\( * \)) holds when \( \mu \) is \( \sigma \)-finite and \( \nu ( \pre{\omega}{2} ) < \mu ( X ) \).
The proof of Theorem~\ref{thm:embeddingCantorinPolish} is based on a simple combinatorial fact, which can be formulated as follows: if we have empty barrels of capacity \( b_1, \dots , b_n \) and sufficiently small amphoræ of capacity \( a_1, \dots , a_m \) so that \( a_1 + \dots + a_m < b_1 + \dots + b_n \), it is possible to pour the wine of the amphoræ into the barrels so that the content of each amphora is poured into a single barrel. 

\begin{lemma}\label{lem:embeddingCantorinPolish}
Let \( 0 < a < b \) and \( 0 < A < B \) be real numbers.
\begin{subequations}
\begin{enumerate-(a)}
\item\label{lem:embeddingCantorinPolish-a}
For all \( b_1 , \dots , b_n > 0 \) such that \( b = b_1 + \cdots + b_n \) there is an \( r > 0 \) with the following property: for all \( 0 < a_1 , \dots , a_m \leq r \) such that \( a = a_1 + \cdots + a_m \), there are pairwise disjoint (possibly empty) sets \( I_1 \cup \dots \cup I_n = \setLR{1 , \dots , m} \) such that for all \( k = 1 , \dots , n \)
\begin{equation}\label{eq:lem:embeddingCantorinPolish-a}
 \sum_{i \in I_k} a_i < b_k .
\end{equation}
\item\label{lem:embeddingCantorinPolish-b}
For all \( A_1 , \dots , A_N > 0 \) such that \( A = A_1 + \cdots + A_N \) there is an \( R > 0 \) with the following property: for all \( 0 < B_1 , \dots , B_M \leq R \) such that \( B_1 + \cdots + B_M = B \), there are pairwise disjoint nonempty sets \( J_1 \cup \dots \cup J_N = \setLR{1 , \dots , M} \) such that for all \( k = 1 , \dots , N \)
\begin{equation}\label{eq:lem:embeddingCantorinPolish-b}
A_k < \sum_{j \in J_k} B_j .
\end{equation}
Moreover the \( J_k \)s can be taken to be consecutive intervals, that is there are natural numbers \( j_0 = 0 < j_1 < \dots < j_N = M \) such that \( J_k = \setLR{ j_{k - 1} + 1 , \dots , j_k } \).
\end{enumerate-(a)}
\end{subequations}
\end{lemma}

\begin{proof}
\ref{lem:embeddingCantorinPolish-a}
Given \( b_1 , \dots , b_n \), let \( r = ( b - a ) / n \).
Suppose we are given \( 0 < a_1 , \dots , a_m \leq r \).
By induction on \( k \), construct pairwise disjoint sets \( I_k \subseteq \setLR{1 , \dots , m} \) that are maximal with respect to~\eqref{eq:lem:embeddingCantorinPolish-a}, and let \( I = I_1 \cup \dots \cup I_n \).
If \( I \neq \setLR{1 , \dots , m} \), then by maximality of \( I_k \),
\[ 
b_k \leq \frac{b - a}{n} + \sum_{ i \in I_k} a_i ,
\] 
so we would have
\[
b = \sum_{k = 1}^n b_k \leq ( b - a ) + \sum_{i \in I} a_i < ( b - a ) + \sum_{i = 1}^m a_i = b - a + a = b ,
\]
a contradiction.

\smallskip

\ref{lem:embeddingCantorinPolish-b}
The proof is similar to the one of~\ref{lem:embeddingCantorinPolish-a}.
If \( N = 1 \) there is nothing to prove, so we may assume otherwise.
Given \( A_1 , \dots , A_N \), let \( R = ( B - A ) / ( N - 1 ) \).
Suppose we are given \( 0 < B_1 , \dots , B_M \leq R \).
By induction on \( k \), we shall construct \( j_0 = 0 < j_1 < \dots < j_N = M \) such that each \( J_k = \setLR{ j_{k - 1} + 1 , \dots , j_k } \) satisfies~\eqref{eq:lem:embeddingCantorinPolish-b}, and it is least such, except possibly the last one \( j_N \).
The definition of \( j_1 \) is clear: it is the least \( j \leq M \) such that \( \sum_{h = 0}^{j} B_h > A_1 \), and such number exists since \( A_1 < A < B \).
We must show that the other \( j_k \)s exist, i.e. that the construction does not break-down before step \( N \).
Towards a contradiction, suppose \( 1 \leq \bar{N} < N \) is least such that \( j_{\bar{N} + 1 } \) is not defined.
By construction \( A_k + R > \sum_{i \in J_k} B_i \) for all \( k \leq \bar{N} \), and therefore 
\[ 
\sum_{k = 1}^{\bar{N}} A_k > \sum_{ i = 1}^{j_{\bar{N}}} B_i - \bar{N} R ,
\] 
and by case assumption \( A_{\bar{N} + 1} > \sum_{ i = j_{\bar{N}} + 1 }^M B_i \), if \( j_{\bar{N}} < M \).
Then 
\[
 A \geq \sum_{k = 1}^{\bar{N} + 1 } A_k > \sum_{i = 1}^M B_i - \bar{N} R \geq B - ( N - 1 ) R = A ,
\]
a contradiction.
\end{proof}

We now turn to the proof of Theorem~\ref{thm:embeddingCantorinPolish}.
The Cantor scheme construction with shrinking diameters guarantees that there is a continuous embedding \( f \colon \pre{\omega}{2} \to X \), but the map \( f \) need not be measure preserving---in fact it can happen that \( f ( \pre{\omega}{2} ) \) is \( \mu \)-null.
Of course we could modify the Cantor scheme by using Borel subsets of \( X \) of appropriate measure, but then we would have no control on the diameters of these Borel sets.
The cure is to carefully mix these two approaches, so that the construction succeeds. 

\begin{proof}[Proof of Theorem~\ref{thm:embeddingCantorinPolish}]
We claim it is enough to prove the result when \( \nu ( \pre{\omega}{2} ) < \mu ( X ) < + \infty \).
In fact if \( Y \in \Bor ( X ) \) and \( \nu ( \pre{\omega}{2} ) < \mu ( Y ) < + \infty \) then there is a finer topology \( \tau \) on \( X \) so that \( Y \) with the topology induced by \( \tau \) is Polish~\cite[][Theorem 13.1]{Kechris:1995kc}, so that any continuous injective measure preserving map \( H \colon \pre{\omega }{2} \to ( Y , \tau ) \) is also continuous as a function \( H \colon \pre{\omega }{2} \to X \) when \( X \) is endowed with the original topology.
Therefore we may assume that
\[
\nu ( \pre{\omega}{2} ) < \mu ( X ) < + \infty .
\]
By a result of Lusin and Souslin~\cite[][Theorem 13.7]{Kechris:1995kc}, \( X \) is the continuous injective image of a closed subset of the Baire space, so we may fix a pruned tree \( T \) on \( \omega \) and a continuous bijection \( f \colon \body{T} \to X \).
To avoid ambiguity we write 
\begin{align*}
\tilde{\Nbhd}_t & = \setofLR{ z \in \body{T}}{ t \subseteq z } , & \Nbhd_s & = \setofLR{ z \in \pre{\omega }{2} }{ s \subseteq z } 
\end{align*}
to denote the basic open neighborhood of \( \body{T} \) and of \( \pre{\omega }{2} \) determined by \( t \in T \) and \( s \in \pre{ < \omega }{2} \).
The measure \( \mu \) together with \( f \) induces a measure \( \mu' \) on \( \body{T} \) defined by 
\[ 
\mu' ( \tilde{\Nbhd}_t ) = \mu ( f ( \tilde{\Nbhd}_t ) ) , 
\]
and by tightness, there is a pruned, finite branching \( T' \subseteq T \) such that \( \nu ( \pre{\omega }{2} ) < \mu' ( \body{T'} ) \).
Without loss of generality we may assume \( T' \) is \markdef{normal}, that is the set of successors of \( t \in T' \) is \( \setofLR{t \conc \seq{ i }}{ i < n } \) for some \( n \in \omega \).
Therefore, it is enough to show that there is an injective, continuous \( g \colon \pre{\omega }{2} \to \body{T'} \), such that \( \nu ( \Nbhd_s ) = \mu' ( g ( \Nbhd_s ) ) \), for all \( s \in \pre{ < \omega }{2} \) since then \( f \circ g \colon \pre{\omega }{2} \to X \) would be injective, continuous, and it follows that \( f \circ g \) is measure-preserving, as required.
Therefore, it all boils-down to prove that:
\begin{quote}
If \( T \) is a pruned, normal, finitely branching tree on \( \omega \), and \( u \colon \pre{ < \omega }{2} \to ( 0 ; + \infty ) \) and \( w \colon T \to ( 0 ; + \infty ) \) induce fully supported, nonsingular, Borel measures \( \nu \) on \( \pre{\omega }{2} \) and \( \mu \) on \( \body{T} \), respectively, such that \( u ( \emptyset ) = \nu ( \pre{\omega }{2} ) < \mu ( \body{T} ) = w ( \emptyset ) \), then there is a continuous \( \varphi \colon \pre{ < \omega }{2} \to T \) such that the induced function \( f_ \varphi \colon \pre{\omega }{2} \to \body{T} \) is injective and \( \nu \left ( \Nbhd_s \right ) = \mu \left ( f_ \varphi ( \Nbhd_s ) \right ) \).
\end{quote}

Suppose we are given \( T \), \( u \) and \( w \) as above.
The function \( \varphi \colon \pre{ < \omega }{2} \to T \) is first defined on \( \bigcup_{k \in \omega} \pre{L_k}{2} \) for some suitable increasing sequence \( ( L_k )_k \), and then extended to all of \( \pre{ < \omega }{2} \) by requiring that when \( L_k < \lh s < L_{k + 1} \), then \( \varphi ( s ) = \varphi ( s \restriction L_k ) \).

We require that 
\[ 
s \in \pre{L_k}{2} \IMPLIES \varphi ( s ) \in \Lev_{M_k} ( T ) \equalsdef \setof{ t \in T }{ \lh ( t ) = M_k } ,
\] 
where \( ( M_k )_k \) is a suitable increasing sequence.
The function \( f_ \varphi \) will be injective, but the same need not be true of the map \( \varphi \): even if \( x \restriction L_k \neq y \restriction L_k \) one might need to reach a much larger \( L_n \) in order to witness \( \varphi ( x \restriction L_n ) \neq \varphi ( y \restriction L_n ) \) and hence \( f_ \varphi ( x ) \neq f_ \varphi ( y ) \).
For each \( t \in \varphi ( \pre{ L_k }{ 2 } ) = \setofLR { \varphi ( s ) \in \pre{ M_k }{ 2 } }{ s \in \pre{L_k}{2} } \) let
\[
\mathcal{A}_k ( t ) = \setof{ s \in \pre{L_k}{2} }{ \varphi ( s ) = t } 
\]
so that \( \setofLR{\mathcal{A}_k ( t ) }{ t \in \varphi ( \pre{ L_k }{ 2 } ) } \) is the partition of \( \pre{ L_k }{ 2 } \) given by the fibers of \( \varphi \).

Set \( \varphi ( \emptyset ) = \emptyset \), \( L_0 = M_0 = 0 \) and let \( \delta_0 \) be a positive real such that \( u ( \emptyset ) < w ( \emptyset ) < u ( \emptyset ) + \delta_0 \).

Fix \( k \in \omega \) and suppose that \( L_k \), \( M_k \), and \( \delta _k \) have been defined, together with the values \( \varphi ( s ) \) for all \( s \in \pre{L_k}{2} \), and suppose that for every \( t \in \varphi ( \pre{L_k}{2} ) \),
\begin{equation}\label{eq:th:embeddingCantorinPolish}
\sum_{s \in \mathcal{A}_k ( t ) } u ( s ) < w ( t ) < \delta_k + \sum_{s \in \mathcal{A}_k ( t ) } u ( s ) . 
\end{equation}
The goal is to define \( L_{k + 1} \), \( M_{k + 1} \), \( \delta _{k + 1} \) and the values \( \varphi ( s ) \in \Lev_{M_{k + 1}} ( T ) \) for \( s \in \pre{L_{k + 1}}{2} \).
Let \( \delta _{ k + 1} = 2^{- 2 L_k } \).
(The actual values of the \( \delta _j \)s are only used in Claim~\ref{cl:embeddingCantorinPolish4} to certify that \( f_ \varphi \) is measure preserving, and play no significant role in the construction of \( \varphi \).)

\begin{claim}\label{cl:embeddingCantorinPolish1}
Let \( R > 0 \).
Then there is \( M \) such that \( w ( t ) < R \) for all \( t \in \Lev _{M} ( T ) \).
Moreover, this relation implies that \( \FORALL{ M' > M } \FORALL{ t \in \Lev _{M'} ( T ) } ( w ( t ) < R ) \).

Similarly, \( \EXISTS{ M } \FORALL{ s \in \pre{M}{2}} u ( s ) < R \), thus \( \FORALL{ M' > M } \FORALL{ s \in \pre{M'}{2} } ( u ( s ) < R ) \).
\end{claim}

\begin{proof}
Otherwise, the tree \( \setofLR{ t \in T}{w ( t ) \geq R} \) would be infinite.
Since it is finitely branching, it would be ill-founded, contradicting non-singularity of \( \mu \).
\end{proof}

Fix a \( t \in \varphi ( \pre{L_k}{2} ) \).
Applying Lemma~\ref{lem:embeddingCantorinPolish}\ref{lem:embeddingCantorinPolish-b} to the numbers
\begin{align*}
A_{ s \conc i } &= u ( s \conc i ) , & ( \text{for } ( s , i ) \in \mathcal{A}_k ( t ) \times 2 )
\\
A & = \textstyle \sum_{ s \in \mathcal{A}_k ( t ) } u ( s ) ,
\\
 B & = w ( t ) ,
\end{align*}
a value \( R_t \) is obtained such that whenever \( B_1 , \ldots , B_M \leq R_t \) and \( B_1 + \ldots + B_M = B \), there exists a partition of \( \set{ 1 , \ldots , M} \) into sets \( J_{ s \conc i } \) such that \( u ( s \conc i ) < \sum_{ h \in J_{ s \conc i } } B_h \).
Let \( R = \min \setofLR{\delta_{k + 1} , R_t }{ t \in \varphi ( \pre{L_k}{2} ) } \).
Applying Claim~\ref{cl:embeddingCantorinPolish1}, let \( M_{k + 1} > M_k \) be such that \( w ( t' ) < R \) for all \( t' \in \Lev _{M_{k + 1}} ( T ) \).
Let 
\[ 
D_t = \setof{ t' \in \Lev _{M_{k + 1}} ( T ) }{t' \restriction M_k = t} .
\]
It follows that \( B_{t'} = w ( t' ) < R \) for \( t' \in D_t \), and \( B = \sum_{ t' \in D_t } B_{t'} \) so there is a partition \( \setof{J_{ s \conc i } }{ ( s , i ) \in \mathcal{A}_k ( t ) \times 2 } \) of \( D_t \) such that 
\begin{equation}\label{eq:th:embeddingCantorinPolish1}
u ( s \conc i ) < \sum_{t' \in J_{ s \conc i }} w ( t' ) .
\end{equation}
Choose 
\[ 
C_{ s \conc i } \subseteq J_{ s \conc i } 
\]
minimal so that \( u ( s \conc i ) < \sum_{t' \in C_{ s \conc i }} w ( t' ) \).
By the choice of \( R \) one also has 
\begin{equation}\label{eq:th:embeddingCantorinPolish2}
 \sum_{t' \in C_{ s \conc i }} w ( t' ) < u ( s \conc i ) + \delta_{k + 1} . 
\end{equation}

Now, for each \( ( s , i ) \in \mathcal{A}_k ( t ) \times 2 \), apply Lemma~\ref{lem:embeddingCantorinPolish}\ref{lem:embeddingCantorinPolish-a} to the numbers 
\begin{align*}
b_{t'} & = w ( t' ) , & ( \text{for } t' \in C_{ s \conc i } )
\\
b & = \textstyle \sum_{t' \in C_{ s \conc i } } w ( t' ) 
\\
a & = u ( s \conc i ) 
\end{align*}
to get a value \( r_{ s \conc i } \) such that whenever \( 0 < a_1 , \dots , a_m \leq r_{ s \conc i } \) and \( a_1 + \ldots + a_m = a \), there are pairwise disjoint, possibly empty, subsets \( I_{t' } \) of \( \set{ 1 , \ldots , m} \) such that \( \bigcup_{t' \in C_{ s \conc i }} I_{t' } = \set{ 1 , \ldots , m} \) and \( \sum_{h \in I_{t' }} a_h < w ( t' ) \).
Let \( r \) be the least of all \( r_{ s \conc i } \).
By Claim~\ref{cl:embeddingCantorinPolish1}, there is \( L_{k + 1} > L_k \) such that \( u ( s ) < r \) for all \( s \in \pre{L_{k + 1}}{2} \).
Set \( E_{ s \conc i } = \setofLR{ s' \in \pre{L_{k + 1}}{2}}{s \conc i \subseteq s' } \), so that \( \sum_{s' \in E_{s \conc i}} u ( s' ) = u ( s \conc i ) \).
By Lemma~\ref{lem:embeddingCantorinPolish}\ref{lem:embeddingCantorinPolish-a}, \( E_{ s \conc i } \) is partitioned into sets \( I_{t' } \), for \( t' \in C_{ s \conc i } \), such that \( \sum_{ s' \in I_{t' }} u ( s' ) < w ( t' ) \).
By~\eqref{eq:th:embeddingCantorinPolish2} we have \( w ( t' ) < \delta_{k + 1} + \sum_{ s' \in I_{t' }} u ( s' ) \).
Let \( \varphi ( s' ) = t' \) for \( s' \in I_{t' } \), so that \( \varphi \) is defined on \( \pre{ L_{ k + 1}}{2} \).
This concludes the definition of \( \varphi \colon \pre{ < \omega }{2} \to T \).
Note that by construction \( \mathcal{A}_{ k + 1 } ( t' ) = I_{t'} \), so~\eqref{eq:th:embeddingCantorinPolish} holds for \( k + 1 \).

\begin{claim}\label{cl:embeddingCantorinPolish2}
The function \( \varphi \colon \pre{ < \omega }{2} \to T \) is continuous.
\end{claim}

\begin{proof}
First notice that \( \varphi \) is monotone, directly from the definition.
Moreover, \( \lim_{k \to \infty } \lh \varphi ( x \restriction L_k ) = \lim_{k \to \infty }M_k = + \infty \), since the sequence \( M_k \) is increasing.
\end{proof}

\begin{claim}\label{cl:embeddingCantorinPolish3}
\( f_{\varphi } \colon \pre{\omega }{2} \to \body{T} \) is injective.
\end{claim}

\begin{proof}
Let \( x , y \) be distinct elements of \( \pre{\omega }2 \), and let \( k \in \omega \) such that \( x \restriction L_k\neq y \restriction L_k \).
Since \( \varphi ( x \restriction L_{k + 1} ) \in C_{x \restriction ( L_k + 1 ) } \), \( \varphi ( y \restriction L_{ k + 1 } ) \in C_{y \restriction ( L_k + 1 ) } \), and \( C_{x \restriction ( L_k + 1 ) }\cap C_{y \restriction ( L_k + 1 ) } = \emptyset \), it follows that \( \varphi ( x \restriction L_{k + 1} ) \neq\varphi ( y \restriction L_{k + 1} ) \), whence \( f_{\varphi } ( x ) \neq f_{\varphi } ( y ) \).
\end{proof}

\begin{claim}\label{cl:embeddingCantorinPolish4}
\( \FORALL{ s \in \pre{ < \omega }{2} } [ \nu ( \Nbhd_s ) = \mu \left ( f_{\varphi } ( \Nbhd_s ) \right ) ] \).
\end{claim}

\begin{proof}
It is enough to establish the claim for \( s \in \pre{ L_k }{ 2 } \), for some \( k > 0 \). 
For \( h \geq k \) let \( X ( h , s ) = \bigcup \setof{C_{ s' \conc i }}{s' \supseteq s \wedge \lh ( s' ) = L_h \wedge i \in 2 } \).

First remark that 
\begin{equation}\label{eq:th:embeddingCantorinPolish5}
f_{\varphi } ( \Nbhd_s ) = \bigcap_{h \geq k + 1} \bigcup_{p \in X ( h , s ) } \Nbhd_p .
\end{equation}
To prove that left-hand side is contained in the right-hand side argue as follows.
Given \( x \supseteq s \), for \( h \geq k + 1 \) choose \( s \subseteq s' \in \pre{L_h}{2} \), \( i \in 2 \), and \( s'' \in \pre{< \omega}{2} \) such that \( \lh ( s'' ) = L_{h + 1} - L_h - 1 \) and \( s' \conc \seq{ i } \conc s'' \subseteq x \); then \( \varphi ( s' \conc \seq{ i } \conc s'' ) \in C_{s' \conc i} \).
Conversely, pick \( y \) in the right-hand side of the equation: for every \( h \geq k + 1 \) there are \( s_h \in \pre{L_h}{2} \), \( i_h \in 2 \), \( p_h \in C_{s_h \conc i_h} \) such that \( s \subseteq s_h \) and \( p_h \subseteq y \), and since all \( p_h \) are compatible, all \( s_h \) must be compatible as well by construction, so their union is an element \( x \in \Nbhd_s \) such that \( f_{\varphi } ( x ) = y \).

Equation~\eqref{eq:th:embeddingCantorinPolish5} yields \( f_{ \varphi } \) as a decreasing intersection of disjoint unions, so
\[
\mu \left ( f_{\varphi } ( \Nbhd_s ) \right ) = \inf_{ h \geq k + 1} \sum_{ p \in X ( h , s ) } w ( p ) .
\]

Now, for any given \( h \geq k + 1 \), letting \( Y ( h ; s ) = \setofLR{ s' \conc i }{ s \subseteq s' \in \pre{L_h}{2}, i \in 2 } \),
\[
\begin{split}
\nu ( \Nbhd_s ) & = \sum_{s'' \in Y ( h ; s ) } u ( s'' ) 
\\
& < \sum_{p \in C_{s''} , s'' \in Y ( h ; s ) } w ( p ) 
\\
& = \sum_{ p \in X ( h , s ) } w ( p )
\\
 & < \sum_{s'' \in Y ( h ; s ) } ( u ( s'' ) + \delta_{h + 1} ) 
 \\
 &= \sum_{ s'' \in Y ( h ; s ) } u ( s'' ) + 2^{L_h + 1} \delta_{h + 1} .
\end{split}
\]
As \( \lim_{h \to \infty} \sum_{s'' \in Y ( h ; s ) } u ( s'' ) + 2^{L_h + 1}\delta_{h + 1} = \nu ( \Nbhd_s ) \) the claim is proved.
\end{proof} 
This completes the proof of Theorem~\ref{thm:embeddingCantorinPolish}.
\end{proof}

\section{The density function}\label{sec:densityfunction}
Let \( ( X , d , \mu ) \) be a fully supported, locally finite metric measure space and let \( A \in \MEAS_\mu \).
For \( x \in X \), the \markdef{upper} and \markdef{lower density of \( x \) at} \( A \) are 
\[
\density^ + _A ( x ) = \limsup_{ \varepsilon {\downarrow} 0}\frac{ \mu ( A \cap \Ball ( x ; \varepsilon ))}{ \mu ( \Ball ( x ; \varepsilon ) )} ,\qquad \density^- _A ( x ) = \liminf_{ \varepsilon {\downarrow} 0}\frac{ \mu ( A \cap \Ball ( x ; \varepsilon ))}{ \mu ( \Ball ( x ; \varepsilon ) )} .
\]
The \markdef{oscillation of \( x \) at} \( A \) is
\[
\oscillation_A ( x ) = \density^ + _A ( x ) - \density^- _A ( x ) .
\]
When \( \oscillation_A ( x ) = 0 \), that is to say: \( \density^ + _A ( x ) \leq \density^- _A ( x ) \), the value \( \density^ + _A ( x ) = \density^- _A ( x ) \) is called the \markdef{density of \( x \) at \( A \)}
\[
\density_A ( x ) = \lim_{ \varepsilon {\downarrow} 0}\frac{ \mu ( A \cap \Ball ( x ; \varepsilon ))}{ \mu ( \Ball ( x ; \varepsilon ) )} .
\]
It is important that in the computation of \( \density _A \) and \( \oscillation_A \) balls of every radius \( \varepsilon \) be considered, and not just for \( \varepsilon \) ranging over a countable set --- see Section~\ref{subsec:basisofdifferentiation}.
Note that if \( \mu ( \set{x} ) > 0 \) and \( x \in A \), then \( \density_A ( x ) = 1 \) for trivial reasons. 

The limit \( \density_A ( x ) \) does not exist if and only if \( \oscillation_A ( x ) > 0 \).
In any case if \( A =_\mu B \) then 
\begin{align*}
 \density_{ A^\complement} ( x ) & = 1 - \density_A ( x ) ,
& \density_A ( x ) & = \density_B ( x ) ,
\\
 \oscillation_{ A^\complement } ( x ) & = \oscillation_A ( x ) ,
&
\oscillation_A ( x ) & = \oscillation_B ( x ) ,
\end{align*}
in the sense that if one of the two sides of the equations exists, then so does the other one, and their values are equal.
Let
\begin{equation}\label{eq:Phi}
 \Phi ( A ) = \setofLR{ x \in X}{ \density_A ( x ) = 1 } . 
\end{equation}
The set of \markdef{blurry points of \( A \)} is
\[
\Blur ( A ) = \setof{ x \in X }{ \oscillation_A ( x ) > 0 },
\]
the set of \markdef{sharp points of \( A \)} is
\[
\Sharp ( A ) = \setof{ x \in X }{ \density_A ( x ) \in ( 0 ; 1 ) }
\]
and 
\[ 
\Exc ( A ) \equalsdef \Blur ( A ) \cup \Sharp ( A ) 
\] 
is the set of \markdef{exceptional points of \( A \)}.
For \( x \in X \) let 
\[
\begin{split}
\exc_A ( x ) & = \sup \setof{ \delta \leq 1 / 2 }{ \delta \leq \density^- _A ( x ) \leq \density^+ _A ( x ) \leq 1 - \delta }
\\
 & = \min \setLR{ \density_A^- ( x ) , 1 - \density_A^+ ( x ) } ,
 \\
 & \leq 1 / 2 .
\end{split} 
\]
If \( x \in \Phi ( A ) \cup \Phi ( A^\complement ) \) then \( \exc_A ( x ) = 0 \), so this notion is of interest only when \( x \in \Exc ( A ) \).
Let 
\begin{equation*} 
\boldsymbol{ \delta }_A = \sup \setof{ \exc_A ( x ) }{ x \in X } \leq 1 / 2 .
\end{equation*}
If \( A \) is either null or co-null, then \( \boldsymbol{ \delta }_A = 0 \), so this justifies the restriction to nontrivial sets in the following definition:
\begin{equation}\label{eq:delta(X)}
\boldsymbol{ \delta } ( X ) = \inf \setof{ \boldsymbol{ \delta }_A }{ \eq{A} \in \MALG \setminus \set{ \eq{\emptyset} , \eq{X} } } .
\end{equation}
The following are easily checked.
\begin{subequations}
\begin{gather}
\oscillation_A ( x ) = 0 \IMPLIES \exc_A ( x ) = \min \setLR{ \density_A ( x ) , 1 - \density_A ( x ) }
\\
\exc_A ( x ) = 1 / 2 \IFF \density_A ( x ) = 1 / 2 
\\
 \exc_A ( x ) = 0 \IFF x \in \Phi ( A ) \cup \Phi ( A^\complement ) \OR \density^-_A ( x ) = 0 \OR \density^+_A ( x ) = 1
\\
\boldsymbol{ \delta }_A = 0 \IMPLIES \Sharp ( A ) = \emptyset .
\end{gather}
\end{subequations}

\begin{remarks}\label{rmks:exc}
\begin{enumerate-(a)}
\item
If \( A \) is clopen and nontrivial in \( X \), then \( \boldsymbol{ \delta } _A = 0 \), so if \( X \) is disconnected, then \( \boldsymbol{ \delta } ( X ) = 0 \).
In particular \( \boldsymbol{ \delta } ( X ) = 0 \) when \( X \) is a closed subset of the Baire space \( \pre{\omega}{\omega} \). 
The case when \( X = \R \) is completely different---see Section~\ref{sec:solidsets}.
\item
The notions above are \( =_\mu \)-invariant, that is if \( A =_\mu B \) then \( \exc_A ( x ) = \exc_B ( x ) \), \( \boldsymbol{ \delta }_A = \boldsymbol{ \delta }_B \), and \( \Blur ( A ) = \Blur ( B ) = \Blur ( A^\complement ) \), and similarly for \( \Sharp \) and \( \Exc \).
\end{enumerate-(a)}
\end{remarks}

\subsection{Density in the real line}\label{subsec:realline}
Let \( \lambda \) be the Lebesgue measure on \( \R \).
For \( A \subseteq \R \) a measurable set, the \markdef{right density} of \( A \) at \( x \) is defined as
\[
\density_A ( x^ + ) = \lim_{ \varepsilon {\downarrow} 0} \frac{ \lambda ( A \cap ( x ; x + \varepsilon ) )}{ \varepsilon } ,
\]
and the \markdef{left density} \( \density_A ( x^ - ) \) is defined similarly.
If \( \density_A ( x^ + ) \) and \( \density_A ( x ^- ) \) both exist, then \( \density_A ( x ) \) exists, and in this case
\[
\density_A ( x ) = \frac{\density_A ( x^ + ) + \density_A ( x ^- )}{2} .
\]
Conversely, 
\[
\density_A ( x ) \in \setLR{0 , 1} \implies \density_A ( x^ + ) = \density_A ( x ^- ) = \density_A ( x ) .
\]
This result cannot be extended to other values.

\begin{example}\label{xmp:densitybutnoleftorrightdensities}
The set
\[
A = \bigcup_{n} ( - 2^{ - 2 n - 1 } ; - 2^{ - 2 n - 2 } ) \cup ( 2^{ - 2n - 1} ; 2^{ -2n } )
\]
is open and such that \( \density_A ( 0 ) = 1 / 2 \), but \( \density_A ( 0^ + ) \) and \( \density_A ( 0 ^-) \) do not exist.
\end{example}

\subsection{Density in the Cantor and Baire spaces}
Suppose \( T \) is a pruned tree on \( \omega \), \( \mu \) is a finite Borel measure on \( \body{T} \) induced by some \( w \colon T \to [ 0 ; M ] \) as in Section~\ref{subsubsec:measureonCantor}.
Since the metric attains values in \( \setLR{0} \cup \setofLR{2^{-n}}{n \in \omega } \), then 
\[
\density^ + _A ( z ) = \limsup_{ n \to \infty } \frac{ \mu ( A \cap \Nbhd_{z \restriction n} )}{ w ( z \restriction n ) } , \qquad \density^- _A ( z ) = \liminf_{ n \to \infty } \frac{ \mu ( A \cap \Nbhd_{z \restriction n} )}{ w ( z \restriction n ) } .
\]
In particular, when \( T = \pre{ < \omega }{2} \) and \( \mu = \measurecantor \), then \( w ( s ) = 2^{- \lh s } \) so \( \frac{ \mu ( A \cap \Nbhd_{z \restriction n} )}{ w ( z \restriction n ) } = \mu ( \LOC{A}{s} ) \), and the equations above become 
\[
\density^ + _A ( z ) = \limsup_{ n \to \infty }\mu ( \LOC{A}{ z \restriction n } ) , \qquad \density^- _A ( z ) = \liminf_{ n \to \infty } \mu ( \LOC{A}{ z \restriction n } ) .
\]

\subsection{Bases for density}\label{subsec:basisofdifferentiation}
Let \( ( X , d , \mu ) \) be a fully supported, locally finite metric measure space. 
Although the definition of \( \density^\pm_A ( x ) \) requires that balls centered in \( x \) of all radii be considered, it is possible to compute the limit along some specific sequences converging to \( 0 \).

\begin{definition}\label{def:densitybasis}
Suppose \( \varepsilon_n {\downarrow} 0 \) and let \( x\in X \).
\begin{enumerate-(i)}
\item\label{def:densitybasis-i}
\( ( \varepsilon _n )_n \) is a \markdef{basis for density at} \( x \) if for all \( A \in \MEAS_\mu \) and all \( r \in [ 0 ; 1 ] \)
\[ 
\lim_n \frac{ \mu ( A \cap \Ball ( x ; \varepsilon_n ) ) }{ \mu ( \Ball ( x ; \varepsilon_n ) ) } = r \IMPLIES \density_A ( x ) = r .
\] 
\item\label{def:densitybasis-ii}
\( ( \varepsilon _n )_n \) is a \markdef{strong basis for density at} \( x \) if for all \( A \in \MEAS_\mu \)
\[
\limsup_{n \to \infty } \frac{ \mu ( A\cap \Ball ( x ; \varepsilon_n ) ) }{\mu ( \Ball ( x ; \varepsilon_n ) ) } = \density^+_A ( x ) .
\]
\end{enumerate-(i)}
\end{definition}

If \( ( \varepsilon _n )_n \) is a strong basis for density at \( x \), then by taking complements 
\[ 
\liminf_{n \to \infty } \frac{ \mu ( A\cap \Ball ( x ; \varepsilon_n ) ) }{\mu ( \Ball ( x ; \varepsilon_n ) ) } = \density^-_A ( x ) 
\]
for all \( A \in \MEAS_\mu \).
The sequence \( \varepsilon _n = 2^{-n} \) is a strong basis for density at every point, both in the Cantor and in the Baire space.

\begin{theorem}\label{thm:densitybasis}
Suppose \( \varepsilon_n {\downarrow} 0 \).
\begin{enumerate-(a)}
\item\label{thm:densitybasis-a}
If \( \lim_n \frac{ \mu ( \Ball ( x ; \varepsilon_{n + 1} ) ) }{ \mu ( \Ball ( x ; \varepsilon_n ) ) } = 1 \) then \( ( \varepsilon_n )_n \) is a strong basis for density at \( x \).

\item\label{thm:densitybasis-b}
If \( ( \varepsilon_n )_n \) is a basis for density at \( x \), and \( r \mapsto \mu ( \Ball ( x ; r ) ) \) is continuous, then \( \lim_n \frac{ \mu ( \Ball ( x ; \varepsilon_{n + 1} ) ) }{ \mu ( \Ball ( x ; \varepsilon_n ) ) } = 1 \), and hence \( ( \varepsilon_n )_n \) is a strong basis for density at \( x \)
\end{enumerate-(a)}
\end{theorem}

\begin{proof}
\ref{thm:densitybasis-a}
Let \( A \) be measurable, and suppose 
\[ 
\limsup_{n \to \infty} \frac{\mu ( \Ball ( x ; \varepsilon _n ) \cap A )}{ \mu ( \Ball ( x ; \varepsilon _n ) )} = r .
\]
Thus \( \density^+_A ( x ) \geq r \).
To prove the reverse inequality we must show that 
\[
 \FORALL{ \varepsilon > 0} \EXISTS{ \delta > 0} \FORALL{0 < \eta < \delta }\Bigl [ \frac{\mu ( \Ball ( x ; \eta ) \cap A )}{ \mu ( \Ball ( x ; \eta ) )} < r + \varepsilon \Bigr ]. 
\]
For each \( \varepsilon > 0 \) choose \( n_1 = n_1 ( \varepsilon ) \in \omega \) be such that
\[
m \geq n_1 \IMPLIES \frac{\mu ( \Ball ( x ; \varepsilon _m ) \cap A )}{ \mu ( \Ball ( x ; \varepsilon _m ) )} < r + \frac{\varepsilon}{2} .
\]
We must takes cases depending whether \( r \) is null or otherwise.

Suppose first \( r = 0 \).
For \( 0 < \varepsilon < 1 \) and let \( n_2 = n_2 ( \varepsilon ) \in \omega \) be such that 
\[
m \geq n_2 \IMPLIES \frac{\mu ( \Ball ( x ; \varepsilon _{m } ) )}{ \mu ( \Ball ( x ; \varepsilon _{ m + 1} ) )} \leq 1 + \varepsilon . 
\]
We claim that \( \delta = \varepsilon _{\bar{n}} \) will do, when \( \bar{n} = \max ( n_1 , n_2 ) \).
Let \( 0 < \eta < \delta \).
Since \( \varepsilon _n { \downarrow } 0 \), fix \( k \geq \bar{n} \) such that 
\begin{equation}\label{eq:th:densitybasis2}
 \varepsilon _{ k + 1} < \eta \leq \varepsilon _k . 
\end{equation}
Then 
\begin{align*}
 \frac{\mu ( \Ball ( x ; \eta ) \cap A )}{ \mu ( \Ball ( x ; \eta ) ) } & \leq \frac{\mu ( \Ball ( x ; \varepsilon _k ) \cap A )}{ \mu ( \Ball ( x ; \varepsilon _{ k + 1 } ) ) } & \text{by~\eqref{eq:th:densitybasis2}}
 \\
 & = \frac{\mu ( \Ball ( x ; \varepsilon _k ) \cap A )}{ \mu ( \Ball ( x ; \varepsilon _{ k } ) ) } \frac{\mu ( \Ball ( x ; \varepsilon _k ) )}{ \mu ( \Ball ( x ; \varepsilon _{ k + 1 } ) ) } 
 \\
 & \leq \frac{ \varepsilon ( 1 + \varepsilon ) }{2} 
 \\
 & < \varepsilon .
\end{align*}

Suppose now \( r > 0 \), and choose \( 0 < \varepsilon < r \).
Let \( n_2 = n_2 ( \varepsilon ) \) be such that 
\[
m \geq n_2 \implies \frac{ \mu ( \Ball ( x ; \varepsilon _m ) ) }{ \mu ( \Ball ( x ; \varepsilon _{ m + 1 } ) ) } \leq 1 + \frac{ \varepsilon }{ 4 r } .
\]
The argument is as before: let \( \delta = \varepsilon _{ \bar{n} } \) where \( \bar{n} = \max ( n_1 , n_2 ) \), and given \( 0 < \eta < \delta \), fix \( k \geq \bar{n} \) such that \( \varepsilon _{ k + 1 } < \eta \leq \varepsilon _k \).
Then
\[
\begin{split}
 \frac{\mu ( \Ball ( x ; \eta ) \cap A )}{ \mu ( \Ball ( x ; \eta ) ) } & \leq \frac{\mu ( \Ball ( x ; \varepsilon _k ) \cap A )}{ \mu ( \Ball ( x ; \varepsilon _{ k } ) ) } \frac{\mu ( \Ball ( x ; \varepsilon _k ) )}{ \mu ( \Ball ( x ; \varepsilon _{ k + 1 } ) ) } 
\\
 & \leq \left ( r + \frac{ \varepsilon }{ 2 } \right ) \left ( 1 + \frac{ \varepsilon }{ 4 r } \right )
 \\
 & < r + \varepsilon .
\end{split}
\]

\ref{thm:densitybasis-b}
Towards a contradiction, suppose there is \( r < 1 \) and a subsequence \( ( \varepsilon_{n_k} )_k \) such that 
\[ 
\lim_{k \to \infty } \frac{ \mu ( \Ball ( x ; \varepsilon_{ n_k + 1 } ) ) }{ \mu ( \Ball ( x ; \varepsilon_{n_k} ) ) } = r .
\]
For each \( n \), let \( \delta_n \in ( \varepsilon_{n + 1} ; \varepsilon_n ) \) be such that \( \mu ( \Ball ( x ; \delta_n ) ) = \frac{1}{2}[ \mu ( \Ball ( x ; \varepsilon_{n + 1 } ) ) + \mu ( \Ball ( x ; \varepsilon_n ) ) ] \).
Define
\[
A = \bigcup_{n }( \Ball ( x ; \delta_n ) \setminus \Ball ( x ; \varepsilon_{ n + 1 } ) ) .
\]

Then \( \mu ( A \cap \Ball ( x ; \varepsilon_n ) ) / \mu ( \Ball ( x ; \varepsilon_n ) ) = \frac 12 \).
On the other hand,
\[
\begin{split}
\frac{\mu (A \cap \Ball ( x ; \delta_n ) ) }{\mu ( \Ball ( x ; \delta_n ) ) }& = \frac{ \frac{1}{2} [ \mu ( \Ball ( x ; \varepsilon_n ) ) - \mu ( \Ball ( x ; \varepsilon_{n + 1 } ) )] + \frac{1}{2}[ \mu ( \Ball ( x ; \varepsilon_{n + 1 } ) ) ] }{ \frac{1}{2}[ \mu ( \Ball ( x ; \varepsilon_{n + 1 } ) ) + \mu ( \Ball ( x ; \varepsilon_n ) ) ] } 
 \\
& = \frac{\mu ( \Ball ( x ; \varepsilon_n ) ) }{\mu ( \Ball ( x ; \varepsilon_{n + 1 } ) ) + \mu ( \Ball ( x ; \varepsilon_n ) ) } 
\\
&= \Bigl ( \frac{\mu ( \Ball ( x ; \varepsilon_{n + 1 } ) ) }{\mu ( \Ball ( x ; \varepsilon_n ) ) } + 1 \Bigr )^{-1} .
\end{split}
\]
Since \( \bigl ( \frac{ \mu ( \Ball ( x ; \varepsilon_{n_k + 1 } ) )}{ \mu ( \Ball ( x ; \varepsilon_{n_k} ) ) } + 1 \bigr )^{-1} \rightarrow \frac {1}{r + 1} > \frac {1}{2} \), then \( (\varepsilon_n )_n \) is not a basis for density at \( x \).
\end{proof}

The next Example shows that ``\( \lim \)'' cannot be replaced by ``\( \limsup \)'' in the statement of Theorem~\ref{thm:densitybasis}.

\begin{example}\label{xmp:oscillatingdensity}
If \( \mu \) is nonsingular then for any \( x \in X \) there is a set \( A \in \MEAS_\mu \) such that for some sequence \( \varepsilon _n {\downarrow} 0 \),
\[
\lim_{n \to \infty} \frac{ \mu ( A \cap \Ball ( x ; \varepsilon_{2n} ) ) }{ \mu ( \Ball ( x ; \varepsilon_{2n} ) ) } = 1 \quad \text{and} \quad \lim_{n \to \infty} \frac{ \mu ( A \cap \Ball ( x ; \varepsilon_{2n + 1} ) )}{ \mu (\Ball ( x ; \varepsilon_{2 n + 1} ) ) } = 0 ,
\]
hence \( \oscillation_A ( x ) = 1 \).
Moreover \( A \) can be taken to be open or closed.

Choose \( ( \varepsilon _n )_n \) strictly decreasing, converging to \( 0 \), and such that 
\begin{equation}\label{eq:convergenceCamillo}
\lim_{n \to \infty} \frac{\mu ( \Ball ( x ; \varepsilon_{n + 1} ) )}{ \mu ( \Ball ( x ; \varepsilon_{n} ) ) } = 0 . 
\end{equation}
This can be done as \( \mu \) is nonsingular.
Let
\[
A = \bigcup_{n} \Ball ( x ; \varepsilon_{2n} ) \setminus \Ball ( x ; \varepsilon_{2n + 1} ) .
\]
Then
\[
\frac{\mu ( A \cap \Ball ( x ; \varepsilon_{2n} ) )}{\mu ( \Ball ( x ; \varepsilon_{2n} ) ) } > \frac{\mu ( \Ball ( x ; \varepsilon_{2n} ) \setminus \Ball ( x ; \varepsilon_{2n + 1} ) )}{\mu ( \Ball ( x ; \varepsilon_{2n} ) ) } = 1 - \frac{ \mu ( \Ball ( x ; \varepsilon_{2n + 1} ) )}{\mu ( \Ball ( x ; \varepsilon_{2n} ) ) } \to 1
\]
and
\[
\frac{\mu ( A \cap \Ball ( x ; \varepsilon_{2n + 1} ) )}{\mu ( \Ball ( x ; \varepsilon_{2n + 1} ) ) } < \frac{\mu ( \Ball ( x ; \varepsilon_{2n + 2} ) )}{\mu ( \Ball ( x ; \varepsilon_{2n + 1} ) ) } \to 0 .
\]

To construct an \( A \) which is open or close, argue as follows.
Let \( ( \varepsilon '_n )_n {\downarrow} 0 \) and satisfying~\eqref{eq:convergenceCamillo} and let \( \varepsilon _n = \varepsilon _{2n}' \).
Then \( \bigcup_{n} \Ball ( x ; \varepsilon_{2n} ) \setminus \Cl \Ball ( x ; \varepsilon_{2n + 1} ) \) and \( \set{x} \cup \bigcup_{n} \Cl \Ball ( x ; \varepsilon_{2n} ) \setminus \Ball ( x ; \varepsilon_{2n + 1} ) \) are as required, and are open and closed, respectively. 
\end{example}

\subsection{The function \( \Phi \)}
Let us list some easy facts about the map \( \Phi \) introduced in~\eqref{eq:Phi}:
\begin{itemize}[leftmargin=1pc]
\item
\( A \subseteq_\mu B \IMPLIES \Phi ( A ) \subseteq \Phi ( B ) \), and therefore \( A =_\mu B \implies \Phi ( A ) = \Phi ( B ) \). 
Thus the map \( \MALG_\mu \to \Pow ( X ) \), \( \eq{A} \mapsto \Phi ( A ) \), is well-defined;
\item
\( \Phi ( A \cap B ) = \Phi ( A ) \cap \Phi ( B ) \) hence \( \Phi ( A^\complement ) \subseteq ( \Phi ( A ) )^\complement \);
\item
\( \Phi ( A \cup B ) \supseteq \Phi ( A ) \cup \Phi ( B ) \); and more generally \( \Phi ( \bigcup_{i \in I} A_i ) \supseteq \bigcup_{i \in I} \Phi ( A_i ) \), provided \( \bigcup_{i \in I} A_i \in \MEAS_\mu \);
\item
\( \Phi ( U ) \supseteq U \), for \( U \) open, and \( \Phi ( C ) \subseteq C \), for \( C \) closed;
\item
 \( \Phi (C_1 \cup C_2 ) = \Phi ( C_1 ) \cup \Phi ( C_2 ) \), if \( C_1 , C_2 \) are disjoint closed sets.
\end{itemize}

\begin{definition}\label{def:DPP}
A Radon metric space \( ( X , d , \mu ) \) has the \markdef{Density Point Property} (DPP) if \( A \symdif \Phi ( A ) \in \NULL \) for each \( A \in \MEAS_\mu \).
\end{definition}

Thus in a DPP space almost every point is in \( \Phi ( A ) \cup \Phi ( A^\complement ) \), so \( \Exc ( A ) \), \( \Blur ( A ) \), and \( \Sharp ( A ) \) are null.
The Lebesgue density theorem states that \( \R^n \) with the Lebesgue measure \( \lambda^n \) and the \( \ell_p \)-norm has the DPP, and this result holds also for \( \pre{\omega }{2} \) with \( \measurecantor \) and the standard ultrametric.
In fact if \( \mu \) is a Borel measure on an ultrametric space \( ( X , d ) \), then \( ( X , d , \mu ) \) has the DPP~\cite[]{Miller:2008fk}. 
Not every Polish measure space is DPP~\cite[][Example 5.6]{Kaenmaki:2015sf}.

\subsection{The complexity of the density function}
\begin{proposition}
If \( ( X , d , \mu ) \) is separable and \( \mu \) finite, then the map 
\[
X \times \cointerval{ 0 }{ + \infty } \to \cointerval{ 0 }{ + \infty } , \qquad ( x , r ) \mapsto \mu ( \Ball ( x ; r ) ) 
\]
is Borel.
\end{proposition}

\begin{proof}
By multiplying by a suitable number, we may assume that \( \mu \) is a probability measure.
By~\cite[][Theorem 17.25]{Kechris:1995kc} with 
\[ 
A = \setofLR{ ( x , r , y ) \in X \times \cointerval{ 0 }{ + \infty } \times X}{ d ( x , y ) < r} 
\] 
then \( ( x , r ) \mapsto \mu ( A_{( x , r )} ) = \mu ( \Ball ( x ; r ) ) \) is Borel.
\end{proof}

Several results can be proved under the assumption that either the measure is continuous or else that the space is a closed subset of the Baire space.
The next definition aims at generalize both situations.

\begin{definition}\label{def:amenablespace}
A fully supported Radon metric space \( ( X , d , \mu ) \) is \markdef{amenable} if there are functions \( \varepsilon _n \colon X \to \cointerval{0}{+\infty} \) such that 
\begin{itemize}
\item
\( ( \varepsilon _n ( x ) )_n \) is a strong basis for density at \( x \), for all \( x \in X \),
\item
the map \( X \to \MALG \), \( x \mapsto \eq{\Ball ( x ; \varepsilon _n ( x ) ) } \) is continuous, for all \( n \in \omega \).
\end{itemize}
\end{definition}

\begin{examples}
\begin{enumerate-(a)}
\item
If \( \mu \) is continuous, then \( ( X , d , \mu ) \) is amenable.
In fact , let \( \varepsilon_n ( x ) \) be largest \( \leq 1 \) such that \( \mu ( \Ball ( x ; \varepsilon _n ( x ) ) ) \leq 1 / n \).
By Theorem~\ref{thm:densitybasis} \( ( \varepsilon _n ( x ) )_n \) is a strong basis for density; since the \( \varepsilon _n \) are continuous, by Lemma~\ref{lem:continuityofmeasure} \( x \mapsto \eq{\Ball ( x ; \varepsilon _n ( x ) ) } \) is continuous.
\item
If \( X \) is a closed subset of the Baire space and \( d \) is the induced metric, then \( ( X , d , \mu ) \) is amenable, as taking \( \varepsilon_n ( x ) = 2^{-n} \) the map \( x \mapsto \Ball ( x ; \varepsilon _n ( x ) ) \) is locally constant.
\end{enumerate-(a)}
\end{examples}

\begin{lemma}\label{lem:amenable}
Suppose \( ( X , d , \mu ) \) is amenable.
Then
\[
f_n \colon X \times \MALG \to [ 0 ; 1 ] , \quad ( x , \eq{A} ) \mapsto \frac{ \mu ( A \cap \Ball ( x ; \varepsilon _n ( x ) ) ) }{ \mu ( \Ball ( x ; \varepsilon _n ( x ) ) ) }
\] 
is continuous.
\end{lemma}

\begin{proof}
It is enough to show that \( ( x , \eq{A} ) \mapsto \mu ( A \cap \Ball ( x ; \varepsilon _n ( x ) ) ) \) is continuous.
This follows from the continuity of \( \hat{ \mu } \colon \MALG \to [ 0 ; + \infty ] \), and
\begin{multline*}
\cardLR{ \mu \left ( \Ball ( x ; \varepsilon _n ( x ) ) \cap A \right ) - \mu \left ( \Ball ( x' ; \varepsilon _n ( x' ) ) \cap A' \right ) } 
\\
\leq \mu \Bigl ( \bigl ( \Ball ( x ; \varepsilon _n ( x ) ) \cap A \bigr ) \symdif \bigl ( \Ball ( x' ; \varepsilon _n ( x' ) ) \cap A ' \bigr )\Bigr )
\\
 \leq \mu \left ( \Ball ( x ; \varepsilon _n ( x ) ) \symdif \Ball ( x' ; \varepsilon _n ( x' ) ) \right ) + \mu ( A \symdif A' ) . \qedhere
\end{multline*}
 \end{proof}

\begin{lemma}\label{lem:densityBaire2}
If \( ( X , d , \mu ) \) is amenable, then \( \density^+ \colon X \times \MALG \to [ 0 ; 1 ] \), \( ( x , \eq{A} ) \mapsto \density_A^+ ( x ) \) is in \( \mathscr{B}_2 \), and similarly for \( \density^- \) and \( \oscillation \).
\end{lemma}

\begin{proof}
Let \( f_n \) be as in Lemma~\ref{lem:amenable}.
Then \( g_n ( x , \eq{A} ) = \sup_{m \geq n} f_m ( x , \eq{A} ) \) is in \( \mathscr{B}_1 \), and therefore \( \limsup_n f_n ( x , \eq{A} ) = \lim_n g_n ( x , \eq{A} ) \) is in \( \mathscr{B}_2 \).
As \( ( \varepsilon _n ( x ) )_n \) is a strong basis for density in \( x \), it follows that \( \density_A^+ ( x ) = \lim_n g_n ( x , \eq{A} ) \).
The case of \( \density^- \) and of \( \oscillation \) is similar.
\end{proof}

By taking the preimage of \( \setLR{1} \) via the map \( \density^-_A \) we get

\begin{corollary}\label{cor:PhiPi03}
If \( ( X , d , \mu ) \) is amenable, then \( \Phi ( A ) \in \bPi^{0}_{3} \).
\end{corollary}

The complexity cannot be lowered in Corollary~\ref{cor:PhiPi03} when \( X \) is the real line or the Cantor space.
When \( A \subseteq \pre{\omega }{2} \) is nontrivial, if \( \Phi ( A ) \) has empty interior, then it is \( \bPi^{0}_{3} \)-complete~\cite[Theorem 1.3]{Andretta:2013uq}.
If \( K \subseteq \R \) is a sufficiently regular Cantor set of positive measure, then \( \Phi ( K ) \) is \( \bPi^{0}_{3} \)-complete~\cite{Carotenuto:2015kq}.

Notice that 
\[
x\in \Sharp ( A ) \IFF \oscillation_A ( x ) = 0 \wedge \EXISTS{q \in \Q_+} \forall ^\infty n \left ( q \leq f_n ( x , A ) \leq 1-q \right ) 
\]
where \( f_n \) is as in Lemma~\ref{lem:amenable}.
Thus in the hypotheses of Lemma~\ref{lem:densityBaire2}, 
\[
\Blur ( A ) \in \bSigma^{0}_{3}, \quad \Sharp ( A ) \in \bPi^{0}_{3}, \quad \Exc ( A ) \in \bSigma^{0}_{3} .
\]

\subsection{Solid sets}\label{sec:solidsets}
\begin{definition}
Let \( ( X , d , \mu ) \) be a Radon metric space.
A measurable \( A \subseteq X \) is 
\begin{itemize}
\item
\markdef{solid} iff \( \Blur ( A ) = \emptyset \),
\item
\markdef{quasi-dualistic} iff \( \Sharp (A ) = \emptyset \),
\item
\markdef{dualistic} iff it is quasi-dualistic and solid iff \( \Exc ( A ) = \emptyset \),
\item
\markdef{spongy} iff \( \Blur ( A ) \neq \emptyset = \Sharp (A ) \) iff it is quasi-dualistic but not solid.
\end{itemize}
\end{definition}

The collections of sets that are solid, dualistic, quasi-dualistic, or spongy are denoted by \( \Solid \), \( \Dual \), \( \qDual \), and \( \Spongy \).
Also
\[
\bDelta^{0}_{1} \subseteq \Dual = \Solid \cap \qDual .
\]
Therefore if the space \( X \) is disconnected, e.g. \( X = \pre{\omega }{2} \), there are nontrivial dualistic sets so adopting the notation of~\eqref{eq:delta(X)}, we conclude that \( \boldsymbol{ \delta } ( X ) = 0 \).
In the Cantor space there are examples of dualistic sets that are not \( =_\mu \) to any clopen set, see~\cite[][Section 3.4]{Andretta:2013uq}.

The situation for \( \R \) is completely different: V.~Kolyada~\cite{Kolyada:1983fk} showed that \( 0 < \boldsymbol{ \delta } ( \R ) < 1 / 2 \), thus, in particular, there are no nontrivial dualistic subsets of \( \R \).
The bounds for \( \boldsymbol{ \delta } ( \R ) \) were successively improved in~\cite{Szenes:2011fk,Csornyei:2008uq}, and in~\cite{Kurka:2011kx} it is shown that \( \boldsymbol{ \delta } ( \R ) \approx 0. 268486 \dots \) is the unique real root of \( 8 x^3 + 8 x^2 + x - 1 \).
A curious consequence is that for each \( \varepsilon > 0 \) there are nontrivial sets \( A \subset \R \) such that \( \ran( \density_A ) \cap ( \boldsymbol{ \delta } ( \R ) + \varepsilon ; 1 - \boldsymbol{ \delta } ( \R ) - \varepsilon ) = \emptyset \); in other words, for any real \( x \) either \( \density _A ( x ) \in [ 0 ; \boldsymbol{ \delta } ( \R ) + \varepsilon ] \cup [ 1 - \boldsymbol{ \delta } ( \R ) - \varepsilon ; 1 ] \) or \( \density^+_A ( x ) \geq 1 - \boldsymbol{ \delta } ( \R ) - \varepsilon \) or else \( \density^-_A ( x ) \leq \boldsymbol{ \delta } ( \R ) + \varepsilon \).
In particular, there is a set \( A \) that does not have points of density \( 1/2 \), in contrast with our intuition that a measurable subset of \( \R \) should have a ``boundary'' like an interval.
We will show in Theorem~\ref{thm:solid} that this intuition is correct when \emph{solid} sets are considered.

Spongy subsets of \( \pre{\omega}{2} \) (or more generally, of closed subsets of \( \pre{\omega}{\omega} \)) are easy to construct, see~\cite[Example 3.8 in][]{Andretta:2013uq}.
The existence of spongy subsets of connected spaces is more problematic.
Theorem~\ref{thm:spongy} shows that there exist a spongy subset \( S \) of \( [ 0 ; 1 ] \), and for such \( S \) we have \( \boldsymbol{ \delta }_S \geq 1 / 3 \). 

The families of sets \( \Solid \), \( \Dual \), \( \qDual \), and \( \Spongy \) are invariant under \( =_\mu \), so they can be defined on the measure algebra as well, that is to say: we can define
\[
 \widehat{\Solid} = \setof{\eq{A} \in \MALG }{ A \in \Solid } ,
\] 
and similarly for \( \widehat{\Dual} \), \( \widehat{\qDual} \) and \( \widehat{\Spongy} \).

\begin{proposition}\label{prop:solidsetBaireclassDensity}
Let \( ( X , d , \mu ) \) be amenable and suppose that \( A \) is solid.
Then \( \density_A \colon X \to [ 0 ; 1 ] \) is in \( \mathscr{B}_1 \).
\end{proposition}

\begin{proof}
Notice that 
\[ 
\density_A ( x ) > a \IFF \EXISTS{ q \in \Q_+ }\FORALLS{\infty}{ n } \Bigl ( \frac{ \mu ( A \cap \Ball ( x ; \varepsilon _n ( x ) ) ) }{ \mu ( \Ball ( x ; \varepsilon _n ( x ) ) ) } \geq a + q \Bigr )
\] 
and apply Lemma~\ref{lem:amenable}.
Similarly for \( \density_A ( x ) < b \).
\end{proof}

By the Baire category theorem we get:

\begin{corollary}\label{cor:solidsetBaireclassDensity}
Let \( ( X , d , \mu ) \) be amenable and completely metrizable.
If there are \( 0 \leq r < s \leq 1 \) such that \( \setofLR{x}{ \density_A ( x ) \leq r } \) and \( \setofLR{x}{ \density_A ( x ) \geq s } \) are dense in some nonempty open set, then \( A \notin \Solid \).
\end{corollary}

\begin{proposition}\label{prop:solidsetGdelta}
Let \( ( X , d , \mu ) \) be amenable and suppose that \( A \) is solid.
Then 
\begin{enumerate-(a)}
\item
\( \Phi ( A ) , \Phi ( A^\complement ) \in \bPi^{0}_{2} \), 
\item
\( \Exc ( A ) = \Sharp ( A ) \in \bSigma^{0}_{2} \),
\item
if \( 1 \) is an isolated value of \( \density_A \), that is to say \( \ran \density_A \subseteq [ 0 ; r ] \cup \setLR{1} \) for some \( r < 1 \), then \( \Phi ( A ) \in \bDelta^{0}_{2} \).
\end{enumerate-(a)}
In particular, if \( A \) is dualistic, then \( \Phi ( A ) \in \bDelta^{0}_{2} \).
\end{proposition}

\begin{proof}
By Proposition~\ref{prop:solidsetBaireclassDensity} \( \density_A \) is Baire class \( 1 \), and since \( \Phi ( A ) \) is the preimage of the singleton \( \setLR{1} \), then it is a \( \Gdelta \).
If \( 1 \) is an isolated value of the density function, then \( \Phi ( A ) = \density_A^{-1} \ocinterval{ r }{ 1} \) is also \( \Fsigma \), thus it is a \( \bDelta^{0}_{2} \).
\end{proof}

The (possibly partial) function \( \density_A \colon \pre{\omega }{2} \to [ 0 ;1 ] \) has graph \( \bPi^{0}_{3} \), since
\[
 \density_ A ( x ) = r \IFF \FORALL{ \varepsilon } \EXISTS{ n } \FORALL{ k > n } \card{ \mu ( \LOC{A}{ x \restriction k} ) - r } \leq \varepsilon
\]
and has domain \( \pre{\omega }{2} \setminus \Blur A \).
So perhaps it is more natural to look at its extension \( \density_A^* \colon \pre{\omega }{2} \to [ 0 ;1 ] \cup \set{*} \), where \( * \) means undefined. 
It is an isolated point in \( [ 0 ;1 ] \cup \set{*} \).

\begin{proposition}
 \( \mathrm{graph} ( \density_A^* ) \) is a boolean combination of \( \bPi^{0}_{3} \) sets.
\end{proposition}

\begin{proof}
 \( ( z , r ) \in \mathrm{graph} ( \density_A^* ) \IFF \left ( \oscillation_A ( z ) = 0 \wedge \density_A ( z ) = r \right ) \vee \left ( \oscillation_A ( z ) > 0 \wedge r = * \right ) \)
 \end{proof}

By~\cite[Theorem 1.7][]{Andretta:2013uq}, working in the Cantor space we have that \( \setof{\eq{A} \in \MALG}{ \Phi ( A ) \text{ is \( \bPi^{0}_{3} \)-complete}} \) is comeager. 

\begin{corollary}
\( \setofLR{\eq{A} \in \MALG ( \pre{\omega}{2} ) }{ \Blur ( A ) \neq \emptyset } = \MALG \setminus \widehat{ \Solid } \) is comeager.
\end{corollary}

We will prove later (Theorem~\ref{thm:blurrypointsSigma03}) that the set of blurry points can be \( \bSigma^{0}_{3} \)-complete, and in fact this is the case on a comeager set in the measure algebra.

\section{Compact sets in the measure algebra}\label{sec:compactsetsinMALG}
Suppose \( ( X , d , \mu ) \) is a separable Radon metric space and \( A \in \MEAS_\mu \).
The \markdef{\( \mu \)-interior} of \( A \) is
\[
\Int_\mu ( A ) = \bigcup \setof{ U \in \bSigma^{0}_{1} ( X ) }{ U \subseteq_ \mu A } ,
\]
 the \markdef{\( \mu \)-closure} of \( A \) is
\[
\begin{split}
\Cl_\mu ( A ) & = \bigcap \setof{ C \in \bPi^{0}_{1} ( X ) }{ A \subseteq_ \mu C }
\\
 & = X \setminus \bigcup \setof{ U \in \bSigma^{0}_{1} ( X ) }{ A \cap U \in \NULL_\mu } ,
\end{split}
\]
and the \markdef{\( \mu \)-frontier} of \( A \) is
\[
\begin{split}
\Fr_\mu ( A ) & = \Cl_\mu ( A ) \setminus \Int_\mu ( A )
\\
 & = \setof{ x \in X }{ \FORALL{ U \in \bSigma^{0}_{1} ( X )} ( x \in U \implies \mu ( A \cap U ) , \mu ( U \setminus A ) > 0 ) } .
\end{split}
\]
Thus \( \Int_\mu ( A ) \) is open, and \( \Cl_\mu ( A ) \) and \( \Fr_\mu ( A ) \) are closed, and they behave like the usual topological operators, i.e. \( ( \Cl_\mu A )^ \complement = \Int_\mu ( A^ \complement ) \) and \( ( \Int_\mu A )^ \complement = \Cl_\mu ( A^ \complement ) \).
(In~\cite{Andretta:2013uq} the sets \( \Cl_\mu ( A ) \) and \( \Int_\mu ( A ) \) were called the outer and inner supports of \( A \), and were denoted by \( \supt^+ ( A ) \) and \( \supt^- ( A ) \), respectively.)
The \markdef{support of \( \mu \)} is \( \supt ( \mu ) = \Cl_\mu ( X ) \), and therefore \( \mu \) is fully supported if and only \( \supt ( \mu ) = X \).

Clearly \( \Int_\mu ( A ) \subseteq \Phi ( A ) \), and the inclusion can be proper; for example if \( ( X , d , \mu ) \) is fully supported, locally finite and DPP take \( A \) to be closed of positive measure with empty interior.
We start with a trivial observation, that will turn out to be useful in the proof of Theorem~\ref{thm:solid}.

\begin{lemma}\label{lem:useless}
Let \( ( X , d , \mu ) \) be a fully supported locally finite and DPP, and let \( A \in \MEAS_\mu \).
Suppose \( \Fr_\mu A \) has nonempty interior.
Then \( \Phi ( A ) \) and \( \Phi ( A^\complement ) \) are dense in \( \Int ( \Fr_\mu A ) \), so if \( ( X , d , \mu ) \) is amenable and completely metrizable, then \( A \) is not solid.
\end{lemma}

\begin{proof}
Let \( U \subseteq \Fr_\mu A \) be nonempty and open in \( X \): as \( U \) is disjoint from \( \Int_\mu ( A ) \cup \Int_\mu ( A^\complement ) \), then \( \mu ( A \cap U ) ,  \mu ( U \setminus A ) > 0 \), and therefore by DPP \( U \) intersects both \( \Phi ( A ) \) and \( \Phi ( A^\complement ) \).
That \( A \) is not solid follows from Corollary~\ref{cor:solidsetBaireclassDensity}.
\end{proof}

By separability \( \Cl_\mu A \) is the smallest closed set \( C \) such that \( A \subseteq_\mu C \), and therefore \( \Cl_\mu ( \Cl_\mu A ) = \Cl_\mu A \) by transitivity of \( \subseteq_\mu \).
If \( C \) is closed, then \( C =_\mu \Cl_\mu ( C ) \), so 
\[
\FORALL{ C , D \in \bPi^{0}_{1}} \left ( \Cl_\mu ( C ) =_\mu \Cl_\mu ( D ) \implies C =_\mu D \right ) 
\]
hence, since the operator \( \Cl_\mu \) is \( =_\mu \)-invariant,
\[
\FORALL{ A , B \in \MEAS_\mu }\left ( \Cl_\mu ( A ) =_\mu \Cl_\mu ( B ) \implies \Cl_\mu ( A ) = \Cl_\mu ( B ) \right ) .
\]
Therefore \( \Cl_\mu \) is a selector for the family \( \mathscr{F} \) defined below.

\begin{definition}
If \( X \) is a topological space with a Borel measure \( \mu \), let
\begin{align*}
\mathscr{F} ( X , \mu ) & = \setof{ \eq{C} \in \MALG ( X , \mu )}{ C \text{ is closed}}
\\
\mathscr{K} ( X , \mu ) & = \setof{ \eq{K} \in \MALG ( X , \mu )}{ K \text{ is compact}} .
\end{align*}
As usual the reference to \( X \) and/or \( \mu \) will be dropped whenever possible.
\end{definition}

\begin{lemma}\label{lem:ClPhi=supt}
If \( ( X , d , \mu ) \) is a separable Radon metric space, \( A \) is measurable, and \( A \subseteq_{\mu } \Phi ( A ) \), then \( \Cl \Phi ( A ) = \Cl_\mu A \).
\end{lemma}

\begin{proof}
First, \( \Phi (A ) \subseteq \Cl_\mu A \), since any point in \( ( \Cl_\mu A)^{ \complement } \) is contained in some open \( U \) with \( \mu ( A \cap U ) = \emptyset \).
Consequently, \( \Cl \Phi ( A ) \subseteq \Cl_\mu A \).

Conversely, given \( x \in \Cl_\mu A \) and any open neighborhood \( U \) of \( x \), one has \( \mu ( U \cap A ) > 0 \), thus \( \mu ( U \cap \Phi ( A ) ) > 0 \), whence \( U \cap \Phi ( A ) \neq \emptyset \).
It follows \( x \in \Cl \Phi ( A ) \).
\end{proof}

Note that when \( X \) is DPP then the assumption \( A \subseteq_{\mu } \Phi ( A ) \) is automatically satisfied.
If \( X \) is a closed subset of \( \pre{\omega}{\omega} \), that is \( X = \body{T} \) for some pruned tree \( T \) on \( \omega \), then \( X \) is DPP and \( \Cl_\mu A = \body{ \densitytree ( A ) } \) where 
\begin{equation}\label{eq:densitytree}
 \densitytree ( A ) = \setofLR{ t \in T }{ \mu ( A \cap \Nbhd_t ) > 0 }
\end{equation}
is the tree of those basic open sets in which \( A \) is non-null~\cite[][Definition 3.3]{Andretta:2013uq}.
Therefore

\begin{corollary}\label{cor:densitytree3}
\( \densitytree \body{ \densitytree ( A ) } = \densitytree ( A ) \), i.e. \( \densitytree ( \Cl \Phi ( A ) ) = \densitytree ( A ) \).
\end{corollary}

A metric space is \markdef{Heine-Borel} if every closed ball is compact.
It is easy to see that any such space is \( \Ksigma \) and Polish.

\begin{theorem}\label{thm:setofcompactsinMALG}
Let \( ( X , d , \mu ) \) be a Heine-Borel space such that every compact set has finite measure.
Then \( \mathscr{K} ( X , \mu ) \) and \( \mathscr{F} ( X , \mu ) \) are \( \bPi^{0}_{3} \) in \( \MALG ( X , \mu ) \). 
\end{theorem}

\begin{proof}
Fix \( \bar{x} \in X \), and let \( B_n = \setofLR{y \in X} {d ( \bar{x}, y ) \leq n + 1 } \) be the closed ball of center \( \bar{x} \) and radius \( n > 0 \).

First we prove that \( \mathscr{K} ( X , \mu ) \) is \( \bPi^{0}_{3} \).
Note that
\[
\eq{A} \in \mathscr{K} \IFF \EXISTS{n} ( A \subseteq_\mu B_n ) \AND \mu ( A ) \geq \mu ( \Cl_\mu A )
\]
and the right hand side is equivalent to
\[
 \underbracket[0.5pt]{ \exists n ( A \subseteq_\mu B_n ) }_{\upvarphi ( A )} \wedge \forall q \in \Q_+ \bigl ( \underbracket[0.5pt]{\exists n ( A \subseteq_\mu B_n ) \wedge \mu ( \Cl_\mu A ) > q}_{\uppsi ( A , q ) } \IMPLIES \underbracket[0.5pt]{\mu ( A ) \geq q }_{\upchi (A , q ) } \bigr ) .
\]
The formulæ \( \upvarphi ( A ) \) and \( \upchi (A , q ) \) are easily seen to be \( \mathsf{\Sigma ^0_2} \) and \( \mathsf{\Pi^0_1} \) respectively, so it suffices to show that \( \uppsi ( A , q ) \) is \( \mathsf{\Sigma ^0_3} \).
Let \( ( U_n )_n \) be a countable basis for \( X \).
\[
\begin{split}
\uppsi ( A , q ) & \iff \exists n ( A \subseteq_\mu B_n ) \wedge \exists \varepsilon \in \Q_+ \forall n_0 ,\dots , n_h \in \omega 
\\
& \qquad\qquad [ \Cl_\mu A \subseteq U_{n_0} \cup \dots \cup U_{n_h} \implies q + \varepsilon < \mu ( U_{n_0} \cup \dots \cup U_{n_h} ) ]
 \\
 & \iff \exists n ( A \subseteq_\mu B_n ) \wedge \exists \varepsilon \in \Q_+ \forall n_0 ,\dots , n_h \in \omega 
 \\
 & \qquad\qquad [ \exists m_0 , \dots , m_k ( B_n \setminus ( U_{n_0} \cup \dots \cup U_{n_h} ) \subseteq U_{m_0} \cup \dots \cup U_{m_k} \wedge {}
 \\
 & \qquad\qquad\quad \mu ( A \cap ( U_{m_0} \cup \dots \cup U_{m_k} ) ) = 0 ) \implies q + \varepsilon < \mu ( U_{n_0} \cup \dots \cup U_{n_h} ) ] .
\end{split}
\]
The premise of the implication is \( \mathsf{\Sigma ^0_2} \), so \( \uppsi (A , q ) \) is \( \mathsf{\Sigma^0_3} \), as required.

We now prove that \( \mathscr{F} ( X , \mu ) \) is \( \bPi^{0}_{3} \).
Notice that it is enough to show that 
\[ 
\eq{A} \in \mathscr{F} \IFF \FORALL{ n \in \omega} \bigl ( \eq{A} \cap \eq{B_n } \in \mathscr{K} \bigr )
\] 
and use the fact that \( \MALG^2 \to \MALG \), \( ( \eq{X} , \eq{Y} ) \mapsto \eq{ X \cap Y } \), is continuous.
To establish the equivalence, suppose that \( A =_{\mu } F \) for some closed \( F \).
Then \( \eq{A } \cap \eq{ B_n } = \eq{ F \cap B_n } \in \mathscr{K} \).
Conversely, let \( C_n \) be compact such that \( C_n =_{\mu } A \cap B_n \); if \( F = \bigcup_{ n \in \omega } C_n \), then \( A =_{\mu } F \), concluding the proof.
\end{proof}

\begin{lemma} \label{lem:suptisBaire1}
Let \( X \) be compact, metric.
Then the function \( f \colon \MALG ( X ) \to \KK ( X ) \) defined by \( f ( \eq{A} ) = \Cl_\mu A \) is in \( \mathscr{B}_1 \).
\end{lemma}

\begin{proof}
Let \( ( U_n )_n \) be a basis of \( X \) and fix an open subset \( U \subseteq X \).
If \( A \subseteq X \) is measurable, then
\[
\Cl_\mu A \subseteq U \IFF \EXISTS {n_0 , \ldots , n_h} \bigl ( U^{ \complement } \subseteq U_{ n_0 } \cup \ldots \cup U_{n_h} \AND \mu ( A \cap ( U_{n_0} \cup \ldots \cup U_{n_h} ) ) = 0 \bigr )
\]
and this condition is \( \bSigma^0_2 \) on \( \eq{A} \).
Moreover,
\[
\Cl_\mu A\cap U \neq \emptyset \iff \mu ( A \cap U ) > 0 ,
\]
which is an open condition on \( \eq{A} \).
So the preimage under \( f \) of any open subset of \( \KK ( X ) \) is \( \bSigma^0_2 \).
\end{proof}

\begin{lemma} \label{lem:measurefunctionisBaire1}
Let \( X \) be a separable metrizable Radon space whose measure is outer regular.
Then the function \( g \colon \KK ( X ) \to [ 0 ; + \infty ] \) defined by \( g ( K ) = \mu ( K ) \) is in \( \mathscr{B}_1 \).
\end{lemma}

\begin{proof}
Let \( ( U_n )_{ n < \omega } \) be a basis of \( X \).
Fix \( a \geq 0 \); then, for \( K \in \KK ( X ) \), one has
\begin{multline*}
a < \mu ( K ) \IFF {}
\\
\EXISTS{ \varepsilon > 0 } \FORALL {n_0 , \ldots , n_h}\left ( K \subseteq U_{n_0} \cup \ldots \cup U_{n_h} \implies a + \varepsilon < \mu ( U_{n_0} \cup \ldots \cup U_{n_h} ) \right ) .
\end{multline*}
This condition is \( \bSigma^0_2 \) on \( K \).
For \( b > 0 \), one has
\[
\mu ( K ) < b \IFF \EXISTS{n_0 , \ldots , n_h }\left ( \mu ( U_{n_0} \cup \ldots \cup U_{n_h} ) < b \AND K \subseteq U_{n_0} \cup \ldots \cup U_{n_h} \right ) ,
\]
an open condition on \( K \).
So, the preimage under \( g \) of an open subset of \( [ 0 ; +\infty ] \) is \( \bSigma^0_2 \).
\end{proof}

\begin{definition}\label{def:thickset}
Suppose \( \mu \) is a Borel measure on a topological space \( X \), \( U \) is open and nonempty, and \( A \) is measurable.
We say that \( A \) is 
\begin{itemize}
\item
\markdef{thick in} \( U \) if \( \mu ( A \cap V ) > 0 \) for all open nonempty sets \( V \subseteq U \), 
\item
\markdef{co-thick in} \( U \) if \( A^\complement \) is thick in \( U \).
\end{itemize}
If \( U =_\mu X \) we simply say that \( A \) is thick/co-thick.
\end{definition}
Note that \( A \) is thick in \( U \) if and only if \( \Cl_\mu ( A ) \supseteq U \).
In a DPP space, \( A \) is thick in an open set \( U \) iff \( \Phi ( A ) \) is dense in \( U \).

\begin{lemma}\label{lem:thick}
Let \( ( X , d , \mu ) \) be a separable Radon metric space, with \( \mu \) nonsingular.
If \( 0 < \mu ( A ) < \infty \) then for all \( \varepsilon > 0 \) there is a compact set \( K \subseteq A \) with empty interior and such that \( \mu ( A ) - \varepsilon < \mu ( K ) \).
\end{lemma}

\begin{proof}
Fix \( A \) and \( \varepsilon \) as above.
Without loss of generality we may assume that \( \varepsilon < \mu ( A ) \).
Let \( F \subseteq A \) be compact and such that \( \mu ( F ) > \mu ( A ) - \varepsilon / 2 \).
Let \( \setof{ q_n }{ n \in \omega } \) be dense in \( X \) and by our assumption on \( \mu \) choose \( r_n > 0 \) such that \( \mu ( \Ball ( q_n ; r_n ) ) \leq \varepsilon 2^{- ( n + 2 ) } \), so that \( U = \bigcup_{ n \in \omega } \Ball ( q_n ; r_n ) \) has measure \( \leq \varepsilon / 2 \).
Then \( K = F \setminus U \subseteq A \) is compact with empty interior and \( \mu ( K ) \geq \mu ( F ) - \varepsilon / 2 > \mu ( A ) - \varepsilon \).
\end{proof}

\begin{theorem}\label{thm:thick&cothick}
Suppose \( ( X , d , \mu ) \) is separable, fully supported Radon metric space, with \( \mu \) nonsingular.
Then there is a \( \Ksigma \) set which is thick and co-thick.
\end{theorem}

\begin{proof}
As \( X \) is second countable and \( \mu \) is locally finite, fix a base \( \setof{U_n}{ n \in \omega } \) for \( X \) such that \( 0 < \mu ( U_n ) < \infty \) for all \( n \).
We inductively construct compact sets \( C_n \) for \( n \in \omega \) with empty interior such that \( \FORALL{i \leq n} ( \mu ( U_i \cap \bigcup_{j \leq n} C_j ) > 0 ) \). 
Let \( \tilde{n} \geq n \) be least such that \( U_{\tilde{n}} \subseteq U_n \setminus \bigcup_{j < n } C_j \). 
By Lemma~\ref{lem:thick} choose \( C_n \subseteq U_{\tilde{n}} \) compact with empty interior and such that \( 0 < \mu ( C_n ) \leq 2^{-n - 2 } \min \setof{\mu ( U_{\tilde{m} } ) }{ m \leq n } \).

Clearly \( F = \bigcup_{n} C_n \) is \( \Ksigma \) and thick.
In order to prove it is co-thick, it is enough to show that \( \mu ( U_n \setminus F ) > 0 \) for each \( n \).
Fix \( n \in \omega \): as \( U_{\tilde{n}} \subseteq U_n \), it is enough to show that \( \mu ( U_{\tilde{n}} \cap F ) < \mu ( U_{\tilde{n}} ) \).
By construction if \( C_m \cap U_{\tilde{n}} \neq \emptyset \), then \( m \geq n \), and hence \( \mu ( C_m ) \leq 2^{- m - 2 } \mu ( U_{\tilde{n}} ) \) and therefore \( \mu ( F \cap U_{\tilde{n}} ) \leq \mu ( U_{\tilde{n}} ) / 2 \).
\end{proof}

Theorem~\ref{thm:thick&cothick} emphasize a difference between measure and category, since in a topological space any nonmeager subset with the Baire property is comeager in some open set.
 
Working in \( \pre{\omega}{2} \), the function
\[
 \hat{ \Phi } \colon \MALG \to \bPi^0_3, \quad \hat{ \Phi } ( \eq{A} ) = \Phi ( A ) ,
\]
is Borel-in-the-codes~\cite[][Proposition 3.1]{Andretta:2013uq}, while \( \hat{ \mu } \colon \MALG \to [ 0 ; 1 ] \), \( \hat{ \mu } \eq{A} = \mu ( A ) \), is continuous.
The \( \Ksigma \) set \( F \) constructed in Theorem~\ref{thm:thick&cothick} can be of arbitrarily small measure, and hence \( A \cup F \) can be arbitrarily close to any measurable set \( A \).
Therefore the map \( \MALG \to \PrTr_2 \), \( \eq{A} \mapsto \densitytree ( A ) \), where \( \PrTr_2 \) is the Polish space of all pruned trees on \( \set{0 , 1} \), is not continuous, but it is in \( \mathscr{B}_1 \).
To see this apply Lemma~\ref{lem:suptisBaire1} together with the fact that \( \body{\densitytree ( A ) } = \Cl_\mu A \) and that the map \( \KK ( \pre{\omega}{2} ) \to \PrTr_2 \), \( K \mapsto T_K \), is continuous.
If \( A \) is dualistic, then \( \Phi ( A ) \) and \( \Phi ( A^\complement ) \) are \( \bDelta^{0}_{2} \) by Proposition~\ref{prop:solidsetGdelta}. 
In~\cite[][Section 3.4]{Andretta:2013uq} examples of dualistic, solid, spongy sets are constructed.

For any Polish measure space \( ( X , d , \mu ) \) the set \( \mathscr{ K } ( X ) \) is dense by tightness of \( \mu \), and it is meager by~\cite[Theorem 1.6]{Andretta:2013uq}.
(The proof in that paper is stated for \( \pre{\omega }{2} \), but it works in any Polish measure space.)

In a DPP space, if \( C \) is closed and thick in some nonempty open set \( U \), then \( \Phi ( C ) \) is dense in \( U \), and therefore \( C \supseteq U \).
Therefore 

\begin{lemma}\label{lem:thickcothicknotcompact}
In a DPP space \( ( X , d , \mu ) \), if \( A \) is thick and co-thick in some nonempty open set \( U \), then \( \eq{A} \notin \mathscr{ F } ( X , \mu ) \). 
\end{lemma}

\begin{theorem}\label{thm:KisPi03completeCantor}
 \( \mathscr{K} ( \pre{\omega}{2} , \measurecantor ) \) is \( \bPi^{0}_{3} \)-complete in \( \MALG \). 
\end{theorem}

\begin{proof}
By Proposition~\ref{thm:setofcompactsinMALG} \( \mathscr{K} \) is \( \bPi^{0}_{3} \), so it is enough to prove \( \bPi^{0}_{3} \)-hardness.
We define a continuous \( \hat{f} \colon \pre{ \omega \times \omega }{ 2 } \to \MALG \) witnessing \( \boldsymbol{P}_3 \leqW \mathscr{K} \), where
\[
\boldsymbol{P}_3 = \setof{ z \in \pre{ \omega \times \omega }{ 2 } }{ \FORALL{n} \EXISTS{m} \FORALL{k \geq m} z ( n , k ) = 0 }
\]
is \( \boldsymbol{\Pi}^0_3 \)-complete~\cite[p.~179]{Kechris:1995kc}.
More precisely, set \( \hat{f} ( z ) = \eq{ f ( z ) } \) where 
\[
f ( z ) = \bigcup_{n} \varphi ( z \restriction n \times n )
\]
for some suitable function \( \varphi \colon \pre{ < \omega \times \omega }{2} \to \KK ( \pre{\omega }{2} ) \) such that for all \( a \in \pre{ < \omega \times \omega }{2} \)
\begin{subequations}
\begin{gather}
 \Int \varphi ( a ) = \emptyset , \label{eq:thm:KisPi03complete-1}
\\
b \subseteq a \IMPLIES \varphi ( b ) \subseteq \varphi ( a ) , \label{eq:thm:KisPi03complete-2}
\\
a \in \Pre{ ( n + 1 ) \times ( n + 1 ) }{2}\IMPLIES \measurecantor \left ( \varphi ( a ) \setminus \varphi ( a \restriction n \times n ) \right ) \leq 2^{ - ( n + 2 )} . \label{eq:thm:KisPi03complete-3}
\end{gather}
\end{subequations}
For \( a , b \in \pre{ < \omega \times \omega }{2} \) let \( \delta ( a , b ) \) be the largest \( n \) such that \( a \restriction n \times n = b \restriction n \times n \).
Equation~\eqref{eq:thm:KisPi03complete-3} implies that if \( a \in \pre{ n \times n }{2} \) then \( a \subset a' \implies \measurecantor \left ( \varphi ( a' ) \setminus \varphi ( a ) \right ) < 2^{ - ( n + 1 ) } \); thus if \( a , b \in \pre{ < \omega \times \omega }{2} \) are such that \( \delta ( a , b ) = n \), then \( \varphi ( a ) \symdif \varphi ( b ) \subseteq ( \varphi ( a ) \setminus \varphi ( a \restriction n \times n ) ) \cup ( \varphi ( b ) \setminus \varphi ( b \restriction n \times n ) ) \) and hence \( \measurecantor \left ( \varphi ( a ) \symdif \varphi ( b ) \right ) < 2^{ - n } \).
Therefore if \( z , w \in \pre{ \omega \times \omega }{2} \) and \( n \) is largest such that \( z \restriction n \times n = w \restriction n \times n \), then \( \measurecantor \left ( f ( z ) \symdif f ( w ) \right ) \leq 2^{ - n } \), and therefore \( \hat{f} \) is continuous.
We arrange that 
\begin{subequations}
\begin{align}
z \in \boldsymbol{P}_3 & \IMPLIES f ( z ) \in \KK ( \pre{\omega }{2} ) \label{eq:thm:KisPi03complete-5}
\\
z \notin \boldsymbol{P}_3 & \IMPLIES f ( z ) \in \Ksigma ( \pre{\omega }{2} ) \text{ is thick and co-thick in some } \Nbhd_{0^{( j )} \conc 1 } . \label{eq:thm:KisPi03complete-6}
\end{align}
\end{subequations}
By Lemma~\ref{lem:thickcothicknotcompact}, equation~\eqref{eq:thm:KisPi03complete-6} guarantees that if \( z \notin \boldsymbol{P}_3 \) then \( \hat{f} ( z ) \notin \mathscr{K} \), and therefore \( \hat{f} \) witnesses that \( \boldsymbol{P}_3 \leqW \mathscr{K} \).

Here are the details.
Fix \( ( s^j_m )_m \) an enumeration without repetitions of the nodes extending \( 0^{( j )} \conc 1 \), and such that longer nodes are enumerated after shorter ones, that is: \( \lh ( s^j_n ) < \lh ( s^j_m ) \implies n < m \).
\begin{itemize}[leftmargin=1pc]
\item
Set \( \varphi ( \emptyset ) = \set{ 0^{ ( \omega ) } } \).
Then~\eqref{eq:thm:KisPi03complete-1} holds, and~\eqref{eq:thm:KisPi03complete-2} and~\eqref{eq:thm:KisPi03complete-3} do not apply.
\item
Suppose \( a \in \Pre{ n + 1 }{2} \) and that \( \varphi ( a \restriction n \times n ) \) satisfies~\eqref{eq:thm:KisPi03complete-1}--\eqref{eq:thm:KisPi03complete-3}, and let's construct \( \varphi ( a ) \).
If \( a ( j , n ) = 0 \) for all \( j \leq n \), then set \( \varphi ( a ) = \varphi ( a \restriction n \times n ) \) so that~\eqref{eq:thm:KisPi03complete-1}--\eqref{eq:thm:KisPi03complete-3} are still true.
Otherwise, let \( j \leq n \) be least such that \( a ( j , n ) = 1 \).
Then by~\eqref{eq:thm:KisPi03complete-1} for \( \varphi ( a \restriction n \times n ) \), we can define \( k \) to be the least such that \( \mu \bigl (\Nbhd_{ s^j_k } \cap \varphi ( a \restriction n \times n ) \bigr ) = 0 \), and let \( K \subseteq \Nbhd_{ s^j_k } \) be compact with empty interior and such that 
\begin{equation}\label{eq:thm:KisPi03complete-7}
0 < \measurecantor ( K ) \leq \measurecantor ( \Nbhd_{ s^j_k } ) / 2^{ - ( n + 2 ) } .
\end{equation}
Then \( \varphi ( a ) = \varphi ( a \restriction n \times n ) \cup K \) satisfies~\eqref{eq:thm:KisPi03complete-1}--\eqref{eq:thm:KisPi03complete-3}.
\end{itemize}

The proof is complete once we check that~\eqref{eq:thm:KisPi03complete-5} and~\eqref{eq:thm:KisPi03complete-6} hold.
Suppose first \( z \in \boldsymbol{P}_3 \).
Then for each \( j \) there is \( N_j \in \omega \) such that \( z ( j , n ) = 1 \implies n < N_j \), and hence \( \Nbhd_{ 0^{( j )} \conc 1} \cap f ( z ) = \Nbhd_{ 0^{( j )} \conc 1} \cap \varphi ( z \restriction N_j \times N_j ) \) is compact, so \( f ( z ) \) is compact.
Suppose now \( z \notin \boldsymbol{P}_3 \), and let \( j \) be least such that \( \setof{n}{ z ( j , n ) = 1 } \) is infinite.
Then \( f ( z ) \) is thick in \( \Nbhd_{ 0^{( j )} \conc 1 } \): fix \( k \in \omega \), then for \( N \) such that \( \setofLR{ M < N }{ z ( j , M ) = 1 } \) has size at least \( k + 1 \), one has that \( \measurecantor ( \varphi ( z \restriction N \times N ) \cap \Nbhd_{ s^j_k } ) > 0 \).
Moreover \( f ( z ) \) is co-thick in \( \Nbhd_{ 0^{( j )} \conc 1 } \).
To see this fix \( k \in \omega \) and let \( N \) be such that \( \setofLR{ M < N }{ z ( j , M ) = 1 } \) has size \( k \), and let \( H = \varphi ( z \restriction N \times N ) \cap \Nbhd_{ s^j_k } \).
Since \( H \) is closed with empty interior, let \( k' \geq k \) be least with \( s^j_k \subseteq s^j_{k'} \) and \( H \cap \Nbhd_{ s^j_{k'} } = \emptyset \).
Then \( \measurecantor ( f ( z ) \cap \Nbhd_{ s^j_{k'} } ) < \measurecantor ( \Nbhd_{ s^j_{k'} } ) \) by~\eqref{eq:thm:KisPi03complete-7}.
\end{proof}

\begin{corollary}\label{cor:KisPi03completeCantor}
Let \( ( X , d , \mu ) \) be a Polish measure space such that \( \mu \) is nonsingular.
If there is a \( Y \subseteq X \) such that \( 0 < \mu ( Y ) < \infty \), then \( \mathscr{K} ( X , \mu ) \) and \( \mathscr{F} ( X , \mu ) \) are \( \bPi^{0}_{3} \)-hard.
\end{corollary}

\begin{proof}
We may assume that \( Y \) is \( \Gdelta \).
Choose \( r > 0 \) small enough so that Theorem~\ref{thm:embeddingCantorinPolish} can be applied, so that there is an injective continuous \( H \colon \pre{\omega }{2} \to Y \) such that \( r \measurecantor ( A ) = \mu ( H [ A ] ) \) for all measurable \( A \subseteq \pre{\omega }{2} \).
The map \( H \) induces an embedding between the measure algebras 
\[ 
\hat{H} \colon \MALG ( \pre{\omega}{2} , { r \measurecantor} ) \to \MALG ( K , \mu ) , \quad \eq{A} \mapsto \eq{ H [ A ] } ,
\] 
where \( K = \ran H \).
There is a natural embedding \( j \colon \MALG ( K , \mu ) \hookrightarrow \MALG ( X , \mu ) \), sending each \( \eq{A}_K \equalsdef \setof{ B \in \MEAS_\mu \cap \Pow ( K )}{ B =_\mu A } \) to \( \eq{A}_X \equalsdef \setof{ B \in \MEAS_\mu }{ B =_\mu A } \).
Then \( j \circ \hat{H} \) is a reduction witnessing both \( \mathscr{K} ( \pre{\omega }{2} , r \measurecantor ) \leqW \mathscr{ K } ( X , \mu ) \) and \( \mathscr{K} ( \pre{\omega }{2} , r \measurecantor ) \leqW \mathscr{ F } ( X , \mu ) \).
For the second reduction, argue as follows: if \( H ( A ) =_\mu F \) for some closed \( F \subseteq X \), then \( H ( A ) =_\mu F \cap K \) hence \( \eq{A} \in \mathscr{ K } ( \pre{\omega}{2} , { r \measurecantor} ) \).
\end{proof}

By Proposition~\ref{thm:setofcompactsinMALG} and Corollary~\ref{cor:KisPi03completeCantor},

\begin{theorem}\label{thm:KisPi03complete}
Let \( ( X , d , \mu ) \) be a Heine-Borel space such that every compact has finite measure, and suppose \( \mu \) is nonsingular. 
Then \( \mathscr{K} ( X , \mu ) \) and \( \mathscr{F} ( X , \mu ) \) are \( \bPi^{0}_{3} \)-complete. 
\end{theorem}

\section{The set of exceptional points}\label{sec:exceptionalpoints}
\begin{theorem}\label{thm:blurrypointsSigma03}
Suppose \( \emptyset \neq A \subseteq {}^{ \omega }2 \) has empty interior, and \( A = \Phi ( A ) \).
Then \( \Blur ( A ) \) is \( \bSigma^{0}_{3} \)-complete.
\end{theorem}

\begin{proof}
For any \( z \in \pre{ \omega \times \omega }{2} \), let \( z' \in \pre{ \omega \times \omega }{2} \) be defined by the conditions
\[
\begin{cases}
 z' ( 2i , 2j ) = z' ( 2i + 1 , 2j + 1 ) = z ( i , j ) 
 \\
 z' ( 2i , 2j + 1 ) = z' ( 2i + 1 , 2j ) = 0
\end{cases}
\]
for all \( i , j \in \omega \).
The function \( \pre{ \omega \times \omega }{2} \to \pre{ \omega \times \omega }{2} \), \( z \mapsto z' \) is continuous.

Recall the tree \( \densitytree ( A ) \) defined in~\eqref{eq:densitytree}.
Given \( a \in \Pre{ n \times n }{2} \), a node \( \psi ( a ) \in \densitytree ( A ) \) is constructed with the property that
\[
a \subset b \implies \psi ( a ) \subset \psi ( b )
\]
so that defining
\[
f \colon \pre{ \omega \times \omega }{2} \to \body{ \densitytree ( A ) } , \quad f ( z ) = \bigcup_{ n \in \omega }\psi ( z' \restriction n \times n ) ,
\]
the function \( f \) is continuous and will witness \( \boldsymbol{P}_3^\complement \leqW \Blur ( A ) \).
Define \( I_n \), \( \rho \) as in the proof of~\cite[section 7.1]{Andretta:2013uq}, that is \( I_n = \cointerval{ 1 - 2^{ - n} }{ 1 - 2^{ - n - 1} } \) and \( \rho ( s ) = n \iff \measurecantor ( \LOC{A}{ s} ) \in I_n \).

Let \( \psi ( \emptyset ) = \emptyset \).
Given \( a \in \Pre{ ( n + 1 ) \times ( n + 1 ) }{2} \) define \( \psi ( a ) = t \) as follows:
\begin{itemize}
\item 
If \( \FORALL{ j \leq n} [ a ( j , n ) = 0 ] \), by \cite[Proposition 3.5]{Andretta:2013uq} let \( t \in \densitytree ( A ) \) be a proper extension of \( \psi ( a \restriction n \times n ) \) such that \( \rho ( t ) \geq n + 1 \) and
\[
\FORALL{u} \left [ \psi ( a \restriction n \times n ) \subseteq u \subseteq t \IMPLIES \rho ( u ) \geq \rho \left ( \psi ( a \restriction n \times n ) \right ) \right ] .
\]
\item 
If \( \EXISTS{ j \leq n} [ a ( j , n ) = 1 ] \), let \( j_0 \) be the least such \( j \).
By \cite[Proposition 3.5 and Claim 7.0.1]{Andretta:2013uq}, let \( t \in \densitytree ( A ) \) be a proper extension of \( \psi ( a \restriction n \times n ) \) with \( \rho ( t ) = 2 j_0 \) and
\[
\FORALL{u} \left [ \psi ( a \restriction n \times n ) \subseteq u \subseteq t \IMPLIES \rho ( u ) \geq \min \setLR{ \rho \left ( \psi ( a \restriction n \times n ) \right ) , 2 j_0 } \right ] .
\]
\end{itemize}
Suppose \( z \in \boldsymbol{P}_3 \), so that \( z' \in \boldsymbol{P}_3 \) as well.
For every \( k \in \omega \) choose \( m_k \in \omega \) such that \( \FORALL{ m \geq m_k} [ z' ( k , m ) = 0 ] \) and let \( M_k = \max \setLR{ m_0 , \ldots , m_k } \).
Therefore for every \( n \geq \max \setLR{ k , M_k } \), the least \( j \leq n \) such that \( z' ( j , n ) = 1 \)---if such a \( j \) exists---is larger than \( k \) and thus \( \rho \left (\psi ( z' \restriction n \times n ) \right ) > k \).
This shows that \( \lim_{i \to \infty }\rho ( f ( z ) \restriction i ) = + \infty \) hence \( f ( z ) \in \Phi ( A ) \).

Conversely, suppose \( z \notin \boldsymbol{P}_3 \).
Let \( n_0 \) be the least \( n \) such that \( \EXISTSS{ \infty }{ m } z ( n , m ) = 1 \).
This means that \( 2n_0 \) is the least \( n \) such that \( \EXISTSS{ \infty }{ m} [ z' ( n , m ) = 1 ] \); moreover, whenever \( z' ( 2n_0 , m ) = 1 \), then \( z' ( 2 n_0 , m + 1 ) = 0 \) and \( z' ( 2n_0 + 1 , m + 1 ) = 1 \).
Then there are arbitrarily large values of \( n \) such that
\[ 
\rho \left ( \psi ( z' \restriction n \times n ) \right ) = 4 n_0 , \quad \rho \left ( \psi ( z' \restriction ( n + 1 ) \times ( n + 1 ) ) \right ) = 4 n_0 + 2 
\] 
hence \( \rho ( f ( z ) \restriction i ) = 4 n_0 \) for infinitely many values of \( i \) and \( \rho ( f ( z ) \restriction i ) = 4 n_0 + 2 \) for infinitely many values of \( i \).
From this it follows that \( f ( z ) \in \Blur (A) \).
\end{proof}

In~\cite[Theorems 1.3 and 1.7]{Andretta:2013uq} it is shown that in the Cantor space the set \( \eq{A} \in \MALG \) such that \( A = \Phi ( A ) \) and \( \Int ( A ) = \emptyset \) is comeager in \( \MALG \).

\begin{corollary}\label{cor:blurrycomeager}
\( \setof{\eq{A} \in \MALG ( \pre{\omega}{2} ) }{ \Blur ( A ) \text{ is \( \bSigma^{0}_{3} \)-complete}} \) and \( \setof{\eq{A} \in \MALG ( \pre{\omega}{2} ) }{ \Exc ( A ) \text{ is \( \bSigma^{0}_{3} \)-complete}} \) are both comeager in \( \MALG \).
\end{corollary}

\begin{theorem}\label{thm:sharppointsPi03}
There is a \( K \in \KK ( \pre{\omega }{2} ) \) such that \( \Phi ( K ) \) is open, and \( \Sharp ( K ) \) is \( \bPi^{0}_{3} \)-complete.
Moreover for any given \( r \in ( 0 ; 1 ) \) we can arrange that \( \setof{ x \in \pre{\omega}{2} }{ \density_K ( x ) = r } \) is \( \bPi^{0}_{3} \)-complete.
\end{theorem}

\begin{proof}
We will construct a compact set \( K \subseteq \pre{\omega }{2} \) together with a continuous injective \( f \colon \pre{\omega \times \omega }{2} \to \pre{\omega}{2} \) such that \( \ran f \subseteq \Exc ( K ) \) and \( f \) witnesses that \( \boldsymbol{P}_3 \leqW \Sharp ( K ) \). 
The construction is arranged so that 
\begin{subequations}
\begin{align}
z \in \boldsymbol{P}_3 & \IMPLIES \density_K ( f ( z ) ) = r , \label{eq:th:sharppointsPi03converges}
\\
z \notin \boldsymbol{P}_3 & \IMPLIES \oscillation_K ( f ( z ) ) > 0 , \label{eq:th:sharppointsPi03oscillates}
\end{align} 
\end{subequations}
where \( r \in ( 0 ; 1 ) \) is some fixed value that can be chosen in advance.

We will define a collection \( \tilde{ \mathcal{G} } \subseteq \pre{ < \omega }{2} \) whose elements are called \markdef{good nodes} such that its closure under initial segments 
\begin{equation}\label{eq:th:sharppointsPi03defT}
T = \setof{ t \in \pre{ < \omega }{2} }{ \exists s \in \tilde{ \mathcal{G} } ( t \subseteq s )} 
\end{equation}
is a pruned tree.
The set
\begin{equation}\label{eq:th:sharppointsPi03defK}
K = \body{T} \cup \bigcup_{ s \in \tilde{ \mathcal{G} } } s \conc U_s ,
\end{equation}
where the \( U_s \) are clopen, is compact.
We will arrange the construction so that 
\begin{subequations}
\begin{gather}
 \measurecantor ( \body{T} ) = 0 , \label{eq:th:sharppointsPi03-a}
\\
\forall s \in \tilde{ \mathcal{G} } \left ( \body{T} \cap ( s \conc U_s ) = \emptyset \right ) , \label{eq:th:sharppointsPi03-b}
\\
\ran f \subseteq \body{T} = \Exc ( K ) . \label{eq:th:sharppointsPi03-c}
\end{gather}
\end{subequations}
Therefore \( \Phi ( K ) = \bigcup_{ s \in \tilde{ \mathcal{G} } } s \conc U_s \) is open.

We define the function \( \rho \colon T \to \omega + 1 \)
\begin{equation}\label{eq:th:sharppointsPi03rho}
\rho ( t ) = n \IFF 2^{ - n - 2} \leq \card{ \measurecantor ( \LOC{K}{t} ) - r } < 2^{ - n - 1 } , 
\end{equation}
where \( \rho ( t ) = \omega \) just in case \( \measurecantor ( \LOC{K}{t} ) = r \).
The construction will ensure that \( \rho ( \emptyset ) = 0 \), that is 
\begin{equation}\label{eq:th:sharppointsPi03measureK}
 1 / 4 \leq \card{ \measurecantor ( K ) - r } < 1 / 2.
\end{equation}
We require that any good node \( t \) can be gently extended to a good node \( s \) having any prescribed value of the \( \rho \) function, that is to say: for every \( t \in \tilde{ \mathcal{G} } \)
\begin{subequations}
\begin{align}
m \geq \rho ( t ) & \IMPLIES \EXISTS{s \in \tilde{ \mathcal{G} } } \left ( s \supset t \wedge \rho ( s ) = m \wedge \forall u \left ( t \subseteq u \subset s \implies \rho ( u ) \geq \rho ( t ) \right ) \right ) \label{eq:goingup}
\\
m < \rho ( t ) & \IMPLIES \EXISTS{s \in \tilde{ \mathcal{G} } } \left ( s \supset t \wedge \rho ( s ) = m \wedge \forall u \left ( t \subseteq u \subset s \implies \rho ( u ) \geq m \right ) \right ) . \label{eq:goingdown}
\end{align}
\end{subequations}
Assuming all this can be done, we can define the reduction.

\paragraph{\bfseries The construction of \( f \).}
For \( a \in \Pre{ n \times n }{2} \) let \( \gamma ( a ) \) be the first row (if it exists) where a \( 1 \) appears in column \( n - 1 \):
\[
 \gamma ( a ) = \begin{cases}
 \text{the least \( j \) such that } a ( j , n - 1 ) = 1 & \text{if } \EXISTS{j < n} \left ( a ( j , n - 1 ) = 1 \right ) ,
 \\
 n & \text{otherwise.}
 \end{cases}
\]
The function \( f \) is induced by a Lipschitz \( \varphi \colon \pre{ < \omega \times \omega }{2} \to T \); in fact \( \varphi \) will take values in \( \tilde{ \mathcal{G} } \) and will satisfy that
\[
\rho ( \varphi ( a ) ) = \gamma ( a ) .
\]
Here is the definition of \( \varphi \).
\begin{itemize}[leftmargin=1pc]
\item
Set \( \varphi ( \emptyset ) = \emptyset \).
Then \( \rho ( \varphi ( \emptyset ) ) = \rho ( \emptyset ) = 0 = \gamma ( \emptyset ) \) by~\eqref{eq:th:sharppointsPi03measureK}.
\item
Let us define \( \varphi ( a ) \) for \( a \in \Pre{ ( n + 1 ) \times ( n + 1 ) }{2} \), assuming \( \varphi ( a \restriction n \times n ) \) has been defined.
By~\eqref{eq:goingup} choose a good node \( t \supseteq \varphi ( a \restriction n \times n ) \) such that \( \rho ( t ) = n + 1 \) and such that \( \varphi ( a \restriction n \times n ) \subseteq u \subset t \implies \rho ( u ) \geq \gamma ( a \restriction n \times n ) = \rho ( \varphi ( a \restriction n \times n ) ) \).
\begin{description}
\item[Case 1] \( \gamma ( a ) = n + 1 \).
Then set \( \varphi ( a ) = t \).
\item[Case 2] \( \gamma ( a ) \leq n \).
Apply~\eqref{eq:goingdown} to get a good node \( s \supset t \) such that \( \rho ( s ) = \gamma ( a ) \) and \( t \subseteq u \subset s \implies \rho ( u ) \geq \gamma ( a ) \) and set \( \varphi ( a ) = s \).
\end{description}
\end{itemize}
Let us check that the function \( f = f_ \varphi \) is indeed the required reduction.

Suppose \( z \in \boldsymbol{P}_3 \): for all \( j \) there is \( N_j \) such that if \( n \geq N_j \) then \( \forall j' \leq j \left ( z ( j' , n ) = 0 \right ) \), and therefore \( \gamma ( z \restriction n \times n ) = \rho ( \varphi ( z \restriction n \times n ) ) > j \).
Since 
\[
 \forall j \exists N \FORALL{n \geq N} \left ( \rho ( \varphi ( z \restriction n \times n ) ) > j \right ) \IMPLIES \density_K ( f ( z ) ) = r ,
\]
then \( \density_K ( f ( z ) ) = r \) and \( f ( z ) \in \Sharp ( K ) \).
Thus~\eqref{eq:th:sharppointsPi03converges} holds.

Suppose \( z \notin \boldsymbol{P}_3 \): let \( j \) be least such that \( I = \setof{ n \in \omega }{ z ( j , n ) = 1 } \) is infinite.
Choose \( N > j \) such that for all \( n \geq N \) if \( j' < j \) then \( z ( j' , n ) = 0 \).
Fix \( n' > n > N \) such that \( n - 1 \) and \( n' - 1 \) are consecutive elements of \( I \).
Then for \( m \in \setLR{ n , n' } \)
\[
2^{ - j - 2 } \leq \card{ \measurecantor ( \LOC{K}{ \varphi ( z \restriction m \times m ) } ) - r } < 2^{ - j - 1 }
\]
while by definition of \( \varphi \) there is \( t \) such that \( \rho ( t ) = n \) and \( \varphi ( z \restriction n \times n ) \subset t \subset \varphi ( z \restriction n' \times n' ) \).
Therefore, as \( n > N > j \)
\[
2^{ - n - 2 } \leq \card{ \measurecantor ( \LOC{K}{t} ) - r } < 2^{ - n - 1 } < 2^{ - j - 2 }
\]
hence \( \oscillation _K ( f ( z ) ) > 0 \) and \( f ( z ) \in \Blur ( K ) \).
Thus~\eqref{eq:th:sharppointsPi03oscillates} holds.

Therefore it is enough to construct \( \tilde{\mathcal{G}} \), and hence \( T \) and \( K \), so that~\eqref{eq:th:sharppointsPi03-a}--\eqref{eq:th:sharppointsPi03-c}, \eqref{eq:th:sharppointsPi03measureK}, and \eqref{eq:goingup}--\eqref{eq:goingdown} are satisfied.

\paragraph{\bfseries The construction of \( \tilde{\mathcal{G}} \), \( T \), and \( K \).}
Choose \( r_n \in \mathbb{ D } \) such that 
\begin{equation}\label{eq:th:sharppointsPi03r_n}
2^{ - n - 2 } + 2^{ - n - 4 } \leq \card{ r_n - r } < 2^{ - n - 1 } - 2^{ - n - 4 } .
\end{equation}
Let \( D_n \) be clopen such that \( \measurecantor ( D_ n ) = r_n \), let \( u_n = 0^{ ( n + 6 ) } \) and \( v_n = 1^{ ( n + 6 ) } \), and
\[	
E_n = \bigcup_{ 0 < i \leq n + 5 } \left ( 0^{ ( i ) } \conc 1 \conc D_n \cup 1^{ ( i ) } \conc 0 \conc D_n \right )
\]
Thus \( u_0 \), \( v_0 \), and \( E_0 \) can be visualized as follows (the grey area is \( D_0 \)): 
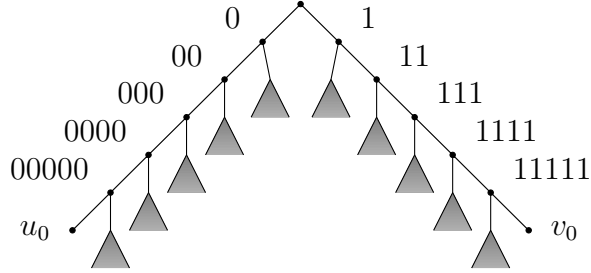
\[
\begin{tikzpicture}[scale=0.5]
\filldraw (2,-6) circle (2pt) -- (3,-5) circle (2pt) --(4, -4) circle (2pt) --(5, -3) circle (2pt) -- (6 , -2) circle (2pt) -- (7 , -1) circle (2pt) -- (8,0) circle (2pt) -- (9 , -1) circle (2pt) -- (10 , -2) circle (2pt) -- (11 , -3) circle (2pt) -- (12 , -4) circle (2pt) -- (13 , -5) circle (2pt) -- (14 , -6) circle (2pt) ;
\node at (7 , -1) [label=165:\( 0 \)]{};
\node at (6 , -2) [label=165:\( 00 \)]{};
\node at (5 , -3) [label=165:\( 000 \)]{};
\node at (4 , -4) [label=165:\( 0000 \)]{};
\node at (3 , -5) [label=165:\( 00000 \)]{};
\node at (2 , -6) [label=180:\( u_0 \)]{};
\node at (9, -1) [label=15:\( 1 \)]{};
\node at (10 , -2) [label=15:\( 11 \)]{};
\node at (11 , -3) [label=15:\( 111 \)]{};
\node at (12 , -4) [label=15:\( 1111 \)]{};
\node at (13, -5) [label=15:\( 11111 \)]{};
\node at (14, -6) [label=0:\( v_0 \)]{};
\fill [top color=gray, bottom color=gray!60] (3,-6)--(2.5, -7)--(3.5, -7)--cycle;
\draw (3 , -5)-- (3,-6)--(2.5, -7);
\draw (3,-6)--(3.5, -7);
\fill [top color=gray, bottom color=gray!60] (4,-5)--(3.5, -6)--(4.5, -6)--cycle;
\draw (4 , -4)-- (4,-5)--(3.5, -6);
\draw (4,-5)--(4.5, -6);
\fill [top color=gray, bottom color=gray!60] (5,-4)--(4.5, -5)--(5.5, -5)--cycle;
\draw (5 , -3)-- (5,-4)--(4.5, -5);
\draw (5,-4)--(5.5, -5);
\fill [top color=gray, bottom color=gray!60] (6,-3)--(5.5, -4)--(6.5, -4)--cycle;
\draw (6 , -2)-- (6,-3)--(5.5, -4);
\draw (6,-3)--(6.5, -4);
\fill [top color=gray, bottom color=gray!60] (7.2, -2)--(7.7 , -3)--(6.7 , -3)--cycle;
\draw (7 , -1)--(7.2, -2)--(7.7 , -3);
\draw (7.2, -2)--(6.7 , -3);
\fill [top color=gray, bottom color=gray!60] (8.8, -2)--(8.3 , -3)--(9.3 , -3)--cycle;
\draw (9 , -1)--(8.8, -2)--(9.3 , -3);
\draw (8.8, -2)--(8.3 , -3);
\fill [top color=gray, bottom color=gray!60] (10,-3)--(9.5, -4)--(10.5, -4)--cycle;
\draw (10 , -2)-- (10,-3)--(9.5, -4);
\draw (10,-3)--(10.5, -4);
\fill [top color=gray, bottom color=gray!60] (11,-4)--(10.5, -5)--(11.5, -5)--cycle;
\draw (11 , -3)-- (11,-4)--(10.5, -5);
\draw (11,-4)--(11.5, -5);
\fill [top color=gray, bottom color=gray!60] (12,-5)--(11.5, -6)--(12.5, -6)--cycle;
\draw (12 , -4)-- (12,-5 )--(11.5, -6);
\draw (12,-5)--(12.5, -6);
\fill [top color=gray, bottom color=gray!60] (13,-6)--(12.5, -7)--(13.5, -7)--cycle;
\draw (13 , -5 )-- (13,-6 )--(12.5, -7);
\draw (13,-6 )--(13.5, -7 );
\end{tikzpicture}
\]
Therefore
\begin{equation}\label{eq:th:sharppointsPi03-error}
\measurecantor ( E_n ) = r_n \left ( 1 - 2^{ - n - 5 } \right )
\end{equation}
and 
\begin{equation}\label{eq:th:sharppointsPi03-N_s}
 \Nbhd_{u_n} \cap E_n = \Nbhd_{ v_n } \cap E_n = \emptyset .
\end{equation}
We are now ready to define \( \tilde{ \mathcal{G} } \) and \( T \).
Let 
\[
\begin{split}
\Sigma & = \setofLR{ u_n }{ n \in \omega } \cup \setofLR{ v_n }{ n \in \omega \setminus \set{0} } 
\\
 & = \setofLR{0^{( k )} , 1^{( k + 1 ) } }{ k \geq 6 } .
\end{split} 
\] 
A sequence \( \sigma \in \pre{ < \omega }{ \Sigma } \) is
\begin{itemize}[leftmargin=1pc]
\item 
\markdef{ascending} if it is of the form \( \seq{ u_n , u_{n + 1} , \dots , u_{n + k } } \) with \( n , k \geq 0 \),
\item 
\markdef{descending} if it is of the form \( \seq{ v_n , v_{n - 1} , \dots , v_{n - k } } \) with \( n > k \geq 0 \),
\item 
\markdef{good} if either 
\begin{itemize}
\item
\( \sigma = \emptyset \), or else
\item
it is \markdef{positive}, that is a concatenation of an odd number of blocks of ascending and descending sequences, where the ascending and descending sequences alternate:
\[
\sigma = \seq{ u_0 , \dots , u_{ n_0 } } \conc \seq{ v_{ n_0 + 1 } , \dots , v_{ n_1 } } \conc \seq{ u_{ n_1 - 1 } , \dots , u_{ n_2 } } \conc \dots \conc \seq{ u_{ n_k - 1} , \dots , u_{ n_{k + 1} } } ,
\]
or else
\item
 it is \markdef{negative}, that is a concatenation of an even number of blocks of ascending and descending sequences, where the ascending and descending sequences alternate: 
\[
\sigma = \seq{ u_0 , \dots , u_{ n_0 } } \conc \seq{ v_{ n_0 + 1 } , \dots , v_{ n_1 } } \conc \seq{ u_{ n_1 - 1 } , \dots , u_{ n_2 } } \conc \dots \conc \seq{ v_{ n_k + 1 } , \dots , v_{ n_{k + 1} } } .
\]
\end{itemize}
\end{itemize}

The collection \( \mathcal{G} \) of all good sequences \( \sigma \) is a tree on \( \Sigma \), and can be defined as follows (see Figure~\ref{fig:treeofgoodnodes}): 
\begin{itemize}[leftmargin=1pc]
\item
\( \seq{ u_0 } \) is the least nonempty node, 
\item
if a node \( \sigma \) ends with \( u_k \), then its immediate successors are \( \sigma \conc \seq{ u_{k + 1} } \) and \( \sigma \conc \seq{ v_{ k + 1 } } \), 
\item
if the node \( \sigma \) ends with \( v_k \) then:
\begin{itemize}
\item 
if \( k > 1 \) there are two immediate successors \( \sigma \conc \seq{ u_{ k - 1} } \) and \( \sigma \conc \seq{v_{k - 1 }} \), 
\item
if \( k = 1 \) then there is a unique immediate successor \( \sigma \conc \seq{ u_0 } \).
\end{itemize}
\end{itemize}
\begin{figure}
\[
 \begin{tikzpicture}
\foreach \x in {0,...,3}
\filldraw ( 0 , -\x ) circle [radius=1pt] ;
\foreach \x in {0,...,3}
\draw ( 0 , -\x )-- node[left] {\( u_{\x} \)} (0 , -\x - 1 ) ;
\draw ( 0 , -1 )-- node[above] {\( v_{1} \)} (10 , -2 ) ;
\draw ( 10 , -2 )-- node[left] {\( u_{0} \)} (10 , -3 ) ;
\draw ( 0 , -2 )-- node[above] {\( v_{2} \)} (8 , -3 ) ;
\draw ( 8 , -3 )-- node[right] {\( v_{1} \)} (9 , -4 ) ;
\draw ( 9 , -4 )-- node[right] {\( u_{0} \)} (10 , -5 ) ;
\draw ( 8 , -3 )-- node[left] {\( u_{1} \)} (8 , -4 ) ;
\draw ( 8 , -4 )-- node[left] {\( u_{2} \)} (8 , -5 ) ;
\draw ( 8 , -4 )-- node[right] {\( v_{2} \)} (9 , -5 ) ;
\draw (0 , -3 )-- node[above] {\( v_{3} \)} (2 , -3.33333 ) ;
\draw (2 , -3.33333 )-- node[left] {\( u_{2} \)} (2 , -4.33333 ) ;
\draw (2 , -3.33333 )-- node[above] {\( v_{2} \)} (4 , -3.66666 ) ;
\draw (4 , -3.66666 )-- node[above] {\( v_{1} \)} (6 , -4 ) ;
\draw (4 , -3.66666 )-- node[left] {\( u_{1} \)} (4 , -4.66666 ) ;
\draw (6 , -4 )-- node[left] {\( u_{0} \)} (6,-5 ) ;
\filldraw (10 , -2) circle [radius=1pt] ;
\filldraw (8 , -3 ) circle [radius=1pt] ;
\filldraw (9 , -4 ) circle [radius=1pt] ;
\filldraw (8 , -4 ) circle [radius=1pt] ;
\filldraw (2 , -3.33333) circle [radius=1pt] ;
\filldraw (4 , -3.66666 ) circle [radius=1pt] ;
\filldraw (6 , -4) circle [radius=1pt] ;
\end{tikzpicture}
\]
 \caption{The first few nodes of the tree \( \mathcal{G} \)}
 \label{fig:treeofgoodnodes}
\end{figure}
Given \( \sigma \in \mathcal{G} \) let \( \tilde{ \sigma } \in \pre{ < \omega }{2} \) be the sequence obtained by concatenating the sequences in \( \sigma \).
In other words, if \( \sigma \) is positive as above then 
\[
\tilde{ \sigma } = \underbracket[0.5pt]{u_0 \conc \dots \conc u_{ n_0 }} \conc \underbracket[0.5pt]{v_{ n_0 + 1 } \conc \dots \conc v_{ n_1 }} \conc \underbracket[0.5pt]{ u_{ n_1 - 1 } \conc \dots \conc u_{ n_2 } } \conc \dots \dots\conc \underbracket[0.5pt]{ u_{ n_k - 1} \conc \dots \conc u_{ n_{k + 1}}} ,
\] 
and similarly for negative \( \sigma \).
Let
\[ 
 \tilde{ \mathcal{G} } = \setofLR{ \tilde{ \sigma } }{ \sigma \in \mathcal{G} } \subseteq \pre{ < \omega}{2} .
\]
Note that any \( s \in \tilde{ \mathcal{G} } \) determines a unique \( \sigma \in \mathcal{G} \) such that \( s = \tilde{ \sigma } \).
Using the same notation as before, let \( \boldsymbol{n} ( s ) \) for \( s \in \tilde{ \mathcal{G} } \) be defined by
\[
\boldsymbol{n} ( s ) = \begin{cases}
n_{ k + 1 } + 1 & \text{if \( s \) is positive,}
\\
n_{ k + 1 } - 1 & \text{if \( s \) is negative,}
\\
0 & \text{if } s = \emptyset . 
\end{cases}
\]
A branch of \( \mathcal{G} \) is a sequence \( \seqofLR{ w_n }{ n \in \omega } \) of elements of \( \Sigma \) such that each \( \sigma _n \equalsdef \seq{ w_0 , \dots , w_n } \in \mathcal{G} \), so any branch of \( \mathcal{G} \) yields a branch of \( T \) by letting 
\begin{equation}\label{eq:branchfrombranch}
x = w_0 \conc w_1 \conc \dots = \bigcup_{ n \in \omega } \tilde{ \sigma }_n .
\end{equation}
Conversely, any \( x \in \body{T} \) yields a branch of \( \mathcal{G} \).
A branch \( x \) of \( \body{T} \) is oscillating if \( \setof{n \in \omega }{ \sigma _n \text{ is positive}} \) and \( \setof{n \in \omega }{ \sigma _n \text{ is negative}} \) are both infinite; otherwise \( \sigma _n \) is positive for all sufficiently large \( n \), and \( x \) is said to be positive.
Let 
\[ U_s = E_{ \boldsymbol{n} ( s ) } 
\] 
so that the definition of \( K \) as in~\eqref{eq:th:sharppointsPi03defK} is complete.

\paragraph{\bfseries Checking that the construction works.}
First of all we check that the function \( \rho \) of~\eqref{eq:th:sharppointsPi03rho} is defined on \( \tilde{\mathcal{G}} \).

\begin{claim}\label{claim:sharppointsPi03}
 \( \FORALL{ s \in \tilde{ \mathcal{G} } } \left ( \rho ( s ) = \boldsymbol{n} ( s ) \right ) \).
\end{claim}

\begin{proof}
Fix \( s \in \tilde{ \mathcal{G} } \) and let \( n = \boldsymbol{n} ( s ) \).
Equation~\eqref{eq:th:sharppointsPi03-error} yields that 
\[ 
\card{ \measurecantor ( \LOC{K}{s} ) - r_n } \leq \card{ \measurecantor ( \LOC{K}{s} ) - \measurecantor ( E_n ) } + \card{ \measurecantor ( E_n ) - r_n } \leq 2^{ - n - 5 } + r_n 2^{ - n - 5 } \leq 2^{ - n - 4 } .
\]
The triangular inequality and~\eqref{eq:th:sharppointsPi03r_n} imply that 
\begin{multline*}
 2^{ - n - 2 } \leq \card{ r_n - r } - \card{ \measurecantor ( \LOC{K}{s} ) - r_n } \leq \card{ \measurecantor ( \LOC{K}{s} ) - r } 
\\
{} \leq \card{ \measurecantor ( \LOC{K}{s} ) - r_n } + \card{ r_n - r } < 2^{ - n - 1 } ,
\end{multline*}
which is what we had to prove.
\end{proof}

Note that taking \( s = \emptyset \) we obtain that \( 1 / 4 \leq \card{ \measurecantor ( K) - r } < 1 /2 \) hence~\eqref{eq:th:sharppointsPi03measureK} holds.
Next we check that \( \rho \) is defined on all of \( T \). 

Fix \( s \in \tilde{\mathcal{G}} \) and let \( n = \boldsymbol{n} ( s ) \).
For \( 0 < k \leq n + 5 \) and \( i \in \set{0 , 1 } \) we have that
\[
\LOC{K}{ s \conc i^{( k )}} = i^{( n + 6 - k )} \conc \LOC{K}{ s \conc i^{( n + 6 )}} \cup \bigcup_{0 \leq j \leq n + 5 - k} i^{( j )} \conc ( 1 - i ) \conc D_n
\]
hence
\begin{equation}\label{eq:painful}
\measurecantor \bigl ( \LOC{K}{ s \conc i^{( k )}} \bigr ) = 2^{- n - 6 + k} \measurecantor \bigl ( \LOC{K}{ s \conc i^{( n + 6 )}} \bigr ) + r_n \left ( 1 - 2^{- n - 6 + k} \right ) . 
\end{equation}
Since \( \card{\measurecantor \bigl ( \LOC{K}{ s \conc i^{( n + 6 )}} \bigr ) - r } < 1 / 2 \) and \( \card{r_n - r } < 1 / 2 \) by~\eqref{eq:th:sharppointsPi03r_n}, it follows that \( \card{\measurecantor \bigl ( \LOC{K}{ s \conc i^{( k )}}\bigr ) -r } < 1 / 2 \).
Therefore \( \rho \colon T \to \omega + 1 \) is well-defined.

In order to verify~\eqref{eq:goingup} and~\eqref{eq:goingdown}, it is enough to prove them when \( m = \rho ( t ) + 1 \) and \( m = \rho ( t ) - 1 \), if \( \rho ( t ) \neq 0 \).
So fix \( t \in \tilde{\mathcal{G}} \) and let \( n = \boldsymbol{n} ( t ) = \rho ( t ) \).
If \( n = 0 \), then either \( t = \emptyset \) or else it ends with \( v_1 \), and therefore it has exactly one immediate successor \( s^+ \) in \( \tilde{\mathcal{G}} \), and \( \rho ( s^+ ) = 1 \).
If \( n > 0 \) then it has two immediate successors \( s^+ \) and \( s^- \) in \( \tilde{\mathcal{G}} \), that is \( s^+ = t \conc 0^{ ( n + 6 ) } \) and \( s^- = t \conc 1^{ ( n + 6 ) } \), and \( \rho ( s^+ ) = n + 1 \) and \( \rho ( s^- ) = n - 1 \).
We must check that if \( t \subset u \subset s^+ \) then \( \rho ( u ) \geq n \), and that if \( t \subset u \subset s^- \) then \( \rho ( u ) \geq n - 1 \).
If \( u = t \conc 0^{( k )} \) then 
\begin{align*}
\card{ \measurecantor ( \LOC{K}{ t \conc 0^{( k )} } ) - r } & = \cardLR{ \frac{ \measurecantor ( \LOC{K}{s^+} ) }{2^{ n + 6 - k } } + r_n \left ( 1 - \frac{1}{2^{ n + 6 - k } }\right ) - r} && \text{by~\eqref{eq:painful}}
\\
 & \leq \frac{1}{2^{ n + 6 - k } } \card{\measurecantor ( \LOC{K}{s^+} ) - r } + \left ( 1 - \frac{1}{2^{ n + 6 - k } } \right ) \card{ r_n - r }
 \\
 & < \frac{1}{2^{ n + 6 - k } } 2^{-n - 2} + \left ( 1 - \frac{1}{2^{ n + 6 - k } } \right ) 2^{- n - 1} &&\text{by~\eqref{eq:th:sharppointsPi03r_n}}
 \\
 & < 2^{-n - 1} ,
\end{align*}
and if \( u = t \conc 1^{( k )} \) with similar computations we obtain
\[
\card{ \measurecantor ( \LOC{K}{ t \conc 1^{( k )} } ) - r } \leq \frac{1}{2^{ n + 6 - k } } \card{\measurecantor ( \LOC{K}{s^-} ) - r } + \left ( 1 - \frac{1}{2^{ n + 6 - k } } \right ) \card{ r_n - r } < 2^{-n} .
\]
Therefore~\eqref{eq:goingup} and~\eqref{eq:goingdown} hold.

Let us check that~\eqref{eq:th:sharppointsPi03-a}--\eqref{eq:th:sharppointsPi03-c} hold.
Equation~\eqref{eq:th:sharppointsPi03-a} follows from the fact that \( \lh ( u_n ) , \lh ( v_n ) \geq 6 \) for all \( n \), equation~\eqref{eq:th:sharppointsPi03-b} follows from~\eqref{eq:th:sharppointsPi03-N_s}, equation~\eqref{eq:th:sharppointsPi03-c} follows by definition of \( \varphi \). 
\end{proof}

\begin{remark}
Corollary~\ref{cor:blurrycomeager}shows that \( \Blur ( A ) \) is \( \bSigma^{0}_{3} \)-complete for \emph{most} \( \eq{A} \) in the measure algebra, while Theorem~\ref{thm:sharppointsPi03} constructs \emph{some specific} compact \( K \) such that \( \Sharp ( K )  \) is \( \bPi^{0}_{3} \)-complete.
This asymmetry is to be expected as the proof (and the statement) of Theorem~\ref{thm:sharppointsPi03} hinges on the choice of the value \( r \).
\end{remark}

\section{Spongy and solid sets in \( \R^n \)}\label{sec:solid&spongy}
In this section we shall construct a spongy subset of \( \R \) (Theorem~\ref{thm:spongy}) and we shall show that a solid subset of \( \R^n \) has always points of density \( 1 / 2 \) (Corollary~\ref{cor:nodualisticsetsinRn}).
 
\subsection{Spongy sets}\label{subsec:spongy}
The goal of this section is to prove the following

\begin{theorem}\label{thm:spongy^n}
For each \( n \geq 1 \), there is a bounded spongy set \( S \subseteq \R^n \).
Furthermore \( S \) can be taken to be either open or closed.
\end{theorem}

The crux of the matter is establishing the result for \( \R \) (Theorem~\ref{thm:spongy}), and
 this is achieved by a triadic Cantor-construction of non-shrinking diameter (Section~\ref{subsec:Cantorschemes}) .

\subsubsection{Some notation}
Before we jump in the technical details, let us introduce some notation that will be useful in this section.

For \( a \leq b \), \( [ a ; b ] \) denotes either the \emph{closed interval with endpoints \( a , b \)}, when \( a < b \) or else the \emph{singleton} \( \setLR{a} \), when \( a = b \).

Given an interval \( [ a ; b ] \) of length \( \leq 1 \) let
\[ 
\varepsilon < \frac{b - a}{ 3 + 2 M } \leq \frac{1}{ 3 + 2 M } ,
\] 
where \( M \) is some number greater that \( 1 \), and let \( \Psi_{ \varepsilon } ( [ a ; b ] ) \) be the set obtained by removing from \( [ a ; b ] \) two open intervals \( ( a + \varepsilon ; a + ( 1 + M ) \varepsilon) \) and \( ( b - ( 1 + M ) \varepsilon ; b - \varepsilon ) \), each of length \( M \varepsilon \), that is 
\[ 
\Psi_{ \varepsilon } ( [ a ; b ] ) = [ a ; a + \varepsilon ] \cup [ a + ( 1 + M ) \varepsilon ; b - ( 1 + M ) \varepsilon ] \cup [ b - \varepsilon ; b ] .
\]
The set \( \Psi_{ \varepsilon } ( [ a ; b ] ) \) has three connected components: two side intervals of length \( \varepsilon \), and a middle interval of length \( b - a - 2 ( 1 + M ) \varepsilon \).
By choice of \( \varepsilon \), the middle interval is of length \( > \varepsilon \).
Since \( \varepsilon ^2 < \varepsilon / ( 3 + 2 M ) \) and since each of the three intervals has length \( \geq \varepsilon \), we can apply the operation \( \Psi_{ \varepsilon ^2} \) to each of the three intervals obtained so far, obtaining nine closed intervals.
This procedure can be iterated: at stage \( n \) we have \( 3^{ n } \) closed intervals, and we apply the operation \( \Psi_{ \varepsilon ^{n + 1}} \) to them.
Let 
\[
 H_n ( a , b ) = \Bigl [ a + ( 1 + M ) \sum_{k = 1}^n \varepsilon ^k ; b - ( 1 + M ) \sum_{k = 1}^n \varepsilon ^k \Bigr ] 
\]
be the center-most interval constructed at stage \( n \), i.e. the one containing the point \( ( a + b ) / 2 \).
As \( ( 1 + M ) \sum_{k = 1}^\infty \varepsilon ^k = \frac{ ( 1 + M ) \varepsilon }{ 1 - \varepsilon } < \frac{b - a}{2} \), it follows that 
\begin{equation}\label{eq:connectedcomponent1}
\bigcap_{n} H_n ( a , b ) = \Bigl [ a + ( 1 + M ) \sum_{k = 1}^\infty \varepsilon ^k ; b - ( 1 + M ) \sum_{k = 1}^\infty \varepsilon ^k \Bigr ] 
\end{equation}
is a closed interval.

\subsubsection{The construction}
Fix \( M > 1 \) and let \( 0 < \varepsilon < \frac{1}{ 3 + 2 M } \).
Consider the triadic Cantor-construction obtained by applying the \( \Psi_{ \varepsilon ^{ n + 1 } } \) operations, that is let
\[
\seqof{ K_s , I_s^- , I_s^+ }{ s \in \pre{ < \omega}{ \set{ -1 , 0 , 1 } } } 
\] 
be a sequence of intervals such that
\begin{itemize}
\item
\( K_s = [ a_s ; b_s ] \) and \( K_\emptyset = [ 0 ; 1 ] \), that is \( a_\emptyset = 0 \) and \( b_\emptyset = 1 \),
\item
\( I_s^- = ( a_s + \varepsilon^{ \lh ( s ) + 1} ; a_s + ( 1 + M ) \varepsilon^{ \lh ( s ) + 1} ) \) and \( I_s^+ = ( b_s - ( 1 + M ) \varepsilon^{ \lh ( s ) + 1} ; b_s - \varepsilon^{ \lh ( s ) + 1} ) \).
\end{itemize}
Figure~\ref{fig:spongy} may help to visualize the construction.
\begin{figure}
\centering
\begin{tikzpicture}[scale=1.35,text height=0ex,text depth=0ex,pin distance=1cm]
\node at (0,0.5) {};
\draw[decorate,decoration=brace] (0 , 0.1)-- node[above] {\( \varepsilon ^{\lh ( s ) + 1} \)} ( 0.98, 0.1);
\draw[decorate,decoration=brace] (1.02, 0.1)-- node[above] {\( M \varepsilon ^{\lh ( s ) + 1} \)} ( 3 , 0.1);
\draw[decorate,decoration=brace] (7.05, 0.1)-- node[above] {\( M \varepsilon ^{\lh ( s ) + 1} \)} (8.98 , 0.1);
\draw[decorate,decoration=brace] (9.02, 0.1)-- node[above] {\( \varepsilon ^{\lh ( s ) + 1} \)} ( 10 , 0.1);
\draw[decorate,decoration=brace] (0, 0.5)-- node[above] {\( K _{ s } \)} ( 10 , 0.5);
\draw[thick] (0,0)--(1,0);
\draw[thick] (3,0)--(7,0);
\draw[thick] (9,0)--(10,0);
\draw[very thin] (1,0)--(3,0) ;
\draw[very thin] (7,0)--(9,0) ;
 \node [circle,fill=black,inner sep = 1pt,pin=270:$a_{s \conc \seq{-1}}$] {};
 \node at (1,0) [circle,fill=black, inner sep = 1pt,pin=270:$b_{s \conc \seq{-1}}$] {};
 \node at (3,0) [circle,fill=black, inner sep = 1pt,pin=270:$a_{s \conc \seq{0}}$] {};
 \node at (7,0) [circle,fill=black, inner sep = 1pt,pin=270:$b_{s \conc \seq{0}}$] {};
 \node at (9,0) [circle,fill=black, inner sep = 1pt,pin=270:$a_{s \conc \seq{1}}$] {};
 \node at (10,0) [circle,fill=black, inner sep = 1pt,pin=270:$b_{s \conc \seq{1}}$] {};
\node at (10,0) [right] {\( b_s \)};
\node at (0,0) [left] {\( a_s \)};
\node at (0.5,-0.15) [below]{\( \scriptscriptstyle K_{ s \conc \seq{ -1 }} \)};
\node at (2,-0.15) [below]{\( \scriptscriptstyle I_s^- \)};
\node at (5,-0.15) [below]{\( \scriptscriptstyle K_{ s \conc \seq{ 0 }} \)};
\node at (8,-0.15) [below]{\( \scriptscriptstyle I_s^+ \)};
\node at (9.5, -0.15) [below]{\( \scriptscriptstyle K_{ s \conc \seq{ 1 }} \)};
\end{tikzpicture}
\caption{Where the intervals \( K_{ s \conc \seq{ -1 }} , K_{ s \conc \seq{ 0 }} , K_{ s \conc \seq{ 1 }} \) lie in \( K_s \).}\label{fig:spongy}
\end{figure}
Following the notation in Section~\ref{subsec:Cantorschemes}, let 
\begin{align*}
 K^{( n )} & = \bigcup_{s \in \pre{ n }{ \set{ -1 , 0 , 1 } } } K_s
\\
 K = \bigcap_{n \in \omega } K^{( n )} &= \bigcup_{z \in \pre{ \omega }{\setLR{ - 1 , 0 , 1 }}} \bigcap_{n \in \omega } K_{ z \restriction n } .
\end{align*} 
By induction on \( \lh s \), one checks that \( \card{ K_s } \geq \varepsilon ^{\lh s} \) and \( \varepsilon ^{ \lh ( s ) + 1 } < \card{ K_s } / ( 3 + 2 M ) \), and if \( \lh s > 0 \) then 
 \begin{equation}\label{eq:connectedcomponent2}
s ( \lh ( s ) - 1 ) \in \setLR{ -1 , 1 } \IFF \card{ K_s } = \varepsilon ^{\lh s} .
\end{equation}
Recall that the connected components of \( K \) are the sets 
\[ 
\bigcap_{n \in \omega } K_{ z \restriction n} = [ a_z ; b_z ] 
\]
where \( a_z = \sup_{n \to \infty} a _{z \restriction n} \) and \( b_z = \inf_{n \to \infty} b _{z \restriction n} \).
By~\eqref{eq:connectedcomponent1} and~\eqref{eq:connectedcomponent2} \( a_z < b_z \IFF z \in F \), where
\[
F = \setofLR{ z \in \pre{ \omega }{\setLR{ - 1 , 0 , 1 }} }{ \EXISTS{n} \FORALL{m \geq n } ( z ( n ) = 0 ) } .
\]
Therefore \( \Int ( K ) = \bigcup \setofLR{ ( a_z ; b_z )}{ z \in F } \) and \( \lambda ( K ) > 0 \).

Let \( s \in \pre{ < \omega }{ \setLR{ - 1 , 0 , 1 }} \).
By induction on \( \lh s \), it can be checked that
\begin{equation}\label{eq:spongydisjoint}
\left ( ( a_s - M \varepsilon ^{ \lh s } ; a_s ) \cup ( b_s ; b_s + M \varepsilon ^{ \lh s } ) \right ) \cap K^{ ( \lh s ) } = \emptyset ,
\end{equation}
hence \( ( a_s - M \varepsilon ^{ \lh s } ; a_s ) \cup ( b_s ; b_s + M \varepsilon ^{ \lh s } ) \) is disjoint from \( K \), and that
\begin{equation}\label{eq:measurespongy}
\begin{split}
\lambda ( K_s \cap K ) & = \card{ K_s } - 2 M \sum_{ i = 0 }^{ \infty } 3^i \varepsilon^{ \lh ( s ) + i + 1 } 
\\
 & = \card{ K_s } - \frac{ 2 M \varepsilon^{ \lh ( s ) + 1 }}{ 1 - 3 \varepsilon } .
\end{split} 
\end{equation}
Clearly \( K = K ( M , \varepsilon ) \subseteq [ 0 ; 1 ] \) is compact, and depends on \( M \) and \( \varepsilon \).
Note that the construction above requires that \( \varepsilon < \frac{1}{ 3 + 2 M } \).
If  this requirement is strengthened by imposing that 
\[
 0 < \varepsilon < \varepsilon _0 \equalsdef \frac{ M - 1 }{ M ( 3 + 2 M ) - 3 } ,
\] 
a spongy set is obtained.

\begin{theorem}\label{thm:spongy}
\( \FORALL{ M > 1} \FORALL{ \varepsilon \in ( 0 ; \varepsilon _0 ) } \) the sets \( K ( M , \varepsilon ) \) and \( \Int \left ( K ( M , \varepsilon ) \right ) \) are spongy.
\end{theorem}

\begin{proof}
We are going to show that for \( M >1 \) and \( \varepsilon < \varepsilon _0 \)
\begin{equation*}
\FORALL{z \in \pre{ \omega }{\setLR{ - 1 , 0 , 1 }} } \left ( \oscillation_K ( a_z ) , \oscillation_K ( b_z ) > 0 \right ) .
\end{equation*}
Therefore \( \oscillation_K ( x ) > 0 \) for all \( x \in K \setminus \Int ( K ) = \setofLR{a_z , b_z}{ z \in \pre{ \omega }{\setLR{ - 1 , 0 , 1 }} } \), thus \( K \) is spongy and closed.
Since \( \Fr ( K ) = \Exc ( K ) \), by the Lebesgue density theorem \( K =_\mu \Int ( K ) \), so \( \Int ( K ) \) is spongy and open.

The idea behind the proof is an elaboration of the argument used in Examples~\ref{xmp:densitybutnoleftorrightdensities} and~\ref{xmp:oscillatingdensity}.

Let \( x \in K_{ s \conc \seq{ -1 } } \).
By~\eqref{eq:spongydisjoint} we have (see Figure~\ref{fig:spongy}):
\begin{equation}\label{eq:spongycontained}
( x - \varepsilon ^{ \lh ( s ) + 1} ; x + \varepsilon ^{ \lh ( s ) + 1} ) \cap K \subseteq ( x - M \varepsilon ^{ \lh ( s ) + 1} ; x + M \varepsilon ^{ \lh ( s ) + 1 } ) \cap K \subseteq K_{ s \conc \seq{ -1 } } 
\end{equation}
hence 
\begin{equation}\label{eq:th:spongy-a}
\begin{aligned}
\frac{\lambda ( ( x - M \varepsilon ^{ \lh ( s ) + 1} ; x + M \varepsilon ^{ \lh ( s ) + 1 } ) \cap K )}{ 2 M \varepsilon ^{ \lh ( s ) + 1}} & < \frac{ \card{ K_{ s \conc \seq{ -1 } } } }{ 2 M \varepsilon ^{ \lh ( s ) + 1}}
\\
 & = \frac{1}{ 2 M } & \text{by~\eqref{eq:connectedcomponent2}} 
\end{aligned}
\end{equation}
and by~\eqref{eq:measurespongy} with \( s \conc \seq{ -1 } \) in place of \( s \),
\begin{equation}\label{eq:th:spongy-b}
\begin{split}
\frac{ \lambda ( ( x - \varepsilon ^{ \lh ( s ) + 1} ; x + \varepsilon ^{ \lh ( s ) + 1} ) \cap K ) }{ 2 \varepsilon ^{ \lh ( s ) + 1}} & = \frac{1}{ 2 \varepsilon ^{ \lh ( s ) + 1}} \Bigl [ \varepsilon ^{ \lh ( s ) + 1} - \frac{ 2 M \varepsilon ^{ \lh ( s ) + 2} }{1 - 3 \varepsilon } \Bigr ]
 \\
& = \frac{1 - ( 3 + 2 M ) \varepsilon }{2 - 6 \varepsilon } 
\\
& \equalsdef f ( M , \varepsilon ).
\end{split}
\end{equation}
Note that for fixed \( M \) we have that \( \lim_{ \varepsilon {\downarrow} 0} f ( M , \varepsilon ) = \frac{1}{2}\), and since \( M > 1 \) and \( \varepsilon < \varepsilon_0 \), then 
\begin{equation*}
f ( M , \varepsilon ) > \frac{1 }{ 2 M } .
\end{equation*}
Therefore if \( z \in \pre{ \omega }{ \setLR{ - 1 , 0 , 1 }} \) has infinitely many \( -1 \), then letting \( s = z \restriction n \) with \( z ( n ) = - 1 \), it follows that \( a_z = b_z \in K_{ s \conc \seq{ -1 } } \), so~\eqref{eq:th:spongy-a} implies that 
\begin{equation}\label{eq:D-(az)}
\density^-_K ( a_z ) < \frac{1}{2M} 
\end{equation}
and since \( \varepsilon < \varepsilon_0 \), then~\eqref{eq:th:spongy-b} implies that 
\begin{equation}\label{eq:D+(az)}
\density^+_K ( a_z ) \geq f ( M , \varepsilon ) .
\end{equation}
Thus \( \oscillation_K ( a_z ) > 0 \).
A similar argument applies to the case when \( z \) has infinitely many \( 1 \).

Suppose now \( z \in F \), and let \( s \) be any large enough initial segment of \( z \) so that \( z = s \conc 0^{ ( \omega )} \).
Then \( a_z \) and \( b_z \) are the endpoints of the closed interval \( \bigcap_{n} [ a_{ s \conc 0^{( n ) }} ; b_{ s \conc 0^{( n ) }} ] \).
We only show that \( \oscillation_K ( b_z ) > 0 \), the argument for \( \oscillation_K ( a_z ) > 0 \) being similar.
Since \( ( b_z - r ; b_z ) \subseteq K \) for sufficiently small \( r \), it is enough to prove that \( \density_K^+ ( b_z^+ ) > \density_K^- ( b_z^+ ) \).
For ease of notation, set 
\[
g ( x ) = \frac{ \lambda ( K \cap ( b_z ; x ) ) }{ \card{ b_z - x } }, \quad \text{for }x > b_z .
\]
We will show (see~\eqref{eq:th:spongy-f} below) that for any \( s \) as above, the numbers \( g ( a_{ s \conc \seq{ 1 } } ) \) and \( g ( b_s ) \) are sufficiently far apart so that \( \density_K^+ ( b_z^+ ) > \density_K^- ( b_z^+ ) \) holds.

Recall that \( b_z = \inf_n b_{ s \conc 0^{( n )} } = \inf_n a_{ s \conc 0^{( n )}\conc \seq{1} } \) and 
\[ 
b_{ s \conc 0^{( n + 1 )} } < a_{ s \conc 0^{( n )} \conc \seq{1 } } < b_{ s \conc 0^{( n )} \conc \seq{1 } } = b_{ s \conc 0^{( n )} } 
\]
as summarized by Figure~\ref{fig:spongy2}.
\begin{figure}
\centering
\begin{tikzpicture}[scale=1,text height=0ex,text depth=0ex,pin distance=0.5cm]
\draw[very thin] (-0.5,0)--(-0.25,0);
\draw[very thick] (-0.25,0)--(0,0);
\draw[very thin] (0,0)--(1,0);
\draw[very thick] (1,0)--(1.5,0);
\draw[very thin] (1.5,0)--(3.5,0);
\draw[very thick] (3.5,0)--(4.5,0) ;
\draw[very thin] (4.5,0)--(8.5,0) ;
\draw[very thick] (8.5,0)--(10.5,0) ;
 \node at (-0.25,0) [circle,fill=black, inner sep = 1pt] {};
 \node at (0,0) [circle,fill=black, inner sep = 1pt] {};
 \node at (1,0) [circle,fill=black, inner sep = 1pt,pin=90:$a_{s \conc \seq{001}}$] {};
 \node at (1.5,0) [circle,fill=black, inner sep = 1pt,pin=270:$b_{s \conc \seq{001}}$] {};
 \node at (3.5,0) [circle,fill=black, inner sep = 1pt,pin=90:$a_{s \conc \seq{01}}$] {};
 \node at (4.5,0) [circle,fill=black, inner sep = 1pt,pin=270:$b_{s \conc \seq{01}}$] {};
 \node at (8.5,0) [circle,fill=black, inner sep = 1pt,pin=90:$a_{s \conc \seq{1}}$] {};
 \node at (10.5,0) [circle,fill=black, inner sep = 1pt,pin=270:$b_{s \conc \seq{1}}$] {};
\node at (10.5,0) [right] {\( b_s \)};
\node at (-0.5,0) [left] {\( b_z \)};
\end{tikzpicture}
\caption{}\label{fig:spongy2}
\end{figure}
\begin{subequations}\label{eq:subequations}
\begin{gather}
 a_{ s \conc 0^{( n ) } \conc \seq{ 1 }} = b_{ s \conc 0^{( n ) } } - \varepsilon ^{ \lh ( s ) + n + 1 } \label{eq:as1>bs}
 \\
 b_{ s \conc 0^{ ( n + 1 ) }} = b_{ s \conc 0^{ ( n ) }} - ( 1 + M ) \varepsilon^{\lh ( s ) + n + 1 } \label{eq:as1>bs2}
 \\
 b_{z} = b_{ s} - \frac{( 1 + M ) \varepsilon^{\lh ( s ) + 1 } }{ 1 - \varepsilon } . \label{eq:bz>bs}
\end{gather} 
\end{subequations}
Since \( K \cap \ocinterval{b_z}{b_s} \subseteq \bigcup_{n \in \omega } [ a_{ s \conc 0^{( n ) } \conc \seq{ 1 } } ; b_{ s \conc 0^{( n ) } \conc \seq{ 1 }} ] = \bigcup_{ n \in \omega } K_{ s \conc 0^{( n ) } \conc \seq{ 1 } } \), then
\begin{align*}
g ( b_s ) & = \frac{ 1 - \varepsilon }{( 1 + M ) \varepsilon^{\lh ( s ) + 1 }} \sum_{n = 0}^\infty \lambda ( K \cap K_{ s \conc 0^{( n ) } \conc \seq{ 1 } } ) &\text{by~\eqref{eq:bz>bs}}
\\
& = \frac{ 1 - \varepsilon } {( 1 + M ) \varepsilon^{\lh ( s ) + 1 } }\sum_{n = 0}^\infty \bigl [ \card{ K_{ s \conc 0^{( n ) } \conc \seq{ 1 } } } - \frac{ 2 M \varepsilon^{ \lh ( s ) + n + 2 }}{ 1 - 3 \varepsilon } \bigr ] &\text{by~\eqref{eq:measurespongy}}
\\
& = \frac{ 1 - \varepsilon } {( 1 + M ) \varepsilon^{\lh ( s ) + 1 } }\sum_{n = 0}^\infty \bigl [ \varepsilon^{\lh ( s ) + n + 1 } - \frac{ 2 M \varepsilon^{ \lh ( s ) + n + 2 }}{ 1 - 3 \varepsilon } \bigr ] &\text{by~\eqref{eq:connectedcomponent2}}
\\
& = \frac{ 1 - \varepsilon } {( 1 + M ) }\sum_{n = 0}^\infty \bigl [ 1- \frac{ 2 M \varepsilon }{ 1 - 3 \varepsilon } \bigr ] \varepsilon^{ n }
\\
& = \frac{ 1 - \varepsilon ( 3 + 2 M ) }{ ( 1 + M ) ( 1 - 3 \varepsilon ) }.
\end{align*}
For fixed \( M \), the map \( \varepsilon \mapsto \frac{ 1 - \varepsilon ( 3 + 2 M ) }{ ( 1 + M ) ( 1 - 3 \varepsilon ) } \) is decreasing, and since \( \varepsilon < \varepsilon_0 \),
\begin{equation*}
g ( b_s ) > \frac{ 1 }{ M ( 1 + M ) } .
\end{equation*}
By the equations~\eqref{eq:subequations}, 
\begin{equation}\label{eq:th:spongy-d}
 \frac{ \card{ b_z - b_{s \conc \seq{ 0 } } } }{ \card{ b_z - a _{ s \conc \seq{ 1 }} } } = \frac{ ( 1 + M ) \varepsilon / ( 1 - \varepsilon ) }{ M + ( 1 + M ) \varepsilon / ( 1 - \varepsilon ) } .
\end{equation}
 As \( K \cap ( b_{ s \conc \seq{ 0 } } ; a_{ s \conc \seq{ 1 } } ) = \emptyset \), then
\begin{align*}
g ( a_{ s \conc \seq{ 1 }} ) & = \frac{ \lambda ( K \cap \ocinterval{ b_z }{ b_{s \conc \seq{ 0 } } } ) }{ \card{ b_z - a_{ s \conc \seq{ 1 } } } } 
\\
& < \frac{ \card{ b_z - b_{s \conc \seq{ 0 } } } }{ ( 1 + M ) \card{ b_z - a _{ s \conc \seq{ 1 }} } } && 
\\
& = \frac{ \varepsilon / ( 1 - \varepsilon ) }{ M + ( 1 + M ) \varepsilon / ( 1 - \varepsilon ) } && \text{by~\eqref{eq:th:spongy-d}} 
\\
&= \frac{ \varepsilon }{ M + \varepsilon } .
\end{align*}
For fixed \( M \) the map \( \varepsilon \mapsto \frac{ \varepsilon }{ M + \varepsilon } \) is increasing, and since \( \varepsilon < \varepsilon_0 \), then 
\begin{equation}\label{eq:th:spongy-f}
g ( a_{ s \conc \seq{ 1 }} ) < \frac{M - 1}{ 2 M^3 + 3M^2 - 2M - 1} < \frac{1}{M ( M + 1 )} < g ( b_s ) . 
\end{equation}
Therefore \( \density_K^- ( b_z ^+ ) < \density_K^+ ( b_z ^+ ) \) as required.
\end{proof}

\begin{remarks}\label{rmk:spongysubset}
\begin{enumerate-(a)}
\item\label{rmk:spongysubset-a}
Since \( 0 = a_{ -1 ^{ ( \omega ) } } = a_\emptyset \) and \( 1 = b_{ 1^{( \omega )} } = b_\emptyset \), equations~\eqref{eq:D-(az)} and~\eqref{eq:D+(az)} imply that \( \oscillation_S ( 0 ) , \oscillation_S ( 1 ) > 0 \), where \( S = K \) or \( S = \Int K \).
\item\label{rmk:spongysubset-b}
Choosing suitable \( M \) and \( \varepsilon \), a spongy set \( S \subseteq \R \) is obtained so that \( S \times \R \) is spongy in \( \R^2 \).
This result will appear elsewhere.
\end{enumerate-(a)}
\end{remarks}

\begin{corollary}
For every \( m \in ( 0 ; 1 ) \) there is a spongy set \( X \subset [ 0 ; 1 ] \) such that \( \inf X = 0 \), \( \sup X = 1 \), and \( \lambda ( X ) = m \).
Moreover \( X \) can be taken to be open or closed.
Furthermore we can arrange the construction so that \( 0 < \oscillation_X ( 0 ) , \oscillation_X ( 1 ) \) or \( \oscillation_X ( 0 ) = \oscillation_X ( 1 ) = 0 \).
\end{corollary}

\begin{proof}
Let \( S \) be an open, spongy set as in Theorem~\ref{thm:spongy} and let \( 0 < M = \lambda ( S ) < 1 \).
By Remark~\ref{rmk:spongysubset}\ref{rmk:spongysubset-a}, \( 0 < \oscillation_S ( 0 ) , \oscillation_S ( 1 ) \).
We first prove the existence of an open spongy set \( X \) of measure \( m \) and such that \( \oscillation_X ( i ) = \oscillation_S ( i ) \) for \( i = 0 , 1 \).
The affine map \( [ 0 ; 1 ] \to [ a ; b ] \), \( x \mapsto a + ( b - a ) x \), preserves densities, thus  the image of \( S \) under this map, call it \( S_{a , b} \), is a spongy subset of \( [ a ; b ] \) such that \( \oscillation_{ S_{a,b} } ( a ) = \oscillation_S ( 0 ) \) and \( \oscillation_{ S_{a,b} } ( b ) = \oscillation_S ( 1 ) \), and \( \lambda ( S_{a,b} ) = ( b - a ) M \).
For each \( 0 < \alpha < 1 / 2 \) the sets \( X^- ( \alpha ) = S_{ 0 , \alpha } \cup S_{ 1 - \alpha , 1 } \) and \( X^+ ( \alpha ) = X^- ( \alpha ) \cup \left ( \alpha ; 1 - \alpha \right ) \) are open, spongy, and have measure \( 2 M \alpha \) and \( 1 - 2 \alpha ( 1 - M ) \), respectively, and therefore for each \( m \in ( 0 ; 1 ) \) there is an open \( X \) as in the statement. 
The requirement ``\( X \) closed'' can be fulfilled by starting with a closed \( S \) and using \( [ \alpha ; 1 - \alpha ] \) in the definition of \( X^+ ( \alpha ) \).

Let us now show how to modify the construction in order to attain \( \oscillation_X ( 0 ) = \oscillation_X ( 1 ) = 0 \).
Choose \( \varepsilon _n {\downarrow} 0 \) be such that \( \varepsilon _0 \leq 1/2 \) and let \( X_n^0 \subseteq ( \varepsilon _{ 2 n + 1} ; \varepsilon _{2 n} ) \) and \( X_n ^1 \subseteq ( 1 - \varepsilon _{ 2n } ; 1 - \varepsilon _{ 2n + 1 } ) \) be spongy sets such that \( \lambda ( X^i_n ) / ( \varepsilon _{ 2 n } - \varepsilon _{ 2 n + 1 } ) \leq 2^{ - n } \), for \( i = 0 , 1 \).
Then \( X = \bigcup_{n \in \omega } X^0_n \cup X^1_n \) is spongy and \( \oscillation_X ( 0 ) = \oscillation_X ( 1 ) = 0 \).
\end{proof}

\subsection{Solid sets}
Balls in \( \R^n \) are typical examples of solid sets.
A ball in \( \R \) of center \( x \) and radius \( r \) is just the interval \( ( x - r ; x + r ) \) and the points of its frontier \( \set{x - r , x + r } \) have density \( 1 / 2 \).
The same is true for \( B_2 = \setof{ \mathbf{y} \in \R^{ n + 1 } }{ \absval{ \mathbf{y} - \mathbf{x} }_2 < r } \), the ball in \( \R^{n + 1} \) with center \( \mathbf{x} \) and radius \( r \): its frontier is the \( n \)-dimensional sphere \( S_2 = \setof{ \mathbf{y} }{ \absval{ \mathbf{y} - \mathbf{x} }_2 = r } \) which, being a differentiable manifold, can be smoothly approximated with a hyperplane at every point, and therefore \( \density_{B_2} ( \mathbf{y} ) = 1 / 2 \) for all \( \mathbf{y} \in S_2 \).
The index \( 2 \) refers to the fact that we used the \( \ell_2 \)-norm, but a similar argument works for the \( \ell_p \)-norm, with \( 1 < p < + \infty \).
When \( p \in \set{ 1 , + \infty } \) the ball \( B_p \) is still solid, but \( S_p \) is no longer smooth, and we get the weaker result that \( \density_{B_p} ( \mathbf{y} ) = 1 / 2 \) for comeager many (in fact: all but finitely many) \( \mathbf{y} \in S_p \).
 
\begin{definition}\label{def:quasiEuclidean}
A Polish measure space \( ( X , d , \mu ) \) is \markdef{quasi-Euclidean} if it is locally compact, connected, \( \mu \) is continuous, fully supported, locally finite and satisfies the DPP. 
\end{definition}

Thus \( \R^n \) with the \( \ell_p \)-metric (\( 1 \leq p \leq \infty \)) and the \( n \)-th dimensional Lebesgue measure is quasi-Euclidean.
Note that all \( \ell_p \) metrics on \( \R^n \) are equivalent.

\begin{theorem}\label{thm:solid}
Suppose \( ( X , d , \mu ) \) is quasi-Euclidean and that \( A \subseteq X \) is nontrivial and solid.
Suppose \( d' \) is an equivalent metric such that every \( \Ball' ( x ; r ) = \setofLR{ z \in X}{ d' ( z , x ) < r } \) is solid and there is a \( \rho \in ( 0 ; 1 ) \) such that 
\[
 \FORALL{ x , y \in X } \FORALL{ r > 0 } \left [ d' ( y , x ) = r \implies \density_{ \Ball' ( x ; r ) } ( y ) = \rho \right ] .
\]
Then 
\begin{enumerate-(a)}
\item\label{thm:solid-a}
\( \Fr_\mu ( A ) \) is closed and nonempty,
\item\label{thm:solid-b}
\( \setofLR{ x \in X }{ \density_A ( x ) = \rho } \) is a dense subset of \( \Fr_\mu ( A ) \), and
\item\label{thm:solid-c}
\( \rho = 1 / 2 \).
\end{enumerate-(a)}
\end{theorem}

\begin{remark}
The density function \( \density \) in Theorem~\ref{thm:solid} refers to the metric \( d \), not to \( d' \).
\end{remark}

By Proposition~\ref{prop:solidsetBaireclassDensity} \( \setofLR{ x \in X }{ \density_A ( x ) = \rho } \) is \( \Gdelta \) for any \( \rho \), so

\begin{corollary}\label{cor:solid}
If \( X , d , d' , \mu , A \) are as in Theorem~\ref{thm:solid}, then \( \setofLR{ x \in X }{ \density_A ( x ) = 1 / 2 } \) is \( \Gdelta \) dense in \( \Fr_\mu ( A ) \), and \( \setofLR{ x \in X }{ \density_A ( x ) = \rho } \) is not dense in \( \Fr_\mu ( A ) \) for any \( \rho \in ( 0 ; 1 ) \setminus \set{ 1 / 2 } \).
\end{corollary}

\begin{corollary}\label{cor:nodualisticsetsinRn}
Work in \( \R^n \) with the \( \ell_p \) metric \( ( 1 \leq p \leq \infty ) \) and the Lebesgue measure.
If \( A \subseteq \R^n \) is nontrivial and solid, then \( \density_A ( \mathbf{x} ) = \frac{1}{2} \) for comeager many \( \mathbf{x} \in \Fr_\mu ( A ) \).

In particular, there are no nontrivial dualistic sets.
\end{corollary}

\begin{proof}[Proof of Theorem~\ref{thm:solid}]
\ref{thm:solid-a} follows from the fact that \( A \) is nontrivial and \( X \) is connected.
For the sake of notation let \( F = \Fr_\mu ( A ) \).

The crux of the matter is the proof of~\ref{thm:solid-b}.
Towards a contradiction, suppose that \( \density_A ( x ) \neq \rho \) for all \( x \in U \cap F \), where \( U \) is open in \( X \) and \( U \cap F \neq \emptyset \).
Then the sets
\begin{align*}
F^ + & = \setofLR{ x \in F \cap U }{ \density_A ( x ) > \rho }
\\
F^- & = \setofLR{ x \in F \cap U }{ \density_A ( x ) < \rho }
\end{align*}
partition \( F \cap U \). 
Since 
\[
F^ + = \bigcup_{m , k} F^ + _{m, k } \qquad\text{and}\qquad F^- = \bigcup_{m, k} F^-_{m, k} 
\]
where
\begin{align*}
F^ + _{m, k} & = \setofLR{ x \in F \cap U }{ \FORALL{n > m} \Bigl [ \frac{ \mu ( \Ball ( x ; 1 / n ) \cap A ) }{ \mu ( \Ball ( x ; 1 / n ) )} \geq \rho + 2^{-k} \Bigr ] }
\\
F^-_{m,k} & = \setofLR{ x \in F \cap U }{ \FORALL{n > m} \Bigl [ \frac{ \mu ( \Ball ( x ; 1 / n ) \cap A ) }{ \mu ( \Ball ( x ; 1 / n ) )} \leq \rho - 2^{-k} \Bigr ] } ,
\end{align*}
by the continuity property of the measure in the Definition~\ref{def:quasiEuclidean}, the sets \( F^ + _{m, k} \) and \( F^-_{m, k} \) are closed in \( F \cap U \), and hence both \( F^ + \) and \( F^- \) are \( \bSigma^{0}_{2} \), and therefore are \( \bDelta^{0}_{2} \).
If we show that both \( F^ + \) and \( F^- \) are dense in \( F \cap U \), a contradiction follows applying the Baire category theorem.

Fix \( x \in F \cap U \) towards proving that \( x \in \Cl ( F^+ ) \cap \Cl ( F^- ) \).
The solidity of \( A \) together with Lemma~\ref{lem:useless} yield that \( \Int ( F ) = \emptyset \).

\begin{claim}
If \( x \in \Fr ( \Int_\mu A ) \) then \( x \in \Cl ( F^+ ) \). 
\end{claim}

\begin{proof}
Towards a contradiction, choose \( \delta \) such that \( \Ball' ( x ; \delta ) \) is compact and such that 
\begin{equation}\label{eq:antiSzenes-absurd2}
 \Ball ' ( x ; \delta ) \cap F^ + = \emptyset .
\end{equation}
Pick \( y \in \Int_\mu ( A ) \cap \Ball ' ( x ; \delta / 2 ) \).
By compactness there is \( w \in \Ball' ( x ; \delta ) \setminus \Int_\mu ( A ) \) such that
\[ 
0 < r = d ' ( y , X \setminus \Int_\mu ( A ) ) = d ' ( y , w ) < \delta / 2 .
\]
Since \( d' ( x , w ) \leq d' ( x , y ) + d ' ( y , w ) < \delta /2 + r < \delta \), then \( w \in \Ball ' ( x ; \delta ) \), so \( w \in \Int_\mu ( A^ \complement ) \cup F^- \) by~\eqref{eq:antiSzenes-absurd2}, and therefore \( \density_A ( w ) < \rho \).
Moreover, \( z \in \Ball' ( y ; r ) \implies z \in \Int_\mu ( A ) \); so \( \Ball' ( y ; r ) \subseteq \Int_\mu ( A ) \).
By assumption \( \density_{ \Ball' ( y ; r ) } ( w ) = \rho \) hence \( \density^-_{\Int_\mu ( A )} ( w ) \geq \rho . \)
Since \( \Int_\mu ( A ) \subseteq \Phi( A ) =_\mu A \), then \( \density^-_{\Int_\mu ( A )} ( w ) \leq \density_A ( w ) \), contradicting the preceding calculations.
\end{proof}

Similarly, if \( x \in \Fr ( \Int_\mu A^\complement ) \) then \( x \in \Cl ( F^- ) \).
Therefore if \( x \in \Fr ( \Int_\mu A ) \cap \Fr ( \Int_\mu A^\complement ) \) then \( x \in \Cl ( F^+ ) \cap \Cl ( F^- ) \), as required. 
 
\begin{claim}
If \( x \notin \Fr \Int_\mu ( A^\complement ) \) then \( x \in \Cl ( F^- ) \).
\end{claim}

\begin{proof}
Fix \( \gamma \) sufficiently small such that \( \Ball ( x ; \gamma ) \cap \Int_\mu ( A^ \complement ) = \emptyset \).
Since \( x \notin \Int_\mu ( A ) \), then \( \mu ( \Ball ( x ; \gamma ) \cap A ) < \mu ( \Ball ( x ; \gamma ) ) \) so by DPP there is \( y \in \Ball ( x ; \gamma ) \) such that \( \density_A ( y ) = 0 \), and therefore \( y \in F^- \).
\end{proof}

Similarly if \( x \notin \Fr \Int_\mu ( A ) \) then \( x \in \Cl ( F^+ ) \).

Therefore we have shown that if \( x \in F \cap U \) then \( x \in \Cl ( F^+ ) \cap \Cl ( F^- ) \).
This concludes the proof of part~\ref{thm:solid-b} of the theorem.

Now we argue for part~\ref{thm:solid-c}.
Fix \( y \in X \) and \( r > 0 \), and \( A = \Ball' ( y ; r ) ^ \complement \), so by part~\ref{thm:solid-b} there is \( x_0 \in X \) such that \( \density_A ( x_0 ) = \rho \). 
On the other hand \( \density_A ( x ) = 1 - \density_{A^\complement} ( x ) \), and \( \density_{A^\complement} ( x ) \in \setLR{ 0 , \rho , 1} \).
Thus \( \rho = 1 - \rho = 1 / 2 \).
\end{proof}

\printbibliography

\end{document}